\documentclass[10pt]{amsart}
   \usepackage{amsthm,amsmath,amssymb,amscd,graphicx,enumerate, stmaryrd,xspace,verbatim, epic, eepic,color,url}
\usepackage{eurosym}
\usepackage[backref=page]{hyperref}
\usepackage{pgf}
\usepackage{amsmath}
\usepackage{amsthm}
\usepackage{amsfonts}
\usepackage{mathrsfs}
\usepackage{amssymb}
\usepackage{amscd}
\usepackage[all]{xy}
\usepackage{tikz}
\usepackage{amscd}
\usepackage{graphicx}
\usepackage{setspace}
\usepackage{enumerate}
\usepackage{xcolor}
\usepackage[
    hscale=0.7,
    vscale=0.75,
    headheight=13pt,
    centering,
    ]{geometry}

\makeindex

\usepackage{color, colortbl}

    \newtheorem{Lem}{Lemma}[section]
    \newtheorem{Lem-Def}[Lem]{Lemma-Definition}
    \newtheorem{Prop}[Lem]{Proposition}
       
     \newtheorem*{thm*}{Theorem}
\newtheorem*{thmA}{Theorem A}
\newtheorem*{thmB}{Theorem B}

    \newtheorem{Thm}[Lem]{Theorem}

\theoremstyle{definition}

    \newtheorem{Def}[Lem]{Definition}
    
    \newtheorem{Exa}[Lem]{Example}

    \newtheorem{Rem}[Lem]{Remark}
    \newtheorem{Cons}[Lem]{Construction}
    \newtheorem{Not}[Lem]{Notation}
    \newtheorem{CorDef}[Lem]{Corollary-Definition}
    
\newcommand{\mb}{\mathbb}
\newcommand{\mc}{\mathcal}

\newcommand{\ult}{\underline{t}}
\newcommand{\G}{\Gamma}
\newcommand{\tit}{\textit}
\newcommand{\tbf}{\textbf}

\newcommand{\ora}[1]{\overrightarrow{#1}}

\newcommand{\E}{\mathcal E}

\newcommand{\F}{\mathcal F}

\newcommand{\I}{\mathcal I}
\newcommand{\M}{\mathcal M}
\newcommand{\N}{\mathcal N}

\renewcommand{\L}{\mathcal L}
\renewcommand{\O}{\mathcal O}
\newcommand{\C}{\mathcal C}

\newcommand{\D}{\mathcal D}
\newcommand{\bu}{\mathbf{u}}
\newcommand{\bv}{\mathbf{v}}
\newcommand{\be}{\mathbf{e}}
\newcommand{\nf}{\mathbf{f}}
\newcommand{\R}{\mathbb R}
\renewcommand{\P}{\mathcal{P}}
\newcommand{\col}{\colon}

\newcommand{\ra}{\rightarrow}
\newcommand{\ol}{\overline}
\newcommand{\supp}{\text{supp}}

\newcommand{\Q}{\mathbb{Q}}

\newcommand{\ul}{\underline}

\newcommand{\wt}{\widetilde}

\newcommand{\wh}{\widehat}

\newcommand{\J}{\mathcal{J}}

\renewcommand{\:}{\colon}

\newcommand{\NR}{N_{\mathbb{R}}}

\DeclareMathOperator{\qs}{qs}

\DeclareMathOperator{\sing}{sing}
\DeclareMathOperator{\Spec}{Spec}

\DeclareMathOperator{\stab}{s}
\DeclareMathOperator{\semi}{ss}
\DeclareMathOperator{\cof}{cof}

\DeclareMathOperator{\sm}{sm}

\renewcommand{\l}{\ell}

\newcommand{\cone}{\textnormal{cone}}
\newcommand{\val}{\text{val}}

\newcommand{\trop}{\text{trop}}
\renewcommand{\div}{\textnormal{div}}
\newcommand{\Div}{\text{Div}}
\newcommand{\Pic}{\text{Pic}}
\newcommand{\Prin}{\text{Prin}}
\renewcommand{\Im}{\text{Im}}

\newcommand{\ord}{\textnormal{ord}}

\newcommand{\Conv}{\textnormal{Conv}}

\newcommand{\boundellipse}[3]
{(#1) ellipse (#2 and #3)
}

\title[Abel Maps for nodal curves via tropical geometry]{Abel Maps for nodal curves via tropical geometry}

\author{Alex Abreu, Sally Andria, and Marco Pacini}

\begin{document}

\maketitle

\begin{abstract}
We consider Abel maps for regular smoothing of nodal curves with values in the Esteves compactified Jacobian.
In general, these maps are just rational, and
an interesting question is to find an explicit resolution.
We translate this problem into an explicit combinatorial problem by means of tropical  and toric geometry. We show that the solution of the combinatorial problem gives rise to an explicit resolution of the Abel map. We are able to use this technique to construct and study all the Abel maps of degree one.
\end{abstract}

\bigskip



\tableofcontents

\section{Introduction}
   
\subsection{History and motivation}

This work is dedicated to the problem of the construction of an explicit resolution of the Abel map for a one-parameter family of an algebraic nodal curves, with smooth generic fiber and singular special fiber. Our main contribution is to translate this problem into an explicit combinatorial problem by means of tropical  and toric geometry. The solution of the combinatorial problem gives rise to an explicit resolution of the Abel map. We are able to use this technique to construct all the degree-$1$ Abel maps of a nodal curve. 

The theory of Abel maps for algebraic curves goes back to work by Abel in the nineteenth century. It was Riemann, in his seminal paper \cite{R}, that introduced Abel maps; we refer to \cite{K} for the history. 
The Abel map of a smooth curve $C$ is the map $\alpha_d\col S^d(C)\ra \mathcal J^d_C$  taking $Q_1+\dots+Q_d$ to the invertible sheaf $\mathcal O_C(Q_1+\dots+Q_d)$. Here $S^d(C)$ is the $d$-th symmetric product of $C$ and $\mathcal J^d_C$ is the degree-$d$ Jacobian of $C$.
The Abel map of a smooth curve encodes many important geometric properties of the curve. For instance, the complete linear systems and the Brill-Noether varieties 
are, respectively, the fibers and the images of the Abel maps.

A natural research direction is the construction of Abel maps for nodal curves. In fact, 
 a powerful tool  for studying smooth curves consists in taking degenerations to singular curves: it was through this technique that the Brill-Noether and Gieseker-Petri theorems were proved, see \cite{G} and \cite{GH}. It is natural to expect that a construction of the Abel map for singular curves could be very useful to study the degeneration process from smooth curves to singular ones. \par

The first problem to solve when one tries to construct an Abel map for a singular curve is the choice of a target. It is not enough to consider the generalized Jacobian of a curve (parametrizing invertible sheaves
), since, in general, it is not a proper space. Hence one has to fix a compactified Jacobian as a target of the Abel map. In this work we consider Esteves compactified Jacobian, parametrizing torsion-free rank-1 sheaves on the curve that are quasistable with respect to a fixed section and polarization (see \cite{EE01}). The other natural choice would be Caporaso compactified Jacobian, parametrizing balanced line bundles on curves (see \cite{C94}).  At any rate, our choice of target is not a restricting one, since any map towards Esteves compactified Jacobian induces a map towards Caporaso compactified Jacobian  (Caporaso space is the target of a morphism from Esteves space).  

  Several works have been dedicated to Abel maps of singular curves. The first construction appeared in \cite{AK} for irreducible curves. The reducible case is more complicated, principally due to the fact that, in general, the compactified Jacobian does not parametrize all the invertible sheaves on the curve. In this case it is usual to resort to a one-parameter deformation to smooth curves, and construct an Abel map as a limit of the Abel maps of the smooth fibers. This has been started in degree one in \cite{CE} and \cite{CCE}, in degree two in \cite{CEP}, for curves of compact type and any degree in \cite{CP}, and for curves with two components and any degree in \cite{AAJCMP}.   For the degree-$1$ map, see also the recent work \cite{NS}.
  The general problem is still open.

In our work we employ tropical and toric geometry. These techniques are not used in these previous works.
However, the use of tropical and toric geometry in  these types of problems already appeared in the recent works \cite{ID} and \cite{AP2}. A resolution of degree-$d$ Abel map with trivial polarization, for circular nodal curves and nodal curves with two components, is constructed in \cite{ID} by means of toric geometry. 
 Tropical and toric geometry are employed in \cite{AP2} to give a resolution of the universal Abel map, that is the rational map $\overline {\mathcal M}_{g,n}\ra \overline{\mathcal J}_{\mu,g}$ taking an $n$-tuple of points on a stable curve to the associated sheaf. Here $\overline{\mathcal J}_{\mu,g}$ is Esteves universal compactified Jacobian, constructed in \cite{Melo}  building on the work of Esteves \cite{EE01}. The philosophy of our work is very close to the one inspiring \cite{AP2}. The crucial idea is that, by looking at the analogous tropical problem, one obtains a ``tropically inspired" resolution of the Abel map for a one-parameter family of curves, attained by means of  toric geometry.   
 It is worth mentioning that if $\C\to B$ is a one parameter family we have only a rational map $\C^d\dashrightarrow \ol{\M}_{g,d}$. Hence, to use the universal setting to get a resolution of the Abel map $\C^d\dashrightarrow \overline{\mc J}^\sigma_\mu$, we should first resolve the map $\C^d\dashrightarrow \ol{\M}_{g,d}$, which would give  a larger resolution than needed. We will
not pursue this approach here.

\subsection{The results}

Let us explain the details of our results.  Let $C$ be a nodal curve defined over an algebraically closed field $k$. Let $\pi\col\mc C\ra B=\Spec(k[[t]])$ be a one-parameter family of curves, with smooth generic fiber and with $C$ as a special fiber. Assume that $\mc C$ is a regular surface. We call $\pi\col \mc C\ra B$ a  smoothing of $C$. Let $\sigma\col B\ra \C$ be a section of $\pi$ through its smooth locus. Let $\mc L$ and $\mu$ be respectively an invertible sheaf and a polarization of the same degree on $\C/B$. We want to study the Abel map 
\[
\alpha^d_{\mc L} \col\mc C^d=\mc C\times_B \dots\times_B \mc C\dashrightarrow \overline{\mc J}_\mu^\sigma
\]
taking a tuple $(Q_1,\dots,Q_d)$ of smooth points on a fiber $\C_b$ of $\pi$ to the sheaf $\mc L|_{\C_b}(d\sigma(b)-Q_1-\dots-Q_d)$.  Here $\overline{\mc J}_\mu^\sigma$ is Esteves compactified Jacobian, parametrizing torsion-free rank-$1$ sheaves of fixed degree on the fibers of $\C/B$ which are $(\sigma,\mu)$-quasistable (see \cite{EE01}). Of course, $\alpha^d_{\mc L}$ is only a rational map since not every sheaf is $(\sigma,\mu)$-quasistable. In general we need to perform a blowup of $\mc C^d$ to get a resolution of the Abel map. We would like a geometrically meaningful description of the blowup, that is, we would like to understand exactly what are the centers of a blowup giving rise to a resolution of the map. This is in general a very difficult problem. 

The local setting could help in a better understanding of the problem. In a way, it is enough to consider the most degenerate points of $\mc C^d$, which are the points of the form $\mc N=(N_1,\dots,N_d)$, where $N_i$ are nodes of the special fiber $C$ of the family. In fact, thanks to the fact that quasistability is an open property (see \cite[Proposition 34]{EE01}), a resolution of the Abel map at a point implies a resolution locally around this point.  The structure of the completion of the local ring of $\mc C^d$ at $\mc N$ is well-understood:
\begin{equation}
\label{eq:etalenodeint}
\widehat{\mc O}_{\mc C^d,\mc N}\cong \frac{k[[x_1,y_1,\dots,x_d,y_d,t]]}{(x_1y_1-t,\dots,x_dy_d-t)},
\end{equation}
where $x_i,y_i$ are \'etale local coordinates of $\mathcal C$ at $N_i$.
This explicit presentation allows the study of all possible blowups of $\C^d$ whose centers are divisors contained in $C^d$. Then one can try to write down a criterion testing which are the blowups giving rise to a resolution of the Abel map.   
This is the approach of the work  \cite{AAJCMP}, where the authors are able to give a criterion for a resolution of the Abel map locally around a point $\mc N\in \mc C^d$. This criterion involves certain numerical invariants 
and translates the condition of the admissibility (ensuring that the relevant sheaves are torsion-free) and the quasistability. As we shall see below, in the same spirit of the work \cite{AAJCMP}, we give another criterion for the local resolution of the Abel map, which involves the tropical Abel map.

Thus we come to the tropical side of the story.
Let $\Gamma$ be a graph with an orientation and a fixed vertex $v_0$. Let $X$ be the tropical curve with $\Gamma$ as an underlying graph and all edges of unitary length. Let $p_0\in X$ be the point corresponding to $v_0$. Although all the setting and results could be extended to all tropical curves, we only develop partially the general theory. In fact, tropical curves with unitary length function are already enough for our main purpose which, geometrically, consists in the analysis of a smoothing of a nodal curve with smooth total space.

The tropical curve $X$ comes equipped with a tropical Jacobian $J(X)$, parametrizing equivalence classes of divisors on $X$, and a tropical Abel map $\alpha^{\trop}_{d,\mc D^\dagger} \col X^d\ra J(X)$, taking a tuple $(p_1,\dots,p_d)$ to the class of the divisor $\mc D^\dagger-\sum_{1\le i\le d} p_i$, where $\mc D^\dagger$ is a fixed divisor on $X$. It seems that the tropical Abel map is a honest map and nothing is ``wrong" with it. But, when we look closely at the hidden polyhedral structures, we realize that, in general, this map does not respect these finer structures.

In fact, the tropical Jacobian $J(X)$ has a distinguished structure of polyhedral complex introduced in \cite{AAMPJac} by means of $(p_0,\mu)$-quasistability (the tropical analogue of Esteves quasistability), where $p_0\in X$ is a fixed point and $\mu$ is a polarization on $X$. Loosely speaking, the cells $\P_{(\E,D)}$ of the polyhedral decomposition of $J(X)$ parametrize $(p_0,\mu)$-quasistable divisors with a fixed combinatorial type, which amounts to choosing a divisor $D$ on the graph $\Gamma$ (retaining the part of the divisor on the tropical curve supported on the vertices of $\Gamma$) and a subset of edges $\E$ of $\Gamma$ (retaining the part of the divisor on the tropical curve supported in the interiors of the edges). 

On the other hand, since $X$ is a union of segments, it is clear that the product $X^d$ is a union of hypercubes, so it has a natural structure of polyhedral complex. In general the tropical Abel map does not take a hypercube of $X^d$ to a cell $\P_{(\E,D)}$, so it fails to be a morphism of polyhedral complexes. Nevertheless, it is expected that a refinement of the hypercubes of $X^d$ (e.g., into simplexes) should give a finer structure to $X^d$ for which the Abel map becomes a map of polyhedral complexes. 

One realizes that the condition that a simplex of $X^d$ is sent to a cell of the tropical Jacobian $J(X)$ has nice consequences. This property ensures that certain numerical inequalities hold. These numerical inequalities are very important for our main results and are proved in Theorem \ref{Thm:4.2_versão_tropical}. They involve certain combinatorial numbers 
 (see Definitions \ref{Def:a_tropical} and \ref{Def:b_tropical}) which are
  closely related to the numerical invariants introduced in \cite{AAJCMP}.

Next, we come back to geometry. Recall that $\pi\col \mc C\ra B$ is a smoothing of a nodal curve $C$. We let $\Gamma$ be the dual graph of $C$ and $X$ the tropical curve with $\Gamma$ as underlying graph and unitary lengths. We resort to toric geometry,  arguing locally around a point $\mc N=(N_1,\dots,N_d)$ of $\mc C^d$. Combinatorially, this corresponds to a hypercube $\mc H\subset X^d$, since the local equations of $\mc C^d$ at $\mc N$ are the equations of the toric variety $\wh{U}_{\mc H}$ associated to the cone over $\mc H$. The key condition is the existence of a   triangulation of $\mc H$ in simplexes each one of which is sent to a cell of the tropical Jacobian via the tropical Abel map. We call this condition \emph{compatibility with the tropical Abel map}, which is a necessary condition for the tropical Abel map from the refined hypercube $\mc H$ to $J(X)$ to be a morphism of polyhedral complexes (actually, in our situation, the two conditions are equivalent). Using the theory of toric varieties, the refinement of $\mc H$ corresponds to a toric blowup $T\ra \wh{U}_{\mc H}$. The following result is Theorem \ref{thm:tropgeoAbel}.

\begin{thmA}
Let $\C\ra B$ be a smoothing of a curve $C$ with smooth components. Let $\beta\col T\ra \wh{U}_{\mc H}$ be the toric blowup associated to a unimodular triangulation of $\mc H$ compatible with the tropical Abel map.
Then the composed map
\[
T\stackrel{\beta}{\ra}\wh{U}_{\mc H}\cong \Spec(\widehat{\O}_{\C^d,\N})\stackrel{\alpha_\L^d}{\dashrightarrow} \overline{\mathcal{J}}_\mu^\sigma
\]
is a morphism, i.e., it is defined everywhere.
\end{thmA}

In Example \ref{exa:two} we illustrate how one can use the result to get an explicit resolution of degree-$3$ Abel map locally at a point of  $\mc C^3$. 

We also apply our results to study Abel maps of degree $1$ of the family $\pi\col \C\ra B$ for every given polarization on $\mc C/B$.  In the case of the degree $1$ Abel map, the hypercubes in $X$ are simply the edges of the tropical curve, which are already unitary. This allows us to easily check the combinatorial condition of the previous result.
The following result is Theorems \ref{thm:Abel1} and \ref{thm:injective}.

\begin{thmB}
Let $\C\ra B$ be a smoothing of a curve $C$.
The degree-$1$ Abel map $\alpha^1_{\L}\col\C\ra \overline{\J}_{\mu}^\sigma$ is defined everywhere for every polarization $\mu$ and invertible sheaf $\mc L$. Moreover, if the dual graph of $C$ is biconnected, then $\alpha_\L^1\col \C\ra \ol{\J}_\mu^\sigma$ is injective and hence, given a smooth point $P$ of $C$, we have that any point of $\ol{\J}_{C,\mu}^P$ parametrizing an invertible sheaf is contained in a subscheme of $\ol{\J}_{C,\mu}^P$ homeomorphic to $C$.
\end{thmB}

As we already mentioned, the previous result gives rise to an Abel map of degree $1$ with Caporaso compactified Jacobian as a target, recovering in this way the result of  \cite{CE} (where Caporaso and Esteves constructed the Abel map of degree $1$ for the canonical polarization and the trivial bundle).
In a forthcoming paper we will apply our main result to construct Abel maps of degree 2 for nodal curves, solving one of the biggest unanswered question of \cite{CEP}.


\section{Background in algebraic geometry}\label{Chap:Abel maps for nodal curves}

\subsection{Jacobians of nodal curves}

Let $k$ be an algebraically closed ﬁeld.
A \emph{curve} is a reduced, connected, projective scheme of pure dimension $1$ over $k$. We will always assume our curves to be nodal, meaning that the singularities are nodes, that is, analytically like the origin in the union of the coordinate axes of the plane $\mathbb A^2_k$. A curve may have several irreducible components, which will be simply called \emph{components}. 


Let $C$ be a curve. A \tit{subcurve} of $C$ is a reduced union of components of $C$. A subcurve is not necessarily connected. 
If $Y$ is a subcurve of $C$, we define the \tit{complement of $Y$} as $Y^c:=\overline{C\setminus Y}$.

A coherent sheaf $I$ on $C$ has \emph{rank-$1$} if it has generic rank $1$ at each component of $C$, and it is \tit{invertible} if it is locally free of rank $1$.
Let $Y$ be a subcurve of $C$ and let $I$ be a torsion-free rank-1 sheaf. We denote by $I_{Y}$ the quotient of the restriction $I|_{Y}$ by its torsion subsheaf, that is, 
\begin{equation}\label{Eq:I_Y}
I_{Y}:=I|_Y/\mc{T}(I|_{Y}).
\end{equation}
The \textit{degree} of a torsion-free rank-1 sheaf $I$ on a curve $C$ is defined as  
 \[
\deg(I):=\chi(I)-\chi(\mathcal{O}_C),
 \]
  where $\chi(\cdot)$ is the Euler characteristic.
 
 Assume that the curve $C$ has  components $C_1,\dots,C_h$. A \tit{polarization on $C$} is a tuple $\mu=(\mu_1,\dots,\mu_h)$ of rational numbers summing up to an integer number, called the \emph{degree} of $\mu$. 
Given a proper subcurve  $Y$  of $C,$ we set 
 \[
 \mu(Y)=\sum_{C_i\subset Y}\mu_i.
 \]
We also set $\Sigma_{Y}:=Y\cap Y^{c}$ and $\delta_Y=\# \Sigma_Y$. 
We define
 \[
 \beta_I(Y):=\deg(I_Y)-\mu(Y)+\frac{\delta_Y}{2}.
 \]

\begin{Def}\label{stab} Let $C$ be a curve. Let $\mu$ be a polarization of degree $d$ on $C$ and $I$ be a torsion-free rank-$1$ sheaf of degree $d$ on $C$.  We say that $I$ is  \emph{stable} if $\beta_I(Y)>0$ for every proper subcurve $Y$ of $C$.  We say that $I$ is \emph{$\mu$-semistable} if $\beta_I(Y)\ge0$
for every subcurve $Y$ of $C$. Given a smooth point $P$ of $C$, we say that $I$ is $(P,\mu)$-quasistable if it is semistable and $\beta_I(Y)>0$ for every proper subcurve $Y$ of $C$ containing $P$.
\end{Def}

 
 
A \tit{family of curves} is a flat, projective morphism of schemes $\pi\col\mc{C}\ra B,$ whose fibers are curves. We denote by $\C_b:=\pi^{-1}(b)$ the fiber over $b\in B$. The family $\pi\col\C\ra B$ is called \tit{local of dimension 1} if $B=\Spec(k[[t]])$ and \tit{regular} if $\C$ is regular. 
A \emph{smoothing} of a curve $C$ is a regular family of curves which is local of dimension $1$, with smooth generic fiber and special fiber isomorphic to $C$.

Let $\pi\col\C\ra B$ be a family of curves. 
A \emph{sheaf on $\mc{C}/B$} is a $B$-flat coherent sheaf $\I$ on $\mc{C}.$ 
We say that a sheaf $\I$ on $\C/B$ is  \tit{torsion-free} (resp. \tit{rank-1}, resp. \tit{simple})  if, for each closed point $b\in B$, the restriction $\I|_{\C_b}$ of $\I$ to each fiber $\C_b$ of $\C/B$ is torsion-free (resp. rank-1, resp. simple).

A \emph{polarization} $\mu$ on $\C$ is the datum of a polarization $\mu_b$ on each  fiber $\C_b$ of $\pi$ which is compatible with specializations. This means that if $\mathcal Y$ is a subscheme of $\C$, flat over $B$, whose fiber $\mathcal Y_b$ over $b\in B$ is a subcurve of $\mathcal C_b$, $\forall b\in B$, then $\mu_b(\mathcal Y_b)$ is independent of $b$.
The degree of the polarization $\mu$ is the degree of its restriction to any fiber of the family.

Let $\I$ be a sheaf on $\C/B$. Let $\sigma\col T\to\C$ be a section of $\pi$ through its smooth locus. We say that a sheaf $\I$ over $\C/B$ is $\mu$-stable (respectively $\mu$-semistable, respectively  \emph{$(\sigma,\mu)$-quasistable}) if, for any closed point $b\in B$, the restriction of $\I$ to the fiber $\C_b$ is a torsion-free, rank-$1$ and $\mu_b$-stable (respectively $\mu_b$-semistable, respectively $(\sigma(b),\mu_b)$-quasistable) sheaf on $\C_b$.

The relative compactified Jacobian functor of the family $\C/B$ is the contravariant functor 
\[
\overline{\bf{J}}_{\C/B}\col(B\mbox{-schemes})^{\circ}\ra(\mbox{sets})
\]
that associates to each $B$-scheme $T$ the set of simple torsion-free rank-1 sheaves on $\C_{T}/T$ (here $\C_T=\mc C\times_B T$), modulo the  following equivalence relation: 
for $\I_1,\I_2$ on $\C_{T}/T$,
\begin{equation}\label{eq:equivalence}
\I_1\sim\I_2\ \mbox{ if }\  \exists\ \M\mbox{ invertible sheaf on $T$ such that}\ \I_1\cong \I_2\otimes p_2^{*}\M,
\end{equation}
where $p_2\col \C_T\ra T$ is the projection onto $T.$


The functor $\overline{\mathbf{J}}_{\C/B}$ is representable by an algebraic space $\overline{\mc J}_{\mathcal{C}/B}$. Let $\overline{\mc J}_{\mathcal C/B}^d$ be the locus of $\overline{\mc J}_{\mathcal{C}/B}$ parametrizing degree-$d$ sheaves.
In \cite[Theorem B]{EE01} Esteves showed, under certain hypothesis 
(including the case of a smoothing), that $\overline{\mc J}_{\mathcal{C}/B}$ is a scheme. 
 The space $\overline{\mc J}_{\mathcal{C}/B}$ is connected and satisfies the existence part of the valuative criterion but, in general, it is not separated and only locally of finite type. The failure of separatedness is due to the fact that the limit of a sequence of simple torsion-free rank-$1$ sheaves on $\mathcal{C}/B$ might be not unique. To deal with a manageable piece of it, in \cite{EE01} Esteves built proper subspaces of $\overline{\J}_{\mathcal{C}/B}$ (hence in particular separated). This can be done by resorting to polarizations, as we shall see below.

Let $\mu$ be a degree-$d$ polarization on $\C/B$.
Denote by $\overline{\mc J}^{\stab}_{\mu}$ (respectively $\overline{\mc J}^{\semi}_\mu$, respectively $\overline{\mc J}^{\sigma}_\mu$) the subspaces of $\overline{\mc J}^d_{\mathcal{C}/B}$ parametrizing  torsion-free rank-1 sheaves on $\mathcal{C}/B$ that are $\mu$-stable (respectively $\mu$-semistable, respectively $(\sigma,\mu)$-quasistable). 
 We have
\begin{center}
$\overline{\mc J}^{\stab}_\mu\subset  \overline{\mc J}^{\sigma}_\mu \subset 
\overline{\mc J}^{\semi}_\mu\subset
\overline{\mc J}_{\mathcal C/B}^d$.
\end{center}

In \cite{EE01} Esteves showed that all these spaces are open in $\overline{\mc J}_{\mathcal{C}/B}$ and their formation commutes with base change.
More properties of these spaces are described in the following result.

\begin{Thm}[Theorem A in \cite{EE01}]\label{theoA}
The algebraic space $\overline{\mc J}^{\semi}_\mu$ is of finite type over $B$. Moreover
\begin{enumerate}
\item[(1)] $\overline{\mc J}^{\semi}_\mu$ is universally closed over $B$;
\item[(2)]$\overline{\mc J}^{\stab}_\mu$ is a separated scheme over $B$;
\item[(3)] $\overline{\mc J}^{\sigma}_\mu$ is proper over $B$.
\end{enumerate}
\end{Thm}

\subsection{Abel maps for nodal curves}

Let $\pi\col\mathcal{C}\rightarrow B$ be a smoothing  of a nodal curve $C$. 
Let $\sigma\col B\ra\C$ be a section of $\pi$ through its smooth locus. Let $C_1,\ldots,C_h$ be the components of $C$. Denote the smooth locus of $\pi$ by $\dot{\C}$, and let $\dot{C_i}:=C_i\cap\dot{\C}.$ Set $\dot{\mathcal{C}}^d:=\dot{\mathcal{C}}\times_B \dot{\mathcal{C}}\times_B ...\times_B \dot{\mathcal{C}}$, the product of $d$ copies of $\dot{\C}$ over $B$. 
For each $d$-tuple $\underline{i}=\left(i_1,...,i_d\right)$, we define $\dot{C}_{\underline{i}}=\dot{C}_{i_1}\times ...\times \dot{C}_{i_d}$. Notice that $\dot{C}_{\underline{i}}$ is a divisor of $\dot\C^d$. Moreover, the special fiber of $\dot{\mathcal{C}}^d\rightarrow B$ can be written as:
\[\coprod_{1\leq i_1,...,i_d\leq h}\dot{C}_{\underline{i}}=\coprod_{1\leq i_1,...,i_d\leq h} \dot{C}_{i_1}\times ...\times \dot{C}_{i_d}.
\] 

Let $\mathcal{L}$ be an invertible sheaf on $\C/B$ of degree $k$. Consider the family 
\begin{equation}\label{eq:familyCd}
\dot\C^d\times_B\C\ra \dot\C^d
\end{equation}
whose fibers are the curves of the family $\pi\col \C\ra B$. There exists a \emph{degree-$d$ Abel map} 
\[
\dot{\C}^d\ra \ol{\mc J}_{\C/B}
\]
taking the $d$-tuple $\left(Q_1,...,Q_d\right)$ of points on a fiber $\C_b$ of $\dot\C^d\times_B\C\ra
\dot\C^d$ to the invertible sheaf
\begin{equation}\label{Eq:Ls}
\mathcal{L}|_{\mathcal{C}_b} \left(d\cdot\sigma\left(b\right)-Q_1-...-Q_d\right).
\end{equation}

Our goal is to extend (if possible) this Abel map to $\C^d$ and, to do that, we need to choose a proper target. We consider a polarization $\mu$ on $\C/B$ of degree $k$ and the compactified Jacobian $\overline{\mc J}_\mu^\sigma$ as a target.
However, the sheaf defined in equation \eqref{Eq:Ls} might fail to be $(\sigma,\mu)$-quasistable and thus it does not even induce a map from $\dot{\C}^d$ to $\overline{\mc J}_\mu^\sigma$. We need to modify this sheaf.\par

 Form the fiber diagram  
\[\xymatrix{
\dot{\mathcal{C}}^d\times_B \mathcal{C} \ar[d]^{\pi_d} \ar[r]^{\;\;\;\;\;\;f} & \mathcal{C} \ar[d]^{\pi} \\
\dot{\mathcal{C}}^d \ar[r] & B
}
\]
Let $\Delta_{i,d+1}$ be the preimage of the diagonal subscheme of $\C\times_B\C$ via the projection map \[
\dot{\mathcal{C}}^d\times_B \mathcal{C}\rightarrow \mathcal{C}\times_B \mathcal{C}\] onto the $i$-th and $(d+1)$-th factor.
By \cite[Thm 32 (4)]{EE01} and \cite[Lemma 4.1]{C94}, 
for each $\underline{i}=(i_1,\dots,i_d)$ there is a divisor of $\dot{\mathcal{C}}^d\times_B \mathcal{C}$ of type
\begin{equation}
\label{Eq:twZ}
Z_{\underline{i}}=\sum_{1\le j\le h} \ell_{\underline{i},j} \cdot \dot{C_{\underline{i}}} \times C_j,
\end{equation}
for integers $\ell_{\underline{i},j}$,
 such that the invertible sheaf: 
\[\mathcal{M}:=f^*\mathcal{L}\otimes \mathcal{O}_{\dot{\mathcal{C}}^d\times_B \mathcal{C}} \left(d\cdot f^*\sigma\left(B\right)-\sum_{1\le i\le d} \Delta_{i,d+1}\right)\otimes \mathcal{O}_{\dot{\mathcal{C}}^d\times_B \mathcal{C}}\left(-\sum_{\underline{i}}Z_{\underline{i}}\right)\]
on the family $\dot{\C}^d\times_B\C \ra \dot{\C}^d$
is $(f^*\sigma,f^*\mu)$-quasistable.  Thus, the sheaf $\M$ induces an Abel map:
 \begin{equation}\label{eq:Abelmap}\alpha^d_{\mathcal{L}}\col\dot{\mathcal{C}}^d\rightarrow \overline{\mc J}_\mu^\sigma
 \end{equation}
whose restriction over the generic point of $B$ is the Abel map of the generic fiber of the family $\pi\col \C\ra B$.

\bigskip

In \cite{AAJCMP} the authors describe local conditions to determine when the Abel map $\alpha_{\mathcal{L}}^d\col\dot{\C}^d\rightarrow \overline{\mc J}^\sigma_\mu$ extends to a suitable desingularization of $\C^d$. We will briefly describe this result.

We will always consider curves with smooth irreducible components.
We let $C^{\sing}$ be the set of nodes of a curve $C$. 
Consider the spectra of power series: 
\[
B=\Spec(k[[t]]) \quad \text{ and } \quad
S=\Spec(k[[u_1,\ldots,u_{d+1}]]).
\]
Consider the map $S\ra B$ given by $t=u_1\cdot u_2\cdot\ldots\cdot u_{d+1}$. When no confusion may arise, we will denote by $0$ both the closed points of $B$ and $S$. 

Let $\pi\col\mc{C}\ra B$ be a smoothing of a nodal curve $C$. Let   $\sigma\col B\ra \mc{C}$ be a section of $\pi$ through its smooth locus. We set $P:=\sigma(0)\in C$. Consider $\mc{C}_S:=\mc{C}\times_{B}S$ and let $\pi_{S}\col\mc{C}\ra S$ be the map given by pulling-back $\pi$. So we get the fiber diagram
\[
\xymatrix{
\mathcal{C}_S \ar[d]^{\pi_S} \ar[r]^{f} & \mathcal{C} \ar[d]^{\pi} \\
S \ar[r] & B.
}
\]
The key observation is that the family $\C^d\times_B \C\ra \C^d$ is a compactification of the family in \eqref{eq:familyCd}, and  the map $\pi_S\col \C_S\ra S$ is a local model for a desingularization of the family $\C^d\times_B \C\ra \C^d$.

A section 
$S\ra\mc{C}_S$ of the map $\pi_S$ induces a $B$-map $S\ra\mc{C}$ by composition. Conversely, every $B$-map $S\ra\mc{C}$ induces a section of $\pi_S$. We will abuse notation using the same name for both the section and the $B$-map.

Let $\delta\col S\ra \mc{C}$ be a $B$-map. Assume that $\delta(0)=N$, where $N\in C^{\sing}$ is a node of $C$. We can write the completion of the local ring of $\mc{C}$ at $N$ as
\[\widehat{\mc{O}}_{\mc C,N}\cong \frac{k[[x,y,t]]}{(xy-t)},\]
where $x,y$ are local \'etale coordinates of $\mc C$ at $N$.
The map $\pi\col\mc{C}\ra B$ is given by $xy=t$ locally around $N$. 

For a proper nonempty subset $A$ of $\{1,\ldots,d+1\}$, we define
\[
u_A:=\prod_{j\in A}u_j
\quad \text{ and } \quad 
u_{A^c}:=\prod_{j\in A^c}u_j,
\]
where $A^c=\{1,\ldots, d+1\}\setminus A$. Up to multiplication by an invertible element, the map $\delta$ is given by
\[
x=u_A \quad \mbox{ and }\quad y=u_{A^c},
\]
for some $A\subset \{1,\dots,d\}$.
Notice that $x$ and $y$ are the local equations of Cartier divisors of $\C$ supported on the  components of $\C$ meeting at $N$, that is, $V(x)=C_1$ and $V(y)=C_2$ locally at $N$, with $N=C_1\cap C_2$. Thus, if $Q_{j}$ is the generic point of $V(u_j)\subset S$, we see that $\delta(Q_{j})\subset V(x)$ if and only if $j\in A$.

\begin{Def}
Given sections $\delta_1,\ldots,\delta_m$ of $\pi_S$ passing through nodes of $C$, a subcurve $Y$ of $C$ and a node $N$ of $C$, we define 
\begin{eqnarray}\label{a_geometrico}
a_j^{N}(Y):=\#\{k\ |\ \delta_k(0)=N\mbox{ and }\delta_k(Q_j)\subset Y^{c} \}.
\end{eqnarray}
\end{Def}

Note that if 
$N\notin \Sigma_{Y},$ then the  index $j$ plays no role; in this case we simply write $a^N(Y).$ Also notice that if $N\in Y^{\sing}$, then $a^{N}(Y)=0.$

Let $S_{\{j\}^{c}}$ be the complement of $\bigcup_{i\neq j}V(u_i)$ in $S$, that is,
\begin{eqnarray}
S_{\{j\}^{c}}=\Spec\left(k[[u_1,\ldots,u_{d+1}]]_{u_{\{j\}^{c}}}\right).
\end{eqnarray}
Consider the natural inclusion map $S_{\{j\}^{c}}\ra S$, and let $\mc{C}_{S_{\{j\}^{c}}}:=\mc{C}_S\times_{S} S_{\{j\}^{c}}$. We denote by $g_j\col  \mc{C}_{S_{\{j\}^{c}}}\ra \C_S$ the projection onto the first factor. So we get the fiber diagram:
\[
\xymatrix{
\mathcal{C}_{S_{\{j\}^{c}}} \ar[d]^{\pi_j} \ar[r]^{g_j} & \mathcal{C}_S \ar[d]^{\pi_S} \\
S_{\{j\}^{c}} \ar[r] & S
}
\]
Moreover, we consider the composition:
\[
f_j:=f\circ g_j\col \mathcal{C}_{S_{\{j\}^{c}}}\ra \C.
\]
Let $\delta_1,\ldots,\delta_m$ be sections of $\pi_S$ passing through nodes of $C$, that is $\delta_i(0)$ is a node of $C$. If we restrict these sections to $S_{\{j\}^{c}},$ we get sections $S_{\{j\}^{c}}\ra \mc{C}_{S_{\{j\}^{c}}}$ of $\pi_j$ passing through its smooth locus.

Let $\mu$ be a polarization on $\C/B$ of degree $k$. 
Let $\L$ be a degree-$k$ invertible sheaf over $\mc{C}/B$. Denote by $\L_S$ the pullback of $\L$ to $\mc{C}_S$ via $f$. Consider the sheaf on $\C_S/S$ given by:
\begin{eqnarray}\label{eq:M}
\M:=\L_S\otimes f^{*}\O_{\mc{C}}(m\cdot\sigma(B))\otimes\I_{\delta_1(S)|\C_S}\otimes\cdots\otimes\I_{\delta_m(S)|\C_S}.
\end{eqnarray}
Note that the sheaf $\M$ induces a rational map $S\dashrightarrow \ol{\mc J}^\sigma_\mu$, since the generic fiber of $\pi_S$ is smooth. When we identify $S$ with a local model of a blow up of $\C^d$ and the sections $\delta_i$  with the projections $\C^d\to \C$, this rational map coincide  with the Abel map $\alpha_{\L}^d\col \dot{\C}^d\ra \mc J_{\mu}^\sigma$ defined in \eqref{eq:Abelmap}.\par
For each $j\in\{1,\ldots,d+1\}$, we also define 
\begin{eqnarray}
\M_j:=g^{*}_j\M.
\end{eqnarray}
Since the restrictions of the sections $\delta_1,\ldots,\delta_m$ to $S_{\{j\}^{c}}$ are sections passing through the smooth locus of $C,$ the sheaf $\M_j$ is invertible on $\C_{S_{\{j\}^c}}/S_{\{j\}^c}$. Also, since $S_{\{j\}^{c}}$ is the spectrum of a discrete valuation ring, with closed point $Q_j$, using \cite[Theorem 32 (4)]{EE01} and \cite[Lemma 4.1]{C94}, we find integers $\ell_{j,i}$ for $i\in\{1,\dots,h\}$ such that, if we let 
\begin{eqnarray}\label{eq:Zj}
Z_j=\sum_{i=1}^h\ell_{j,i}\cdot f^{*}_j C_i,
\end{eqnarray}
then the invertible sheaf on $\C_{S_{\{j\}^c}}/S_{\{j\}^c}$ defined as
\begin{eqnarray}
\M_j\otimes\O_{S_{\{j\}^{c}}}(-Z_j)
\end{eqnarray}
is $(f_j^*\sigma,f_j^*\mu)$-quasistable.

\begin{Def}
Given an invertible sheaf $\L$ of  degree $k$ on $\C/B$, sections $\delta_1,\ldots,\delta_m$ of $\pi_S$, a subcurve $Y$ of $C$ and a node $N$ in the intersection of components $C_r$ and $C_s$ of $C$, for every $j\in\{1,\dots,d+1\}$ we define
\begin{equation}\label{b_geometrico}
b_{j}^{N}(Y,\L):=\left\{\begin{array}{ll}
\ell_{j,s}-\ell_{j,r},&\mbox{if }C_r\subset Y\mbox{ and } C_s\not\subset Y;\\
\ell_{j,r}-\ell_{j,s},&\mbox{if }C_r\not\subset Y\mbox{ and } C_s\subset Y;\\
0 &\mbox{otherwise.}
\end{array}\right. \end{equation}
\end{Def}

Recall that, since we are working with curves with smooth irreducible components, for every node $N$ there are distinct indices $r$ and $s$ such that $N$ is in the intersection of $C_r$ and $C_s$. 

The key result about the local resolution of the Abel map is given by the following theorem. Recall that the sheaf $\M$ induces a rational map $S\dashrightarrow \ol{\mc J}^\sigma_\mu$ which we can see as a local model of the rational Abel map $\alpha_{\L}^d\col \dot{\C}^d\ra \ol{\mc J}_{\mu}^\sigma$ (after a suitable blowup).

\begin{Thm}[Theorem 4.2 in \cite{AAJCMP}]\label{Thm:4.2}
Let $\L$ and $\mu$ be, respectively, a polarization and an invertible sheaf on $\C/B$ of degree $k$.
Let $\delta_1,\ldots,\delta_m$ be sections of $\pi_S$ passing through nodes of $C$. There exists a map $S\ra\overline{\mc J}^\sigma_\mu$ extending the rational map defined by $\M$  if the following two conditions hold for every subcurve $Y\subset C$ containing $P=\sigma(0)$:
\begin{itemize}
\item[(1)]for every $j_1,j_2=1,\ldots,d+1$ and every node $N\in\Sigma_Y,$ we have
\[
\left|\big(a_{j_1}^{N}(Y)-b_{j_1}^{N}(Y,\L)\big)-\big(a_{j_2}^{N}(Y)-b_{j_2}^{N}(Y,\L)\big)\right|\leq1.
\]
\item[(2)] 
for every function $j\col C^{\sing}\ra\{1,\ldots,d+1\}$, we have 
\[
-\frac{\delta_Y}{2}<\deg(\L|_Y)-\mu(Y)
+\sum_{N\in C^{\sing}}
\big(a_{j(N)}^{N}(Y)-b_{j(N)}^{N}(Y,\L)\big)
\leq\frac{\delta_Y}{2}
\]
\end{itemize}
\end{Thm}

\begin{Rem}
If we incorporate the degree $\O_{\C}(m\cdot\sigma(B))$ in the degree of invertible sheaf $\L$, that is, if we start with an invertible sheaf $\L$ on $\C/B$ of degree $k+m$ and we define the sheaf  
\[
\M':=\L_S
\otimes\I_{\delta_1(S)|\C_S}\otimes\cdots\otimes\I_{\delta_m(S)|\C_S}
\]
instead of the sheaf $\M$ in equation \eqref{eq:M}, then a result as in  Theorem \ref{Thm:4.2} holds with $\M$ replaced by $\M'$, where the Condition (2) translates into the following condition:
\begin{itemize}
    \item[$(2)'$] For every function $j\col C^{\sing}\ra\{1,\ldots,d+1\},$ we have
\[-\frac{\delta_Y}{2}<\deg(\L|_Y)-\mu(Y)+\sum_{N\in C^{\sing}}\big(a_{j(N)}^{N}(Y)-b_{j(N)}^{N}(Y,\L)\big)-m\leq\frac{\delta_Y}{2}.\]
\end{itemize}
\end{Rem}

\subsection{Toric varieties}\label{Chap:Tórica}

 Throughout the paper, we will follow the notation in Fulton's book \cite{Fulton}. In particular, we write $N$ for a lattice and $M$ for its dual.  For a rational strongly convex polyhedral cone $\sigma\subset N_{\mathbb{R}}$, we let $\sigma^{\vee}$ be its dual, $S_\sigma=\sigma^{\vee}\cap M$ be its associated semigroup, $A_{\sigma}:=k[S_{\sigma}]$ be its associated algebra and $U_{\sigma}=\Spec{A_\sigma}$ be its associated affine toric variety.
  For a finite fan $\Delta$ in $N$ of rational strongly convex polyhedral cones, we denote by $X(\Delta)$ the associated toric variety. Recall that if $\Delta'$ is a refinement of $\Delta$, then there is an induced birational morphism $X(\Delta')\to X(\Delta)$.\par
 
 A \tit{polytope} in $N_{\mathbb{R}}$ is a set of the form \[\P=\Conv(W)=\left\{\sum_{\bu\in W}\lambda_\bu \bu\ |\ \lambda_\bu\in\R_{\ge0} \text{  and } \sum_{\bu\in W}\lambda_\bu=1\right\}\subset N_{\mathbb{R}},\]
where $W\subset N_{\mathbb{R}}$ is a finite set. In this case, we say that $\P$ is the \tit{convex hull of $W.$}  The \tit{set of vertices of $\P$}  is the smallest set $V_{\P}$ such that $\Conv(V_{\P})=\P$.\par


A polytope $\P\subset N_{\mathbb{R}}$ gives rise to a polyhedral cone in $N_{\mathbb{R}}\times\mathbb{R}$, called the \tit{cone over $\P$} and denoted $\cone(\P)$, which is defined as
\[
\cone(\P)=\{\lambda\cdot(\bu,1)\in N_{\mathbb{R}}\times \mathbb{R}\ |\ \bu\in \P \text{ and } \lambda\geq0\}.
\]
If $\P=\Conv(W)
$, then $\cone(\P)=\mbox{cone}(W\times\{1\}).$  To avoid cumbersome notations we will simply write $S_{\P}$ and $S_{W}$ for the semigroup associated to  $\cone(\P)$ and $\cone(\Conv(W))$. Moreover, we will write $A_{\P}$ and $A_{W}$ for the algebra associated to these semigroups, and $U_\P$ and $U_{W}$ for the associated affine toric variety.

The \tit{hypercube of dimension $n$} is the polytope $\mathcal{H}_n=\Conv(V_n)$, where 
\[
V_n=\{(a_1,\ldots,a_n)\in\NR\ |\ a_i\in\{0,1\}\}.
\]
We simply write $\mathcal{H}$ if the dimension is clear.

\begin{Prop}\label{Prop:alg_do_cone}
Let $\mathcal{H}_n$ be a hypercube of dimension $n$ and $\cone(\mc{H}_n)$ be the cone over $\mc H_n$. Then $k[S_{\mc{H}_n}]$ is a finitely generated algebra and we have
\[k[S_{\mc{H}_n}]\cong\frac{k[X_1,Y_1,\ldots,X_n,Y_n]}{I}\]
where $I$ is the ideal generated by the polynomials $X_iY_i-X_jY_j$, with $1\leq i<j\leq n.$
\end{Prop}
\begin{Rem}
We point out that the equations defining the completion $\widehat{k[S_{\mc{H}_n}]}$ are precisely the equations defining the local ring appearing in equation \eqref{eq:etalenodeint}.
\end{Rem}
\begin{proof}
The proof of the result can be found in  \cite[Proposition 3.4.1]{ID}. For the reader's convenience, we recall the argument. Let us denote $\sigma=\cone(\mc{H}_n)$. We claim that $\sigma^{\vee}$ is generated by the vectors
\[
\be_i^{*}
\quad \text{ and }\quad
\nf_i^*=-\be_i^*+\be_{n+1},
\]
 for $i=1,\ldots, n$. Notice that $\be_i^*+\nf_i^*=\be_{n+1}^*$ for all $1\leq i\leq n$.

First, we will prove that  $\cone\left(\left\{\be_i^*,\nf_i^*\right\}_{1\leq i\leq n}\right)=\sigma^{\vee}$. 
Consider a vector $\bv\in \cone\left(\mc{H}_n\right)$, with  $\bv=\sum\lambda_j \left(\bv_j,1\right)$ for $\lambda_j\geq 0$ and for the vertices $\bv_j$ of the polytope $\mc{H}_n$. Then 
\[
\left\langle \be_i^*,\bv\right\rangle=\sum\lambda_j\left\langle \be_i^{*},\left(\bv_j,1\right)\right\rangle=\sum_{\langle\be_i^*,(v_j,1)\rangle=1}\lambda_j\geq 0
\]
and 
\[
\left\langle \nf_i^{*},\bv\right\rangle=\sum\lambda_j\left\langle \nf_i^{*},\left(\bv_j,1\right)\right\rangle=\sum_{\langle\be_i^*,(v_j,1)\rangle=0} \lambda_j\geq 0.
\]
This is because $\langle \be_i^*,(v_j,1)\rangle$ is either equal to $0$ or $1$.
Thus $\be_i^{*}$ and $\nf_i^{*}$ are in  $\sigma^\vee$ for every $1\le i\le n$,  hence $\cone\left(\left\{\be_i^{*},\nf_i^{*}\right\}_{1\leq i\leq n}\right)\subset\sigma^{\vee}$.

Next we prove the other inclusion. Let $\bu=\left(b_1,b_2,...,b_n,h\right)\in \sigma^{\vee}$. Define $B^+=\left\{i | b_i\geq 0\right\}$ and $B^-=\left\{i|b_i< 0\right\}$. Let $\bv$ be the vertex of $\mc{H}_n$ with $i$-th entry equal to $1$ for $i\in B^-$ and $i$-th entry equal to $0$ for $i\in B^+$. So $\left\langle \bu,\left(\bv,1\right)\right\rangle=h+\sum_{i\in B^-}b_i\geq 0$ .
Thus
\begin{equation}
\label{eq:ueifi}
\bu=\sum_{i\in B^+}b_i \be_i^{*}+\sum_{i\in B^-}\left(-b_i\right) \nf_i^{*}+\left(h+\sum_{i\in B^-}b_i\right)\left(0,...,0,1\right). 
\end{equation}
Since all the coefficients in the above linear combination are positive, it follows that the vector $\bu$ is contained in $\cone\left(\left\{\be_i^{*},\nf_i^{*}\right\}_{1\leq i\leq n}\right)$, and hence $\sigma^{\vee}\subset\cone\left(\left\{\be_i^{*},\nf_i^{*}\right\}_{1\leq i\leq n}\right)$ .

Now, if $\bu=\left(b_1,b_2,...,b_n,h\right)\in S_\sigma=\sigma^{\vee}\cap M$ then $b_1,...,b_n,h\in \mathbb Z$. By equation \eqref{eq:ueifi}, we have that $u$ is a $\mb{Z}$-linear combination of  $\left\{\be_i^{*},\nf_i^{*}\right\}$, that is,  $S_\sigma$ is also generated (as a semigroup) by  $\left\{\be_i^{*},\nf_i^{*}\right\}$.
Define $X_i:=\chi^{\be_i}$ e $Y_i:=\chi^{\nf_i^{*}}$. The algebra $k\left[S_\sigma\right]$ is finitely generated by  $X_i$ and $Y_i$. Since $\left\{\be_i^{*}\right\}$ and $\left\{\nf_i^{*}\right\}$ are sets of linearly independent vectors and since the only relations follow from the equality  $\be_i^{*}+\nf_i^{*}=\left(0,...,0,1\right)$, we conclude that the only relations between $\left\{X_i\right\}$ and $\left\{Y_i\right\}$ are of type $X_iY_i=X_jY_j$ for all $1\leq i,j\leq n$. This concludes the proof.
\end{proof}

Next, consider the projection map 
\begin{equation}
\label{eq:mapa_cone}
\begin{aligned}
\pi\col\quad\cone\left(\mc{H}_n\right)\ \quad&\rightarrow \R_{\geq 0}\\
\left(a_1,a_2,...,a_n,h\right)&\mapsto \ h.
\end{aligned}
\end{equation}
We can write the map $\pi$ as  \[\pi\left(a_1,a_2,...,a_n,h\right)=\left\langle \left(a_1,a_2,...,a_n,h\right),\left(0,...,0,1\right) \right\rangle.\] Recall that $S_{\R_{\geq0}}\subset \textnormal{Hom}\left(\R,\R\right)$ and $S_{\mc{H}_n}\subset \textnormal{Hom}\left(\R^{n+1},\R\right)$). The map $\pi$ induces a dual map of semigroups: 
\begin{align*}
\pi^\vee\col S_{\R_{\geq0}}&\ra\  S_{\mc{H}_n}\\
\varphi\ &\mapsto \varphi\circ\pi.
\end{align*}
 Since $\R_{\geq0}=\cone(\be_1)$, we have $S_{\R_{\geq0}}=\left\langle \be^{*}_1\right\rangle$ and $\pi^\vee(\be^{*}_1)=\left(0,...,0,1\right)$.
Now, the  map $\pi^\vee$ also induces a morphism of $k$-algebras 
\[
k\left[S_{\R_{\geq0}}\right]\cong k\left[t\right]\rightarrow k\left[S_{\mc{H}_n}\right]
\]
that takes $t=\chi^{e_1^{\ast}}$ to  $\chi^{\left(0,...,0,1\right)}$. Since $\be_i+\nf_i=\left(0,...,0,1\right)$, the morphism of algebras takes $t$ to $\chi^{\be_i}\cdot\chi^{\nf_i}=X_i\cdot Y_i$. This induces the toric morphism 
\[\Spec\left(\dfrac{k[X_1,Y_1,\ldots,X_n,Y_n,t]}{\langle X_1Y_1-t,\dots,X_nY_n-t\rangle}\right)\to \Spec\left(k[t]\right)=\mathbb A^1.
\]

\begin{Def}
\label{def:volume}
A \tit{$d$-simplex} (or just a \tit{simplex}) is a polytope of   dimension $d$ with $d+1$ vertices. 
The \emph{volume} of a simplex $\Delta$ of maximal dimension and vertices $V=\{\bv_0,\ldots,\bv_n\}$ is
\[\mbox{vol}(\Delta)=\frac{1}{n!}\cdot \det[(\bv_0,1),\ldots,(\bv_n,1)].\]
An integral simplex $\Delta$ of maximal dimension is said to be \emph{unitary} if vol$(\Delta)=\frac{1}{n{\displaystyle !\,}}.$
\end{Def}

Given a simplex $\Delta$, we define the map
\begin{align*}
    \beta\col&\ \quad \cone(\Delta)\quad \ra\mb{R}_{\geq0}\\
    &(a_1,\ldots,a_n,h)\mapsto\ h
\end{align*}
The dual map of semigroups induced by $\beta$ is a map $\beta^{\vee}\: S_{\mb{R}_{\geq0}}\ra S_{\Delta}$ such that $\beta^{\vee}(\be^{*}_1)=\be_{n+1}^*$.

\begin{Prop}\label{Prop:simpplexo_unitario} 
Let $\Delta=\Conv(V)$ be a unitary simplex contained in the hypercube $\mc{H}_n$, for some subset of vertices $V=\{\bv_0,\ldots,\bv_n\}$ of  $\mc{H}_n$. Then there are vectors $\bu_0,\bu_1,...,\bu_n\in M\times \mathbb Z$ such that  $S_\Delta=\left\langle \bu_0,\bu_1,...,\bu_n\right\rangle$ and $\left\langle \bu_i,\left(\bv_j,1\right)\right\rangle=\delta_{ij}$ (Kronecker's delta). Moreover, we have
\[
k\left[S_\Delta\right]\cong k\left[U_0,U_1,...,U_n\right]
\]
where $U_i=\chi^{\bu_i}$, and the map of $k$-algebras $k\left[t\right] \rightarrow k\left[S_{\Delta}\right]$ induced by $\beta^{\vee}$ is given by $t=U_0 U_1\ldots U_n$.
\end{Prop}

\begin{proof}
The proof of the result can be found in  \cite[Proposition 3.5.1]{ID}. For the reader's convenience, 
we will recall the argument.
Let $A$ be the matrix $\left[\left(\bv_0,1\right),\left(\bv_1,1\right),...,\left(\bv_n,1\right)\right]$, where $\left(\bv_i,1\right)$ is the $i$-th row of $A$. Let $A^{-1}=\left[\bu_0,\bu_1,...,\bu_n\right]$, where $\bu_i$ is the $i$-th column of $\cof\left(A\right)$. We have that $\bu_i\in M\times \mathbb{Z}$ because $\det(A)=1$ (since $\Delta$ is unitary). Thus $\left\langle \left(\bv_i,1\right),\bu_j\right\rangle=\delta_{ij}$, and hence $\{\bu_0,\dots,,\bu_n\}\subset \tau^{\vee}$, where $\tau=\cone\left(\Delta\right)$. \par

Let $\bu\in S_\Delta=S_{\tau}=\tau^{\vee}\cap M$ and take
\[
\bv:=\bu-\left\langle \bu,\left(\bv_0,1\right)\right\rangle\cdot \bu_0-\left\langle \bu,\left(\bv_1,1\right)\right\rangle\cdot \bu_1-...-\left\langle \bu,\left(\bv_n,1\right)\right\rangle\cdot \bu_n.
\] 
We have $\left\langle \bv,(\bv_i,1)\right\rangle=0$ for all $1\leq i\leq n$, then $\bv\cdot A=0$. Since $\det\left(A\right)\neq 0$ we have that $\bv=0$. It means that the semigroup $S_{\tau}=S_\Delta$ is generated by the vectors $\bu_0,\dots,\bu_n$. \par
    If we let $\bu=\bu_0+\bu_1+...+\bu_n$, then we have   $\left\langle \bu,(\bv_i,1)\right\rangle=1$ for all $i\in\{0,\dots,n\}$. We also have  $\left\langle \left(0,...,0,1\right),\left(\bv_i,1\right)\right\rangle=1$ for all 
    $i\in\{1,\dots,n\}$. Then, as before, we have that $\bu=\bu_0+\bu_1+...+\bu_n=\left(0,...,0,1\right)$.
Since the map of semigroups $S_{\R_{\geq0}} \rightarrow S_{\Delta}$ sends $\be_1^{\ast}$ to $\left(0,...,0,1\right)$, the morphism $k\left[t\right] \rightarrow k\left[S_{\tau}\right]$ takes $t=\chi^{\be_1^{\ast}}$ to $\chi^{\bu_0} \chi^{\bu_1}\ldots\chi^{\bu_n}=U_0 U_1\ldots  U_n$. 
\end{proof}

We conclude this section with some definitions that will became important later. We will use toric varieties to describe the local equations of the products of a family of curves, and their blowups. For this reason, we discuss the possibility of taking the completions of the algebras defined so far, and give notations for the varieties obtained after taking these completions.

Let $\mc{H}_n$ be a hypercube of dimension $n.$ The $k$-algebra associated to the $\cone(\mc{H}_n)$ is 
\[A_{\mc{H}_n}=\frac{k[X_1,Y_1,\ldots,X_n,Y_n]}{\langle X_iY_i-X_jY_j\rangle_{1\leq i,j\leq n}}\]
and its toric variety is $U_{\mc{H}_n}=\Spec(A_{\mc{H}_n}).$ Let us denote the $\Spec$ of the completion of $A_{\mc{H}_n}$ by
\[\wh{U}_{\mc{H}_n}:=\Spec\left( \frac{k[[X_1,Y_1,\ldots,X_n,Y_n]]}{\langle X_iY_i-X_jY_j\rangle_{1\leq i,j\leq n}}\right).
\]
We call $\widehat{U}_{\mc H_n}$ the \emph{local toric variety associated to $\cone(\mc H_n)$}. 
Notice that there is an inclusion 
\[
A_{\mc{H}_n}\longrightarrow \frac{k[[X_1,Y_1,\ldots,X_n,Y_n]]}{\langle X_iY_i-X_jY_j\rangle_{1\leq i,j\leq n}}
\]
which induces a morphism: 
\[
\wh{U}_{\mc{H}_n}\longrightarrow U_{\mc{H}_n}.\]
If $\Delta$ is a fan that refines  $\cone(\mc{H}_n)$, 
then there is a toric morphism $X(\Delta)\ra  U_{\mc{H}_n}$.
 In this case, we define the scheme $\widehat{X}(\Delta)$ as
 \begin{equation}\label{eq:localtoricvariety}
 \widehat{X}(\Delta):= \wh{U}_{\mc{H}_n}\times_{U_{\mc{H}_n}} X(\Delta).
 \end{equation}
We call $\widehat{X}(\Delta)\ra \wh{U}_{\mc{H}_n}$ the \emph{local toric blowup associated to the refinement $\Delta$}.

\section{Graphs and tropical curves}

Throughout the paper, we will use the notations and terminology introduced in  \cite[Sections 2.1, 3.1, 4, 5]{AAMPJac}.
  In particular, given a graph $\Gamma$ and a  
subset of edges $\E\subset E(\Gamma)$, we define the valence $\val_\E(v)$ of a vertex $v$ of $\Gamma$ as the number of edges in $\E$ incident to v (with loops counting twice). Moreover, given a subset of vertices $V\subset V(\Gamma)$, we let $\delta_{\Gamma,V}$ (or simply $\delta_V$) be the number of edges connecting a vertex of $V$ and a vertex in $V^c=V(\Gamma)\setminus V$. Finally, a divisor on $\Gamma$ is a function $D\col V(\Gamma)\ra \mathbb Z$, and  $\Div(\Gamma)$ denotes the group of divisors on $\Gamma$.
In this section we just recall some results which will be important later.

\subsection{Graphs}
Let $\Gamma$ be a graph.
We denote by $\G^{(n)}$ the graph obtained from $\G$ by inserting exactly $n-1$ vertices in the interior of each edge of $\G$. We call $\G^{(n)}$ the \tit{$n$-subdivision} of $\G$. Notice that, if we have an orientation $\ora{\G}$ on $\G$,
then there is a natural induced orientation $\ora{\G^{(n)}}$ on $\G^{(n)}$. Given an edge $e\in E(\ora{\Gamma})$ we let $s(e)$ and $t(e)$ the source and target of $e$. See Figure \ref{Fig:orient} for the orientation induced on the $2$-subdivision of an edge.
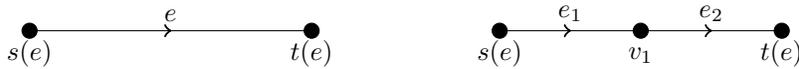
\begin{figure}[h!]
\begin{center}
\begin{tikzpicture}[scale=5]
\begin{scope}[shift={(0,0)}]
\draw (0,0) to (0.75,0);
\draw[fill] (0,0) circle [radius=0.02];
\draw[fill] (0.75,0) circle [radius=0.02];
\draw[->, line  width=0.3mm] (0.375,0) -- (0.377,0);
\node[below] at (0.75,0) {$t(e)$};
\node[below] at (0,0) {$s(e)$};
\node[above] at (0.375,0) {$e$};
\end{scope}
\begin{scope}[shift={(1.25,0)}]
\draw (0,0) to (0.75,0);
\draw[fill] (0,0) circle [radius=0.02];
\draw[fill] (0.75,0) circle [radius=0.02];
\draw[fill] (0.375,0) circle [radius=0.02];
\draw[->, line  width=0.3mm] (0.1875,0) -- (0.1895,0);
\draw[->, line  width=0.3mm] (0.5625,0) -- (0.5642,0);
\node[below] at (0.75,0) {$t(e)$};
\node[below] at (0,0) {$s(e)$};
\node[above] at (0.1875,0) {$e_1$};
\node[above] at (0.5625,0) {$e_2$};
\node[below] at (0.375,-0.02) {$v_1$};
\end{scope}
\end{tikzpicture}\end{center}
\caption{The orientation induced on a $2$-subdivision}\label{Fig:orient}
\end{figure}

There are an inclusion and a surjection 
\begin{equation}\label{eq:FG}
F\col V(\G)\ra V(\G^{(n)})
\quad \text{ and } \quad
G\col E(\G^{(n)})\ra E(\G)
\end{equation}
such that for each edge $e$ of $\G$ there exist distinct vertices $x^e_{0},\ldots, x^e_{n}\in V(\G^{(n)})$ and distinct edges $e_{1},\ldots, e_{n}\in E(\G^{(n)})$ such that
\begin{itemize}
\item[(1)] $x^e_{0}=F(v_{0}),\ x^e_{n}=F(v_{1})$ and $x^e_{i}\notin \mbox{Im}(F)$ for every $i=1,\ldots,n-1,$ where $v_{0}$ and $v_{1}$ are precisely the vertices of $\G$ incident to $e;$
\item[(2)] $G^{-1}(e)=\{e_{1},\ldots,e_{n}\};$
\item[(3)] the vertices in $V(\G^{(n)})$ incident to $e_{i}$ are precisely $x^e_{i-1}$ and $x^e_{i}$, for $i=1,\ldots,n.$
\end{itemize}
The edges $e_{1},\ldots,e_{n}$ (respectively, the vertices $x^e_{1},\ldots,x^e_{n-1}$) are called the edges (respectively, the exceptional vertices) \emph{over} $e$. More generally, a \emph{refinement} of $\Gamma$ is a graph obtained by inserting a certain number of vertices in the interior of each edge of $\Gamma$.

A \tit{cycle} on $\ora{\G}$ is just a cycle on $\G.$ Given a cycle $\gamma$ on $\ora{\G},$ we write 
 \begin{equation}\label{eq:gamma}
 \gamma(e)=
 \begin{cases}
 1 & \text{if $\gamma$ follows the orientation of $e\in \ora{\G}$;}\\
 0 & \text{if $e\not\in \gamma$;}\\
 -1 & \text{otherwise.}
 \end{cases}
 \end{equation}
 Notice that the definition of $\gamma(e)$ depends on the chosen orientation of $\Gamma$. Nevertheless, a change of orientation may only change signs of all  $\gamma(e)$, which is harmless for our purposes.

Given a graph $\G$ and a ring $A$, we define the \tit{$A$-module of 0-chains} and the \tit{$A$-module of 1-chains with coefficients in $A$}, respectively, as the sets
\begin{eqnarray}\label{Eq:c0c1}
 C_0(\G,A):=\bigoplus_{v\in V(\G)}A\cdot v\ \mbox{ and }\ 
 C_1(\G,A):=\bigoplus_{e\in E(\G)}A\cdot e.
\end{eqnarray}
We identify a $0$-chain $g=\sum_{v\in V(\G)}a_v\cdot v$ with the function $g\col V(\G)\ra A$ given by $g(v)=a_v.$ Similarly, a $1$-chain $\varepsilon=\sum_{e\in E(\G)}a_e\cdot e$ is identified with the function $\varepsilon\col E(\G)\ra A$ given by $\varepsilon(e)=a_e.$
 
Fix an orientation on $\G$, i.e., choose a digraph $\ora{\G}$ with $\G$ as underlying graph. 
We define the \tit{differential operator} $\delta\col C_0(\ora{\G},A)\ra C_1(\ora{\G},A)$ as the linear operator taking a generator $v$ of $C_0(\ora{\G},A)$ to
\begin{equation}\label{eq:delta}
\delta(v):=\sum_{\substack{e\in E(\ora{\G})\\ t(e)=v}} e-\sum_{\substack{e\in E(\ora{\G})\\ s(e)=v}}e.
\end{equation}
The \tit{adjoint of $\delta$} is the linear operator $\partial\col C_1(\ora{\G},A)\ra C_0(\ora{\G},A)$ taking a generator $e$ of $C_1(\ora{\G},A)$ to
\begin{equation}\label{eq:partial}
\partial(e):=t(e)-s(e).
\end{equation}
The space of $1$-cycles of $\ora{\G}$ (over $A$) is defined as $H_1(\ora{\G},A):=\ker(\partial).$


Note that we have a natural identification between $C_0(\G,\mathbb{Z})$ and the group $\Div(\G)$ of divisors on $\Gamma$. The composition $\partial\delta\col C_0(\ora{\G},A)\ra C_0(\ora{\G},A)$ does not depend on the choice of the orientation. The \tit{group of principal divisors on $\G$} is the subgroup $\Prin(\G):=\Im(\partial\delta)$ of $\Div(\G).$ Given $D,D'\in\Div(\G)$, we say that $D$ is \tit{equivalent} to $D'$, and we write $D\sim D'$, if $D-D'\in\Prin(\G).$ 
The \tit{degree $d$ Picard group} of a graph $\G$ is defined as 
\[\Pic^d(\G)=\Div^d(\G)/\sim.\]\par

For a subset $\mc{E}\subset E(\G)$, we define the $\mc{E}$-subdivision $\G^{\mc{E}}$ of $\G$ as the graph obtained from $\G$ by adding exactly one vertex in the interior of every edge in $\mc{E}$.
A \textit{pseudo-divisor} on $\G$ is a pair $(\mc{E},D)$ where $\mc{E}\subset E(\G)$ and $D$ is a divisor on $\G^{\mc{E}}$ such that $D(v)=-1$ for every exceptional vertex $v\in V(\G^{\mc{E}})$. If $\mc{E}=\emptyset$, then $(\mc{E},D)$ is just a divisor of $\G$. 

A \emph{flow} $\phi$ on $\ora{\Gamma}$ is a function $\phi\col E(\ora{\Gamma})\ra \mathbb Z$.
Given a flow $\phi$ on $\ora{\G}$, we define the divisor associated to $\phi$ as the image div$(\phi)=\partial(\phi),$ where $\phi$ is seen as an element of $ C_1(\ora{\G},\mathbb{Z})$.

We conclude the section describing the relation between torsion-free rank-1 sheaves on a nodal curve $C$ and pseudo-divisors on the \emph{dual graph} $\G=\G_C$ of $C$, which is the graph defined  as follows. To each component of $C$ we have a vertex of $\G$, and to each node we have an edge of $\G$, where two vertices $v,v'$ are connected by an edge $e$ if and only if the two components of $C$ corresponding to $v$ and $v'$ intersect at the node of $C$ corresponding to $e$.

Let $I$ be a degree-$d$ rank-$1$ torsion-free sheaf on a nodal curve $C$. We can define a pseudo-divisor $(\E_I,D_I)$ on $\Gamma_C$, called the \emph{multidegree} of $I$, as follows. The set $\E_I\subset E(\Gamma_C)$ is precisely the set of edges corresponding to nodes where $I$ is not locally free. For every $v\in V(\Gamma^\E_C)$, we set
\[
D_I(v)=\begin{cases}
        \deg(I|_{C_v}),&\text{ if $v\in V(\Gamma_C)$};\\
				-1,&\text{ if $v$ is exceptional},
				\end{cases}
\]
where $C_v$ is the component of $C$ corresponding to $v\in V(\Gamma_C)$.

Note that if $I$ is invertible, then $\E_I$ is empty and hence $(\E_I,D_I)$ is a divisor on $\Gamma_C$. Moreover, if $\mc C\ra B$ is a smoothing of a nodal curve $C$ and $\L$ is an invertible sheaf of type $\L=\O_{\mc C}(\sum_{v\in V(\Gamma_C)} \ell_v C_v)$ for $\ell_v\in\mathbb Z$,  then $D_{\L|_C}=\partial\delta (\sum_{v\in V(\Gamma_C)} \ell_v v)$, which is a principal divisor of $\Gamma_C$.

\subsection{Tropical curves}
A \tit{metric graph} is a pair $(\G,\ell)$ where $\G$ is a graph and $\ell$ is a function $\ell\col E(\G)\ra\R_{>0}$ called the \tit{length function}.
Let $(\G,\ell)$ be a metric graph. If $\ora{\G}$ is an orientation on $\G,$ we define the \tit{tropical curve} $X$ associated to $(\ora{\G},\ell)$ as
\[X=\frac{\left(\bigcup_{e\in E(\ora{\G})}I_{e}\cup V(\ora{\G})\right)}{\sim}\]
where $I_{e}=[0,\ell(e)]\times\{e\}$ and $\sim$ is the equivalence relation generated by $(0,e)\sim s(e)$ and $(\ell(e),e)\sim t(e).$ 
We usually just write $e$ to represent $I_e$ in $X$, and denote by $e^\circ$  the interior of $e$.

Let $X$ be a tropical curve with a model $\G_{X}$, and let $Y\subset X$ be a tropical subcurve of $X$. Then, there exists a \tit{minimal refinement} $\G_{X,Y}$ of $\G_{X}$ such that $Y$ is induced by a subgraph $\G_{Y}$ of $\G_{X,Y}$. We define
\begin{equation}
\delta_{X,Y}:=\sum_{v\in V(\G_{Y})}\mbox{val}_{E(\G_{X,Y})\setminus E(\G_{Y})}(v).
\end{equation}
When no confusion may rise we will simply write $\delta_{Y}$ instead of $\delta_{X,Y}$.

A \tit{divisor} on $X$ is a function $\mc{D}\col X\ra\mathbb{Z}$ such that $\mc{D}(p)\neq0$ only for finitely many points $p\in X.$ We define the \tit{support of} $\mc{D}$ as the set of points $p$ of $X$ such that $\mc{D}(p)\neq0$ and denote it by $\supp(\mc{D}).$ 

Every divisor on the model $\G$ of $X$ can be seen as a divisor on $X$. Given a divisor $\mc{D}$ on $X$, the degree of $\mc{D}$ is the integer $\deg\mc{D}:=\sum_{p\in X}\mc{D}(p).$ 
We let $\Div(X)$ be the Abelian group of divisors on $X$. We denote by $\Div^{d}(X)$ the subset of degree-$d$ divisors on $X.$ 

A \tit{rational function} on the tropical curve $X$ is a continuous, piece-wise linear function $f\col X\ra \R$ with integer slopes. Given an integer $s$ and an edge $e\in E(\Gamma)$, we say that a rational function $f$ on $\G$ has \tit{slope $s$ over the segment} $[p,q]$, for $p,q\in e$ if the restriction of $f$ to $[p,q]$ is linear and has slope $s$. A \tit{principal divisor} on $\G$ is a divisor 
\[\div(f):=\sum_{p\in X}\ord_{p}(f)p\ \in\Div(X),\]
where $f$ is a rational function on $X$ and $\ord_p(f)$ is the sum of the incoming slopes of $f$ at $p$. It is easy to check that a principal divisor has degree zero. The \tit{support} of a rational function $f$ on $X$ is defined as the set \[\supp(f)=\{p\in X|\ \ord_p(f)\neq0\}.\] We denote by $\Prin(X)$ the subgroup of Div$(X)$ of principal divisors. Given divisors $\mc{D}_1,\mc{D}_2\in\Div(X),$ we say that $\mc{D}_1$ and $\mc{D}_2$ are \tit{equivalent}, and we write $\mc D\sim \mc D'$,  if $\mc{D}_1-\mc{D}_2\in \mbox{Prin}(X).$ 
The \tit{Picard group} $\Pic(X)$ of $X$ and the \tit{degree-$d$ Picard group} $\Pic^d(X)$ of $X$ are defined, respectively, as
\[
\Pic(X):=\Div(X)/\Prin(X) 
\quad \text{ and } \quad
\Pic^{d}(X):=\Div^{d}(X)/\sim.
\]

Let $f$ be a rational function on the tropical curve $X$ and $\G$ be a model of $X$ such that supp$(f)\subset V(\G).$ Then $f$ is linear over each edge of $\G$. 
In this case, we can define a  flow $\phi_f$ on $\ora{\G}$ where $\phi_{f}(e)$ is equal to the slope of $f$ over $e,$ for every $e\in E(\ora{\G}).$ 
Notice that $\div_{X}(f)=\div(\phi_f)$, where div$(\phi_f)$ is seen as a divisor on $X$. 
Conversely, if $\phi$ is a flow on $\Gamma$, we have that $\phi$ induces a rational function $f$ on $X$ such that $\phi_f=\phi$ if, and only if,
\begin{equation}\label{eq:functionflow}
\sum_{e\in E(\G)} \phi(e)\gamma(e)\ell(e)=0
\end{equation}
for every cycle $\gamma$ on $\G$. (Recall the definition of $\gamma(e)$ in equation \eqref{eq:gamma}).


We let $\Omega(X)$ be the real vector \tit{space of harmonic $1$-forms} and $\Omega(X)^{\vee}$ be its dual. 
Given edges $e,e'\in E(\ora{\G_{X}})$ and points $p,q\in e'$, we define the \tit{integration of $de$ over} $\ora{pq}$ as 
\[\int_{p}^{q}de=\left\{
\begin{array}{ll}
\ell([p,q]) & \mbox{if } e'=e;\\
0, & \mbox{otherwise.}
\end{array}\right.\]
We have a natural isomorphism (\cite[Lemma 2.1]{BF}): 
\[\begin{array}{rcl}
H_1(\G_X,\R)&\ra&\Omega(X)^{\vee}\\
\gamma&\mapsto&\int_{\gamma}\col
\end{array}\]
Hence $H_1(\G_X,\mathbb{Z})$ can be viewed as a lattice in $\Omega(X)^{\vee}.$ The \tit{(tropical) Jacobian} of the tropical curve $X$ is defined as the real torus (see also \cite{MZ} for more properties of the tropical Jacobian)
\begin{equation}\label{eq:Jtropdef}
J^{\trop}(X)=\frac{\Omega(X)^{\vee}}{H_1(\G_X,\mathbb{Z})}.
\end{equation}


Fix a point $p_0$ in $X$ and assume that $p_0$ is a vertex of $\ora{\G_X}$. Let $p_{1},\ldots,p_{d}$ be points on $e_1, \ldots, e_d\in E(\ora{\Gamma_X})$. Choose a path $\gamma_i$ on $\ora{\G_X}$ from $p_0$ to $s(e_i)$ for every $i=1,\ldots,d.$ One can define a map 
\begin{equation}
\label{eq:alpha}
\begin{array}{rcl}
  X^{d}&\stackrel{\chi}{\longrightarrow}&\Omega(X)^{\vee}\\
(p_{e_i,a_i})_{i=1,\ldots,d}&\mapsto&\displaystyle\sum_{i=1}^{d}\left(\int_{\gamma_i}+\int_{s(e_i)}^{p_{i}}\right).
\end{array}
\end{equation}
The \tit{degree-$d$ tropical Abel map} of the tropical curve $X$ is the composition 
\begin{equation}
\label{eq:chi}
\alpha_d^\trop\col  X^{d}\stackrel{\chi}{\longrightarrow} \Omega(X)^{\vee}\stackrel{\rho}{\longrightarrow} J^{\trop}(X)
\end{equation}
where $\rho\col\Omega(X)^{\vee}\ra J^{\trop}(X)$ is the natural quotient map. Notice that, while the map $\chi\col X^{d}\ra\Omega(X)^{\vee}$ may depend on the choices of the paths $\gamma_1,\ldots,\gamma_d,$ the map $\alpha_d^\trop$ does not.



\subsection{Quasistability on graphs}\label{sec:quasi-graph}

We introduce the notion of quasistability for pseudo-divisors on graphs. All graphs will be considered connected unless otherwise specified.

Let $\G$ be a graph and  $d$ be an integer. A \tit{degree-$d$ polarization} on $\G$ is a function $\mu\col V(\G)\ra\R$ such that $\sum_{v\in V(\G)}\mu(v)=d.$
For every subset $V\subset V(\G)$ we define $\mu(V):=\sum_{v\in V}\mu(v).$ 



Given a subdivision $\G^{\E}$ of $\G$ for some $\E\subset E(\G)$, there is an induced degree-$d$ polarization $\mu^{\E}$ on $\G^{\E}$ taking a vertex $v\in V(\Gamma^\E)$ to:
\[\mu^{\E}(v):=\left\{
\begin{array}{ll}
\mu(v)& \mbox{if }v\in V(\G);\\
0 & \mbox{otherwise.}
\end{array}\right.
\]

Consider a graph $\G$, a degree-$d$ polarization $\mu$ on $G$, and a subset $V\subset V(\G).$ Given a degree-$d$ divisor $D$ on $\G$, 
we define
\[
\beta_{D}(V):=\mbox{deg}(D|_{V})-\mu(V)+\frac{\delta_{V}}{2}.
\]

\begin{Lem}[Lemma 4.1 in \cite{AAMPJac}]
\label{lem:beta}
Given subsets $V$ and $W$ of $V(\G),$ we have
\[\beta_{D}(V\cup W)+\beta_{D}(V\cap W)=\beta_{D}(V)+\beta_{D}(W)-|E(V,W)|.\]
In particular, $\beta_{D}(V)+\beta_{D}(V^{c})=\delta_{V}.$
\end{Lem}

Let $\G$ be a graph and $\mu$ be a degree-$d$ polarization on $\G$. 
Consider a vertex $v_0\in V(\G)$. We say that a degree-$d$ divisor $D$ on $\G$ is \tit{$(v_0,\mu)$-quasistable} if $\beta_{D}(V)\geq0$ for every proper subset $V\subsetneq V(\G),$ with strict inequality if $v_0\in V$. Equivalently, interchanging $V$ and $V^{c}$, we could ask $\beta_{D}(V)\leq \delta_{V},$ with strict inequality if $v_0\notin V.$ We say that a pseudo-divisor $(\E,D)$ on $\Gamma$ is \tit{$(v_0,\mu)$-quasistable} if $D$ is $(v_0,\mu^{\E})$-quasistable on $\G^{\E}.$

\begin{Thm}[Theorem 4.3 in \cite{AAMPJac}]\label{Thm:quasistable-divisor}
Every divisor $D$ on a graph $\G$ is equivalent to a unique $(v_0,\mu)$-quasistable divisor.
\end{Thm}

\begin{Rem}\label{Rem:div_quasistable} If $D$ is a $(v_0,\mu^{\E})$-quasistable divisor on $\G^{\E}$ then $D(v)=0,-1$ for every exceptional vertex $v\in V(\G^{\E}).$ Moreover if $\widehat{\G}$ is a refinement of $\G$ then $\mu$ induces a polarization $\widehat{\mu}$ in $\widehat{\G}$ given by
\[\widehat{\mu}(v)=
\left\{\begin{array}{ll}
0, & \mbox{if }v\mbox{ is exceptional;}\\
\mu(v), & \mbox{if }v\in V(\G),
\end{array}\right.
\]
where we view $V(\G)$ as a subset of $V(\widehat{\G})$ via the natural injection $F\col V(\G)\ra V(\widehat{\G})$. 
If $D$ is a $(v_0,\widehat{\mu})$-quasistable divisor on $\widehat{\G}$ then for every edge $e\in E(\G),$ we have that $D(v)=0$ for all but at most one exceptional vertex $v$ over $e$; if such a vertex $v$ over $e$ exists, then $D(v)=-1.$ Hence, every $(v_0,\widehat{\mu})$-quasistable divisor $D$ of $\widehat{\G}$ induces a $(v_0,\mu)$-quasistable pseudo-divisor $(\E',D')$ on $\G.$
\end{Rem}




Let $\G$ be a graph, $v_0$ a vertex of $\G,$ and $\mu$ a polarization on $\G$. The set $\mc{QD}_{v_0,\mu}(\G)$ of $(v_0,\mu)$-quasistable pseudo-divisors on $\G$ is a poset, where the partial order is given by $(\E,D)\geq(\E',D')$ if  $(\E,D)$ specializes to $(\E',D')$ (see \cite[Section 2.1 and Figure 3]{AAMPJac}). 

\begin{Rem}\label{Rem:multidegree}
Consider a nodal curve $X$ with a point $p_0$ and a degree-$d$ polarization $\mu$. Let $\I$ be a degree-$d$ simple torsion-free rank-1 sheaf on $X$. We have that $\I$ is $(p_0,\mu)$-quasistable on $X$ if and only if its multidegree $(\E_\I,D_\I)$ is a $(v_0,\mu)$-quasistable pseudo-divisor on the dual graph of $X$, where $v_0$ is the vertex corresponding to the component of $X$ containing $p_0$. 
\end{Rem}

\subsection{Quasistability on tropical curves}


Let $X$ be a tropical curve. A \tit{degree-$d$ polarization} on $X$ is a function $\mu\col X\to\R$ such that $\mu(p)=0$ for all, but finitely many $p\in X$, and $\sum_{p\in X}\mu(p)=d$. We define the \emph{support} of $\mu$ as 
\[
\supp(\mu):=\{p\in X;\mu(p)\neq 0\}.
\]
For every tropical subcurve $Y\subset X$, we define $\mu(Y):=\sum_{p\in Y}\mu(p)$. 
For any divisor $\D$ on $X$ and every tropical subcurve $Y\subset X$, we set 
\[
\beta_\D(Y):=\deg(\D|_Y)-\mu(Y)+\frac{\delta_Y}{2}.
\]
Given a polarization $\mu$ on $X$, we will always consider a model $\Gamma_X$ for the curve $X$ such that $\supp(\mu)\subset V(\G_X)$. 

\begin{Lem}[Lemma 5.1 in  \cite{AAMPJac}]
\label{lem:beta_curvas} Given tropical subcurves $Y$ and $Z$  of a tropical curve $X$, we have
\[
\beta_\D(Y\cap Z)+\beta_\D(Y\cup Z)=\beta_\D(Y)+\beta_\D(Z).
\]
In particular, if $Y\cap Z$ consists  of a finite number of points $p$ such that $p\notin V(\Gamma_X)$, $\mu(p)=0$ and $\D(p)=0$, then 
\[
\beta_\D(Y\cup Z)=\beta_\D(Y)+\beta_\D(Z)-|Y\cap Z|.
\]

\end{Lem}

\begin{Def}\label{def:quasistable}
Let $X$ be a tropical curve and $\mu$ be a degree-$d$ polarization on $X$. Let $\D$ be a degree-$d$ divisor on $X$. 
Given a point $p_0$ of $X$, we say that $\D$ is \emph{$(p_0,\mu)$-quasistable} if $\beta_\D(Y)\ge0$ for every subcurve $Y\subset X$  and $\beta_\D(Y)>0$ for every proper subcurve $Y\subset X$ such that $p_0\in Y$. 
\end{Def}

\begin{Rem}
Equivalently, a divisor $\D$ is $(p_0,\mu)$-quasistable if and only if for every tropical subcurve $Y\subset X$ we have $\beta_\D(Y)\leq\delta_{X,Y}$, with strict inequality if $p_0\notin Y$.
\end{Rem}

Let us illustrate the relationship between 
the quasistability for divisors on tropical curves and on graphs. 

\begin{Prop}[Proposition 5.3 in \cite{AAMPJac}]
\label{prop:quasiquasi}
Let $X$ be a tropical curve with a point $p_0\in X$. Let $\Gamma_X$ be a model of $X$. Let $\mu$ be a polarization on $X$ induced by a polarization on $\Gamma_X$ and $\D$ a degree-$d$ divisor on $X$. Let $\wh{\Gamma}_X$ be a refinement of $\Gamma_X$ such that $\supp(\D)\subset V(\wh \Gamma_X)$. The degree-$d$ divisor $\D$ on $X$ is $(p_0,\mu)$-quasistable if and only if $D$ is $(p_0,\wh\mu)$-quasistable on $\wh{\Gamma}_X$, where $D$ is the divisor $\D$ seen as divisor on $\wh{\Gamma}_X$ and $\wh\mu$ is the polarization induced on $\wh{\Gamma}_X$ (see Remark \ref{Rem:div_quasistable}).
\end{Prop}
\begin{CorDef}[Corollary-Definition 5.4 in \cite{AAMPJac}]
\label{cor:quasiquasi}
Keep the notations of Proposition \ref{prop:quasiquasi} and let $\D$ be a $(p_0,\mu)$-quasistable degree-$d$ divisor on $X$. Then $\wh{\Gamma}_{X}$ is an $\E$-subdivision of $\Gamma_X$ for some $\E\subset E(\Gamma_X)$, and the pair $(\E,D)$ is a $(p_0,\mu)$-quasistable degree-$d$ pseudo-divisor on $\Gamma_X$, where $D$ is the divisor $\D$ seen as a divisor on $\Gamma_X^\E$. We call $(\E,D)$ the pseudo-divisor on $\Gamma_X$ induced by $\D$.
\end{CorDef}


In a sense, Corollary-Definition \ref{cor:quasiquasi} is the tropical analogue of Remark \ref{Rem:multidegree}.

Next we have a key result stating that quasistable divisors can be chosen as canonical representatives for equivalence classes of divisors on a tropical curve. 

\begin{Thm}[Theorem 5.6 in \cite{AAMPJac}]\label{Thm:div_qs_curve}
\label{thm:quasistable}
Let $X$ be a tropical curve with a point $p_0\in X$. Let $\mu$ be a degree-$d$ polarization on $X$. Given a divisor $\D$ on $X$ of degree $d$, there exists a unique degree-$d$ divisor equivalent to $\D$ which is $(p_0,\mu)$-quasistable.
\end{Thm}
\begin{Rem}
\label{rem:algo}
The proof of Theorem \ref{thm:quasistable} in \cite{AAMPJac} is constructive, so there is an algorithm to compute the $(p_0,\mu)$-quasistable divisor equivalente to $\D$.
\end{Rem}

Now we describe the tropical Jacobian of a tropical curve as a polyhedral complex by means of $(p_0,\mu)$-quasistable divisors. But, before, we need to introduce some notation.\par

A \tit{polyhedron} $\mc{P}\subset N_\R\cong\R^n$ is an intersection of a finite number of half-spaces of $N_\R$. A \emph{face} of a polyhedron $\mc{P}$ is the intersection of $\mc{P}$ and a hyperplane $H\subset N_{\mathbb R}$ such that $\mc{P}$ is contained in a closed half-space determined by $H$.
A morphism $f\col\mc{P}\ra\mc{P}'$ between polyhedra $\mc{P}\subset N_\R\cong\R^{n}$ and $\mc{P}'\subset N'_\R\cong\R^{m}$ is the restriction to $\mc{P}$ of an affine map $T\col N_\R\ra N'_\R$ such that $T(\mc{P})\subset\mc{P}'$. A morphism $f\col \mc{P}\ra\mc{P}'$ of polyhedra is a \tit{face morphism} if the image $f(\mc{P})$ is a face of $\mc{P}'$ and $f$ is an isometry. A \tit{polyhedral complex} $\Sigma$ is the colimit (as topological space) of a finite poset \tbf{D} of polyhedra with face morphisms (see \cite[section 2.2, page 4]{Poly} and \cite[section 3.2, page 13]{AAMPJac}). 
A \emph{morphism of polyedral complexes} $f\col \Sigma\ra \Sigma'$ is a continuous map of topological spaces $f\col\Sigma\to\Sigma'$ such that for every polyhedron $\mc P\in \mathbf{D}$ there exists a polyhedron $\mc P'\in \mathbf{D}'$ such that the induced map $\mc P\to\Sigma'$ factors through a morphism of polyhedra $\mc P\to\mc P'$.

Let $X$ be a tropical curve with length function $\ell$. Let $\mu$ be a degree-$d$ polarization on $X$.  Let $\Gamma_X$ be a model of $X$ such that $\supp(\mu)\subset V(\Gamma_X)$. Let $p_0\in X$ be a point of $X$. For each $(p_0,\mu)$-quasistable degree-$d$ pseudo-divisor $(\E,D)$ on $\Gamma_X$, we define the polyhedra
\begin{align*}
\P_{(\E,D)}:=&\prod_{e\in\E}e=\prod_{e\in\E}[0,\l(e)]\subset \R^{\E}\\
\P^\circ_{(\E,D)}:=&\prod_{e\in\E}e^\circ=\prod_{e\in\E}(0,\l(e))\subset \R^{\E},
\end{align*}
where $e^\circ$ denotes the interior of the edge $e$. 
Notice that if $\E=\emptyset$, then $\P_{\E,D}$ is just a point. If $(\E,D)\ra(\E',D')$ is a specialization of pseudo-divisors on the graph $\Gamma$, then we have an induced face morphism of polyhedra $f\col\P_{(\E',D')}\to\P_{(\E,D)}$. The polyhedron $\P^\circ_{(\E,D)}$ parametrizes $(p_0,\mu)$-quasistable divisors on $X$,  
whose induced pseudo-divisor on $\Gamma_X$ is $(\E,D)$.

\begin{Def}
\label{def:jtrop}
Let $(X,p_0)$ be a pointed tropical curve and $\mu$ be a degree-$d$ polarization on $X$. The \emph{Jacobian of $X$ with respect to $(p_0,\mu)$} is the polyhedral complex
\[
J^\trop_{p_0,\mu}(X)=\lim_{\longrightarrow}\P_{(\E,D)}=\coprod_{(\E,\D)} \P^\circ_{(\E,D)}
\]
where the limit and the union are taken over the poset $\mathcal{QD}_{p_0,\mu}(\Gamma_X)$. 
\end{Def}

\begin{Exa}
Let $X$ be the tropical curve associated to the graph with $2$ vertices and $3$ edges connecting them, where all edges have length $1$. 
 Let $\mu$ be the degree-$0$ polarization on $X$ given by $\mu(p)=0$ for every $p\in X$ and assume that $p_0$ is the leftmost vertex. We draw the polyhedral complex $J^\trop_{p_0,\mu}(X)$ in Figure \ref{fig:jac}.
\begin{figure}[h!]
\begin{tikzpicture}[scale=3]
\draw[ultra thick] (0,0) rectangle (1,1);
\draw[ultra thick] (0,0) -- (0,1) -- (-1,0) -- (-1,-1) -- (0,0);
\draw[ultra thick] (0,0) -- (1,0) -- (0,-1) -- (-1,-1) -- (0,0);
\draw[fill] (0,0) circle [radius=0.03];
\draw[fill] (0,1) circle [radius=0.03];
\draw[fill] (1,0) circle [radius=0.03];
\draw[fill] (1,1) circle [radius=0.03];
\draw[fill] (0,-1) circle [radius=0.03];
\draw[fill] (-1,0) circle [radius=0.03];
\draw[fill] (-1,-1) circle [radius=0.03];
\begin{scope}[shift={(0.3,0.5)},scale=0.5]
\draw (0,0) to [out=45, in=135] (1,0);
\draw (0,0) to (1,0);
\draw (0,0) to [out=-45, in=-135] (1,0);
\draw[fill] (0,0) circle [radius=0.02];
\draw[fill] (1,0) circle [radius=0.02];
\draw[fill] (0.5,0.21) circle [radius=0.02];
\draw[fill] (0.5,0) circle [radius=0.02];
\node[left] at (0,0) {1};
\node[right] at (1,0) {1};
\node[above] at (0.5,0.21) {-1};
\node[below] at (0.5,-0.15) {-1};
\end{scope}
\begin{scope}[shift={(-0.7,0)},scale=0.5]
\draw (0,0) to [out=45, in=135] (1,0);
\draw (0,0) to (1,0);
\draw (0,0) to [out=-45, in=-135] (1,0);
\draw[fill] (0,0) circle [radius=0.02];
\draw[fill] (1,0) circle [radius=0.02];
\draw[fill] (0.5,0.21) circle [radius=0.02];
\draw[fill] (0.5,-0.21) circle [radius=0.02];
\node[left] at (0,0) {1};
\node[right] at (1,0) {1};
\node[above] at (0.5,0.21) {-1};
\node[below] at (0.5,-0.21) {-1};
\end{scope}
\begin{scope}[shift={(-0.25,-0.5)},scale=0.5]
\draw (0,0) to [out=45, in=135] (1,0);
\draw (0,0) to (1,0);
\draw (0,0) to [out=-45, in=-135] (1,0);
\draw[fill] (0,0) circle [radius=0.02];
\draw[fill] (1,0) circle [radius=0.02];
\draw[fill] (0.5,-0.21) circle [radius=0.02];
\draw[fill] (0.5,0) circle [radius=0.02];
\node[left] at (0,0) {1};
\node[right] at (1,0) {1};
\node[below] at (0.5,-0.21) {-1};
\node[above] at (0.5,0.15) {-1};
\end{scope}
\begin{scope}[shift={(1.15,0)},scale=0.2]
\draw (0,0) to [out=45, in=135] (1,0);
\draw (0,0) to (1,0);
\draw (0,0) to [out=-45, in=-135] (1,0);
\draw[fill] (0,0) circle [radius=0.02];
\draw[fill] (1,0) circle [radius=0.02];
\node[left] at (0,0) {0};
\node[right] at (1,0) {0};
\end{scope}
\begin{scope}[shift={(-0.1,1.1)},scale=0.2]
\draw (0,0) to [out=45, in=135] (1,0);
\draw (0,0) to (1,0);
\draw (0,0) to [out=-45, in=-135] (1,0);
\draw[fill] (0,0) circle [radius=0.02];
\draw[fill] (1,0) circle [radius=0.02];
\node[left] at (0,0) {0};
\node[right] at (1,0) {0};
\end{scope}
\begin{scope}[shift={(-1.1,-1.1)},scale=0.2]
\draw (0,0) to [out=45, in=135] (1,0);
\draw (0,0) to (1,0);
\draw (0,0) to [out=-45, in=-135] (1,0);
\draw[fill] (0,0) circle [radius=0.02];
\draw[fill] (1,0) circle [radius=0.02];
\node[left] at (0,0) {0};
\node[right] at (1,0) {0};
\end{scope}
\begin{scope}[shift={(0.9,1.1)},scale=0.2]
\draw (0,0) to [out=45, in=135] (1,0);
\draw (0,0) to (1,0);
\draw (0,0) to [out=-45, in=-135] (1,0);
\draw[fill] (0,0) circle [radius=0.02];
\draw[fill] (1,0) circle [radius=0.02];
\node[left] at (0,0) {1};
\node[right] at (1,0) {-1};
\end{scope}
\begin{scope}[shift={(-0.1,-1.1)},scale=0.2]
\draw (0,0) to [out=45, in=135] (1,0);
\draw (0,0) to (1,0);
\draw (0,0) to [out=-45, in=-135] (1,0);
\draw[fill] (0,0) circle [radius=0.02];
\draw[fill] (1,0) circle [radius=0.02];
\node[left] at (0,0) {1};
\node[right] at (1,0) {-1};
\end{scope}
\begin{scope}[shift={(-1.4,0)},scale=0.2]
\draw (0,0) to [out=45, in=135] (1,0);
\draw (0,0) to (1,0);
\draw (0,0) to [out=-45, in=-135] (1,0);
\draw[fill] (0,0) circle [radius=0.02];
\draw[fill] (1,0) circle [radius=0.02];
\node[left] at (0,0) {1};
\node[right] at (1,0) {-1};
\end{scope}
\begin{scope}[shift={(0.13,0.1)},scale=0.2]
\draw (0,0) to [out=45, in=135] (1,0);
\draw (0,0) to (1,0);
\draw (0,0) to [out=-45, in=-135] (1,0);
\draw[fill] (0,0) circle [radius=0.02];
\draw[fill] (1,0) circle [radius=0.02];
\node[left] at (0.12,0) {-1};
\node[right] at (1,0) {1};
\end{scope}
\begin{scope}[shift={(0.6,-0.6)},scale=0.3]
\draw (0,0) to [out=45, in=135] (1,0);
\draw (0,0) to (1,0);
\draw (0,0) to [out=-45, in=-135] (1,0);
\draw[fill] (0,0) circle [radius=0.02];
\draw[fill] (1,0) circle [radius=0.02];
\draw[fill] (0.5,-0.21) circle [radius=0.02];
\node[left] at (0,0) {1};
\node[right] at (1,0) {0};
\node[below] at (0.5,-0.21) {-1};
\end{scope}
\begin{scope}[shift={(-0.83,0.6)},scale=0.3]
\draw (0,0) to [out=45, in=135] (1,0);
\draw (0,0) to (1,0);
\draw (0,0) to [out=-45, in=-135] (1,0);
\draw[fill] (0,0) circle [radius=0.02];
\draw[fill] (1,0) circle [radius=0.02];
\draw[fill] (0.5,-0.21) circle [radius=0.02];
\node[left] at (0,0) {1};
\node[right] at (0.95,0) {0};
\node[below] at (0.5,-0.21) {-1};
\end{scope}
\begin{scope}[shift={(1.14,0.5)},scale=0.3]
\draw (0,0) to [out=45, in=135] (1,0);
\draw (0,0) to (1,0);
\draw (0,0) to [out=-45, in=-135] (1,0);
\draw[fill] (0,0) circle [radius=0.02];
\draw[fill] (1,0) circle [radius=0.02];
\draw[fill] (0.5,0.21) circle [radius=0.02];
\node[left] at (0,0) {1};
\node[right] at (1,0) {0};
\node[above] at (0.5,0.21) {-1};
\end{scope}
\begin{scope}[shift={(-1.45,-0.5)},scale=0.3]
\draw (0,0) to [out=45, in=135] (1,0);
\draw (0,0) to (1,0);
\draw (0,0) to [out=-45, in=-135] (1,0);
\draw[fill] (0,0) circle [radius=0.02];
\draw[fill] (1,0) circle [radius=0.02];
\draw[fill] (0.5,0.21) circle [radius=0.02];
\node[left] at (0,0) {1};
\node[right] at (1,0) {0};
\node[above] at (0.5,0.21) {-1};
\end{scope}
\begin{scope}[shift={(0.35,1.1)},scale=0.3]
\draw (0,0) to [out=45, in=135] (1,0);
\draw (0,0) to (1,0);
\draw (0,0) to [out=-45, in=-135] (1,0);
\draw[fill] (0,0) circle [radius=0.02];
\draw[fill] (1,0) circle [radius=0.02];
\draw[fill] (0.5,0) circle [radius=0.02];
\node[left] at (0,0) {1};
\node[right] at (1,0) {0};
\node[above] at (0.5,0.1) {\small{-1}};
\end{scope}
\begin{scope}[shift={(-0.65,-1.1)},scale=0.3]
\draw (0,0) to [out=45, in=135] (1,0);
\draw (0,0) to (1,0);
\draw (0,0) to [out=-45, in=-135] (1,0);
\draw[fill] (0,0) circle [radius=0.02];
\draw[fill] (1,0) circle [radius=0.02];
\draw[fill] (0.5,0) circle [radius=0.02];
\node[left] at (0,0) {1};
\node[right] at (1,0) {0};
\node[above] at (0.5,-0.62) {\small{-1}};
\end{scope}
\end{tikzpicture}
\caption{The Jacobian $J_{p_0,\mu}^{\trop}(X)$.}
\label{fig:jac}
\end{figure}
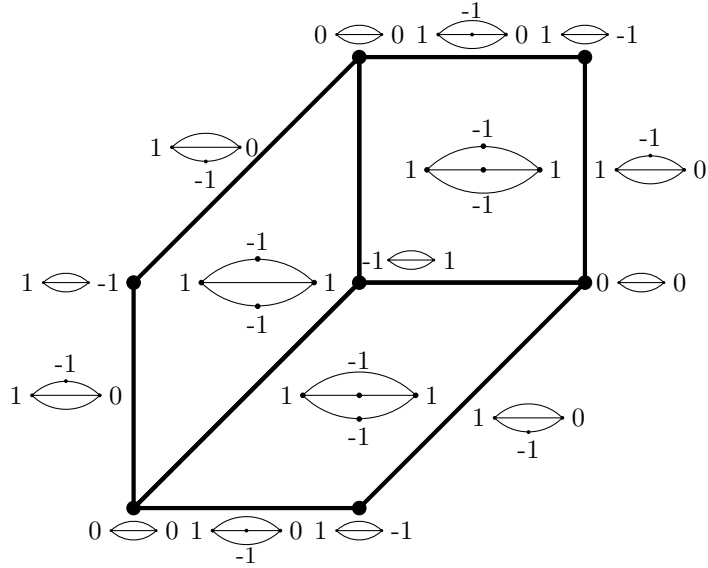
\end{Exa}

The important result relating the Jacobian of $X$ with respect to $(p_0,\mu)$ and the usual tropical Jacobian is given by the following result.

\begin{Thm}[Theorem 5.10 in \cite{AAMPJac}]
\label{thm:jX}
Let $X$ be a tropical curve with a point $p_0\in X$. Let $\mu$ be a degree-$d$ polarization on $X$. We have that $J^\trop_{p_0,\mu}(X)$ 
is homeomorphic to $J^\trop(X)$.
\end{Thm}

\section{The tropical Abel map}

In this section we prove one of our main results, Theorem \ref{Thm:4.2_versão_tropical}. 
This result tells us that the local criterion presented in Section \ref{Chap:Abel maps for nodal curves}  is always satisfied, provided a combinatorial condition involving the tropical Abel map is satisfied. This condition is quite explicit: the tropical Abel map should take the simplexes of a triangulation of the $d$-th product of a tropical curve to a cell of the polyhedral decomposition of the tropical Jacobian. 
An important technical tool to attain this result is given by the rearrangements of divisors on a graph, giving rise to what we call the \emph{organized version of a divisor}. We dedicate part of the section to the study of the properties of the organized divisors.
\par 
Throughout, we will use the following construction. 
   Let $\G$ be a graph and define $\ell\col E(\G)\to \mathbb{R}$ by $\ell(e)=1$ for every $e\in E(\G)$. We denote by $X_{\G}$ the tropical curve induced by the metric graph $(\G,\ell)$. If $\C\to B$ is a smoothing of a nodal curve $C$ with dual graph $\G$, then $X_{\G}$ is what is usually called the ``tropical curve associated to the smoothing". Notice, however, that many results in this section hold for more general length functions.

\subsection{Subdivisions and rearrangements}\label{sec:sub-rear}

Let $\Gamma$ be a graph. 
We can see $\G$ as a  metric graph whose edges have length $1$. We will often denote an oriented edge of $\Gamma$ by $e=\overrightarrow{v_0 v_1}$, where $v_0$ and $v_1$ are, respectively, the source and the target of $e$. 
For every oriented edge $e=\overrightarrow{v_0 v_1}$ of $\G$ and real number $r$ such that $0\le r\le 1$, we let $p_{e,r}$ be the point on $e$ at distance $r$ from the vertex $v_0$. Notice that $p_{e,0}=v_0$ and $p_{e,1}=v_1$. 
When no confusion may arise we will simply write $p_r$ instead of $p_{e,r}.$

\begin{Def}\label{def:convextuple}
Let $n$ be a positive integer and $l$ be a real number. A \emph{$l$-convex $n$-tuple} is a tuple 
 $\underline{t}=(t_1,\ldots, t_n)\in\R^n_{\ge0}$ such that $\sum_{1\le j\le n} t_j\le l$. When $l=1$, we simply call $\ul t$ a \emph{convex $n$-tuple}.
 \end{Def}

Let $\G$ be a graph and $n$ a positive integer. 
From now on, we fix an orientation on the graph $\Gamma$.
Given a $l$-convex $n$-tuple $\underline{t}$, we let $t_{i_1},\dots,t_{i_m}$ be the nonzero entries of $\ul t$, with $1\le i_1<\dots<i_m\le n$. We define a length function $\ell_{\underline{t}}\col E\big(\ora{\G^{(m+1)}}\big)\ra\R_{>0}$ such that, for each edge $e\in E(\Gamma)$,
\begin{equation}
\label{eq:lt}
    \ell_{\underline{t}}(e_{j})=t_{i_j}\ \mbox{ and }\ \ell_{\underline{t}}(e_{m+1})=l-\sum_{1\le j\le m}t_{i_j}.
\end{equation}
where $\{e_1,\dots,e_{m+1}\}=G^{-1}(e)$ (recall equation \eqref{eq:FG}). 
Thus we obtain a metric graph  $\Gamma^{(\ul t)}=(\G^{(m+1)},\ell_{\ul t})$. When $l=1$, the metric graphs $\G$ and $\Gamma^{(\ul t)}$ are models of the same tropical curve. 

 Let  $n$ be a positive integer and $\ul t$ be a convex $n$-tuple.  We let  
 \begin{equation}\label{eq:r_j}
     t_0:=0,\; t_{n+1}=1-\sum_{1\leq j\leq n} t_j\; \text{ and }\;r_j=\sum_{i=1}^{j} t_i.
 \end{equation}
 As usual, we let $p_{e,r_j}$ be the point on $X_\Gamma$ lying on an edge $e$ of $\G$ at distance $r_j$ from the source $s(e)$ of $e$. These points are precisely the vertices of $V(\G^{\ul t})$ seen as points on $X$. See Figure \ref{Fig:casa} for an example.
\begin{figure}[h!]
\begin{center}\begin{tikzpicture}[scale=5.5]
\begin{scope}[shift={(0,0)}]
\draw (0,0) to [out=-30, in=-150] (0.75,0);
\draw (0,0) to (0.375,0.15);
\draw (0.375,0.15) to (0.75,0);
\draw[->, line  width=0.3mm] (0.37,-0.11) to (0.375,-0.11);
\draw[->, line  width=0.3mm] (0.1825,0.075) to (0.1875,0.0756);
\draw[<-, line  width=0.3mm] (0.5625,0.0755) to (0.5675,0.075);

\draw[fill] (0,0) circle [radius=0.015];
\draw[fill] (0.75,0) circle [radius=0.015];
\draw[fill] (0.375,0.15) circle [radius=0.015];
\node[left] at (0,0) {$v_1$};
\node[above] at (0.375,0.15) {$v_2$};
\node[right] at (0.75,0) {$v_3$};

\node[above] at (0,0.1) {$\ora{\G}$};
\node[below] at (0.12,0.16) {$e_1$};
\node[below] at (0.63,0.16) {$e_2$};
\node[below] at (0.375,-0.12) {$e_3$};
\end{scope}

\begin{scope}[shift={(1.2,0)}]
\draw (0,0) to [out=-30, in=-150] (0.75,0);
\draw (0,0) to (0.375,0.15);
\draw (0.375,0.15) to (0.75,0);

\draw[fill] (0,0) circle [radius=0.015];
\draw[fill] (0.75,0) circle [radius=0.015];
\draw[fill] (0.375,0.15) circle [radius=0.015];
\node[left] at (0,0) {$v_1$};
\node[above] at (0.375,0.15) {$v_2$};
\node[right] at (0.75,0) {$v_3$};

\draw[fill] (0.25,-0.1) circle [radius=0.015];
\draw[fill] (0.5,-0.1) circle [radius=0.015];
\draw[fill] (0.125,0.05) circle [radius=0.015];
\draw[fill] (0.25,0.1) circle [radius=0.015];
\draw[fill] (0.625,0.05) circle [radius=0.015];
\draw[fill] (0.5,0.1) circle [radius=0.015];

\node[above] at (-0.1,0.1) {$\G^{(\ult)}$};
\node[below] at (0.03,0.14) {$p_{e_1,r_1}$};
\node[below] at (0.2,0.22) {$p_{e_1,r_2}$};
\node[below] at (0.75,0.14) {$p_{e_2,r_1}$};
\node[below] at (0.6,0.22) {$p_{e_2,r_2}$};
\node[below] at (0.25,-0.1) {$p_{e_3,r_1}$};
\node[below] at (0.5,-0.1) {$p_{e_3,r_2}$};
\end{scope}
\end{tikzpicture}\end{center}
\caption{For $n=2$, the points $p_{e,r_i}$ on $\G^{(\ul t)}=(\G^{(3)},\ell_{\ul t})$.}\label{Fig:casa}
\end{figure}
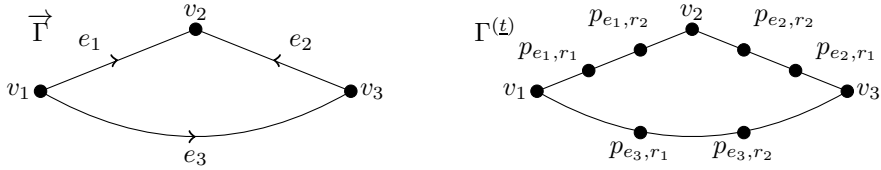

\begin{Def}
Let $a$ be an integer. We say that
a sequence of integers $(a_j)_{j=0}^k$ is \emph{$a$-admissible} if $\sum_{j=0}^k a_j=a$, with $a_j\in \{0,a,-a\}$, and the inequality
\[
|a_s+a_{s+1}+\dots+a_t|\le |a|
\]
holds, 
for every $0\le s\le t\le k$. We call \tit{admissible} every sequence that is $(-1)$-admissible.
\end{Def}

\begin{Rem}
A sequence of integers is a $a$-admissible if and only if the following properties hold:
\begin{enumerate}
    \item The values $a$ and $-a$ alternate each other in the sequence obtained by deleting the terms equal to $0$;
    \item The first and last nonzero elements are equal to $a$.  
\end{enumerate}
\end{Rem}

\begin{Def}\label{def:dI}
 Let $e=\ora{v_{0}v_{1}}$ be an oriented edge with length $l$. Let $\ul t=(t_1,\dots,t_n)$ be a $l$-convex $n$-tuple.  Set $t_0$, $t_{n+1}$ and $r_j$ as in equation \eqref{eq:r_j}. Consider a subset $I\subset \{-1,1,\dots,n\}$ and  define:
 \[
 d_I=\begin{cases}
\sum_{j\in I}t_{j} & \mbox{if }-1\notin I;\\
l-\sum_{j \in I\setminus\{-1\}}t_{j} & \mbox{if }-1\in I.
\end{cases}
\]
Moreover, for an integer $a$, we let $\D^{a,I}_{\ul t}=a p_{e,d_I}$ be the divisor on $e$ supported on $p_{e,d_I}$  of degree $a$. Finally, for a sequence of integers $\ul a=(a_0,a_1,\ldots,a_{n+1})$, we define the divisor on the metric edge $e$:
\[\widetilde{\D}_{\ult}^{\ul a}= \sum_{0\le j\le n+1}a_j p_{e, r_j}.
\]
If the sequence $\ul a$ is $a$-admissible for some integer $a$, we say that the divisor 
$\widetilde{\D}_{\ult}^{\ul a}$ is \emph{$a$-admissible}; if $\ul a$ is admissible, we simply say that $\widetilde{\D}_{\ult}^{\ul a}$ is \emph{admissible}.
\end{Def}

\begin{Rem}
Notice that for an index $k\in\{1,\dots,n\}$ we can consider the divisors $\D_{\ult}^{a,I}$ and $\widetilde{\D}^{\ul a}_{\ult}$ with $\ul t=(0,0,\ldots,0)$ or $\ult=(0,\ldots,l,\ldots,0)$, where $l$ sits in the $k$-th entry.
\end{Rem}

\begin{Exa}
Let $e=\ora{v_0v_1}$ be an oriented metric edge of length $1$. Let $\ult=(t_1,t_2,t_3)$ be a convex $3$-tuple and fix $a\in \mathbb Z$.
In Figure \ref{Fig:exa_inicial} we illustrate  examples of a divisor $\D_{\ult}^{a,I}$ for $I=\{2\}$ and $\{-1,1,3\}$. In Figure 
\ref{Fig:exa_inicial23} we specialize the previous examples to the cases $\ult=(0,0,0),$ $(1,0,0)$, $(0,1,0)$. 
In Figure \ref{Fig:exa_inicial4} we illustrate 
a divisor of type $\widetilde{\D}_{\ult}^{\ul a}$ and its specializations to the  cases $\ult=(0,0,0),$ $(1,0,0)$, $(0,1,0)$.

\begin{figure}[h!]
\begin{center}\begin{tikzpicture}[scale=5]

\begin{scope}[shift={(-0.5,-0.25)}]
\draw (0,0) to (0.5,0);
\draw[fill] (0,0) circle [radius=0.015];
\draw[fill] (0.5,0) circle [radius=0.015];
\draw[fill] (0.35,0.02) rectangle (0.36,-0.02);
\node[above] at (0.355,0.05) {$a$};
\node[below] at (0,0) {$v_0$};
\node[below] at (0.5,0) {$v_1$};
\node[below] at (0.175,0) {$t_2$};
\end{scope}


\begin{scope}[shift={(0.5,-0.25)}]
\draw (0,0) to (0.5,0);
\draw[fill] (0,0) circle [radius=0.015];
\draw[fill] (0.5,0) circle [radius=0.015];
\draw[fill] (0.18,0.02) rectangle (0.19,-0.02);
\node[above] at (0.18,0.05) {$a$};
\node[below] at (0,0) {$v_0$};
\node[below] at (0.5,0) {$v_1$};
\node[below] at (0.32,0) {$t_1+t_3$};
\end{scope}
\end{tikzpicture}
\end{center}
\caption{Examples of the divisor $\D_{\ult}^{a,I}$.}
\label{Fig:exa_inicial}
\end{figure}
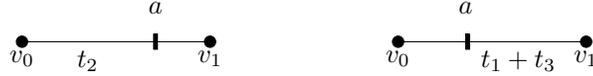

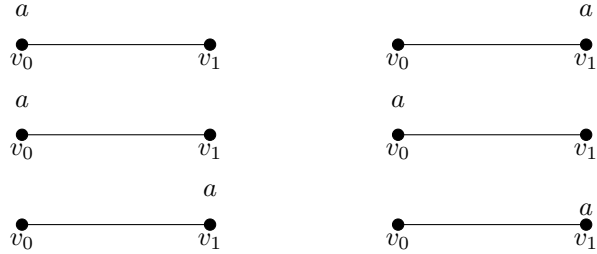
\begin{figure}[h!]
\begin{center}
\begin{tikzpicture}[scale=5]
\begin{scope}[shift={(-0.5,0.3)}]
\draw (0,0) to (0.5,0);
\draw[fill] (0,0) circle [radius=0.015];
\draw[fill] (0.5,0) circle [radius=0.015];
\node[above] at (0,0.05) {$a$};
\node[below] at (0,0) {$v_0$};
\node[below] at (0.5,0) {$v_1$};
\end{scope}
\begin{scope}[shift={(0.5,0.3)}]
\draw (0,0) to (0.5,0);
\draw[fill] (0,0) circle [radius=0.015];
\draw[fill] (0.5,0) circle [radius=0.015];
\node[above] at (0.5,0.05) {$a$};
\node[below] at (0,0) {$v_0$};
\node[below] at (0.5,0) {$v_1$};
\end{scope}
\end{tikzpicture}\end{center}

\begin{center}\begin{tikzpicture}[scale=5]
\begin{scope}[shift={(-0.5,0)}]
\draw (0,0) to (0.5,0);
\draw[fill] (0,0) circle [radius=0.015];
\draw[fill] (0.5,0) circle [radius=0.015];
\node[above] at (0,0.05) {$a$};
\node[below] at (0,0) {$v_0$};
\node[below] at (0.5,0) {$v_1$};
\end{scope}
\begin{scope}[shift={(0.5,0)}]
\draw (0,0) to (0.5,0);
\draw[fill] (0,0) circle [radius=0.015];
\draw[fill] (0.5,0) circle [radius=0.015];
\node[above] at (0,0.05) {$a$};
\node[below] at (0,0) {$v_0$};
\node[below] at (0.5,0) {$v_1$};
\end{scope}
\end{tikzpicture}\end{center}
\begin{center}\begin{tikzpicture}[scale=5]
\begin{scope}[shift={(-0.5,-0.3)}]
\draw (0,0) to (0.5,0);
\draw[fill] (0,0) circle [radius=0.015];
\draw[fill] (0.5,0) circle [radius=0.015];
\node[above] at (0.5,0.05) {$a$};
\node[below] at (0,0) {$v_0$};
\node[below] at (0.5,0) {$v_1$};
\end{scope}
\begin{scope}[shift={(0.5,-0.3)}]
\draw (0,0) to (0.5,0);
\draw[fill] (0,0) circle [radius=0.015];
\draw[fill] (0.5,0) circle [radius=0.015];
\node[above] at (0.5,0) {$a$};
\node[below] at (0,0) {$v_0$};
\node[below] at (0.5,0) {$v_1$}; 
\end{scope}
\end{tikzpicture}\end{center}
\caption{$\D_{\ult}^{a,I}$ in the special cases  $\ul t=(0,0,0)$, $\ul t=(1,0,0)$ and $\ult=(0,1,0)$.}\label{Fig:exa_inicial23}
\end{figure}


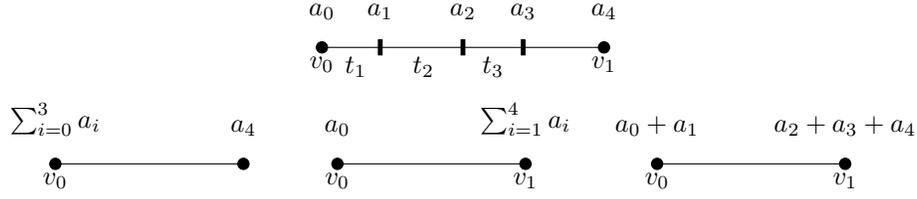
\begin{figure}[h!]
\begin{center}\begin{tikzpicture}[scale=5]
\begin{scope}[shift={(0,0)}]
\draw (0,0) to (0.75,0);
\draw[fill] (0,0) circle [radius=0.015];
\draw[fill] (0.75,0) circle [radius=0.015];
\draw[fill] (0.15,0.02) rectangle (0.16,-0.02);
\draw[fill] (0.37,0.02) rectangle (0.38,-0.02);
\draw[fill] (0.53,0.02) rectangle (0.54,-0.02);
\node[above] at (0,0.05) {$a_0$};
\node[above] at (0.155,0.05) {$a_1$};
\node[above] at (0.375,0.05) {$a_2$};
\node[above] at (0.535,0.05) {$a_3$};
\node[above] at (0.75,0.05) {$a_4$};
\node[below] at (0,0) {$v_0$};
\node[below] at (0.75,0) {$v_1$};
\node[below] at (0.09,0) {$t_1$};
\node[below] at (0.27,0) {$t_2$};
\node[below] at (0.455,0) {$t_3$};
\end{scope}
\end{tikzpicture}\end{center}

\begin{center}\begin{tikzpicture}[scale=5]
\begin{scope}[shift={(-1,0)}]
\draw (0,0) to (0.5,0);
\draw[fill] (0,0) circle [radius=0.015];
\draw[fill] (0.5,0) circle [radius=0.015];
\node[above] at (0,0.05) {$\sum_{i=0}^{3} a_i$};
\node[above] at (0.5,0.05) {$a_4$};
\node[below] at (0,0) {$v_0$};
\end{scope}

\begin{scope}[shift={(-0.25,0)}]
\draw (0,0) to (0.5,0);
\draw[fill] (0,0) circle [radius=0.015];
\draw[fill] (0.5,0) circle [radius=0.015];
\node[above] at (0,0.05) {$a_0$};
\node[above] at (0.5,0.05) {$\sum_{i=1}^{4} a_i$};
\node[below] at (0,0) {$v_0$};
\node[below] at (0.5,0) {$v_1$};
\end{scope}

\begin{scope}[shift={(0.6,0)}]
\draw (0,0) to (0.5,0);
\draw[fill] (0,0) circle [radius=0.015];
\draw[fill] (0.5,0) circle [radius=0.015];
\node[above] at (0,0.05) {$a_0+a_1$};
\node[above] at (0.5,0.05) {$a_2+a_3+a_4$};
\node[below] at (0,0) {$v_0$};
\node[below] at (0.5,0) {$v_1$};
\end{scope}
\end{tikzpicture}\end{center}
\caption{$\widetilde{\D}_{\ult}^{\ul a}$ for  $n=3$ and in the special cases $\ult=(0,0,0),$ $(1,0,0)$, $(0,1,0)$.}\label{Fig:exa_inicial4}
\end{figure}
\end{Exa}

\begin{Prop}\label{prop:arrumar}
 Let $e=\ora{v_{0}v_{1}}$ be an oriented metric edge with length $l$. Let $\ul{t}=(t_1,\dots,t_n)$ be an $l$-convex $n$-tuple. 
Consider a divisor on $e$ of type $\D_{\ul t} ^{a,I}$, for some integer $a$ and some subset
$I\subset \{-1,1,\ldots, n\}$. 
Then there is an $a$-admissible sequence  $\ul a=(a_0,\ldots, a_{n+1})$ and a rational function $f$ on $e$ such that 
\begin{enumerate}
\item[(1)] $\D_{\ul t} ^{a,I}=
\widetilde{\D}_{\ult} ^{\ul a} + \div(f)$;
\item[(2)] $f(v_{0})=f(v_{1})$;
 \item[(3)]
 if either $\ult=(0,\ldots,0)$, or $\ult=(0,\ldots,l,\ldots,0)$ (where $t_k=l$ for some $1\le k\le n$), then 
 $\D_{(t_1,\ldots,t_n)}^{a,I}=
\widetilde{\D}_{(t_1,\ldots,t_n)}^{\ul a}$, and they are divisors supported on $v_0$ and $v_1$.
\end{enumerate}
 \end{Prop}
 
\begin{figure}[htp]
\begin{center}\begin{tikzpicture}[scale=5]
\begin{scope}[shift={(0,0)}]
\draw (0,0) to (0.75,0);
\draw[fill] (0,0) circle [radius=0.015];
\draw[fill] (0.75,0) circle [radius=0.015];
\draw[fill] (0.45,0.02) rectangle (0.46,-0.02);

\node[above] at (0.45,0.05) {$a$};
\node[below] at (0,0) {$v_0$};
\node[below] at (0.75,0) {$v_1$};
\node[below] at (0.225,0) {$d$};
\end{scope}
\begin{scope}[shift={(1,0)}]
\draw (0,0) to (0.375,0);
\draw[dashed] (0.375,0) to (0.535,0);
\draw (0.535,0) to (0.75,0);
\draw[fill] (0,0) circle [radius=0.015];
\draw[fill] (0.75,0) circle [radius=0.015];
\draw[fill] (0.15,0.02) rectangle (0.16,-0.02);
\draw[fill] (0.37,0.02) rectangle (0.38,-0.02);
\draw[fill] (0.53,0.02) rectangle (0.54,-0.02);
\node[above] at (0,0.05) {$a_0$};
\node[above] at (0.155,0.05) {$a_1$};
\node[above] at (0.375,0.05) {$a_2$};
\node[above] at (0.535,0.05) {$a_n$};
\node[above] at (0.75,0.05) {$a_{n+1}$};
\node[below] at (0,0) {$v_0$};
\node[below] at (0.75,0) {$v_1$};
\node[below] at (0.09,0) {$t_1$};
\node[below] at (0.27,0) {$t_2$};
\end{scope}
\end{tikzpicture}\end{center}
\caption{The divisor $\D_{\ul t} ^{a,I}$ and its equivalent divisor $\widetilde{\D}_{\ult} ^{\ul{a}}$.}
\end{figure}

\begin{proof} 
 For $I=\emptyset$ we have nothing to do: in this case $\ul{a}=(a,0,\dots,0)$ and $f$ is constant, and items (1), (2), (3) clearly hold. 
 
 We now prove the proposition for $I\subset \{1,\ldots, n\}$ nonempty. We proceed by induction on $n$. For $n=1$ and $I=\{1\}$, the divisor $\D_{\ul t}^{a,I}=a p_{e,r_{1}}$ has already the desired form: we can take $\ul{a}=(0,a,0)$ and $f$ constant.  It is straightforward to check that items (1), (2), (3) hold. We argue similarly for $n=2$ and $I=\{1\}, \{1,2\}$.

 For $n=2$ and $I=\{2\}$, we can find a rational function $f$ on $e$ such that $\D_{\ul t} ^{a,I}-\div(f)=a p_{e,t_{2}}-\div(f)$ is of the desired form, as illustrated in  Figure \ref{Fig:div_Dt2} (when $t_2\geq t_1$, the other case being analogous). In particular we have that $\ul a=(a,-a,a,0)$. Notice that items (1) and (2) holds.
 \begin{figure}[h]\begin{center}
\begin{tikzpicture}[scale=5]
\begin{scope}[shift={(0,0)}]
\draw (0,0) to (0.75,0);
\draw[fill] (0,0) circle [radius=0.015];
\node[left] at (0,0) {$v_0$};
\draw[fill] (0.75,0) circle [radius=0.015];
\node[right] at (0.75,0) {$v_1$};
\node[left] at (-0.375,0) {$\D_{\ul t} ^{a,\{2\}}$};
\draw[fill] (0.4,0) circle [radius=0.015];
\node[above] at (0.4,0.03) {$a$};
\node[below] at (0.2,0) {$t_2$};

\draw (0,-0.4) to (0.15,-0.2);
\draw (0.15,-0.2) to (0.4,-0.2);
\draw (0.4,-0.2) to (0.55,-0.4);
\draw (0.55,-0.4) to (0.75,-0.4);
\draw[fill] (0,0) circle [radius=0.015];
\node[left] at (-0.38,-0.3) {$f$};
\node[left] at (0.05,-0.3) {$1$};
\node[right] at (0.43,-0.3) {$-1$};
\node[above] at (0.27,-0.2) {$0$};
\node[above] at (0.65,-0.4) {$0$};

\draw[dotted] (0,-0.4) to (0,-0.6);
\draw[dotted] (0.15,-0.2) to (0.15,-0.6);
\draw[dotted] (0.4,-0.2) to (0.4,-0.6);
\draw[dotted] (0.55,-0.4) to (0.55,-0.6);
\draw[dotted] (0.75,-0.4) to (0.75,-0.6);

\draw (0,-0.6) to (0.75,-0.6);
\draw[fill] (0,-0.6) circle [radius=0.015];
\node[above] at (0,-0.57) {$-a$};
\node[left] at (0,-0.6) {$v_0$};
\draw[fill] (0.75,-0.6) circle [radius=0.015];
\node[above] at (0.75,-0.57) {$0$};
\node[right] at (0.75,-0.6) {$v_1$};
\node[left] at (-0.375,-0.6) {$\div(f)$};
\draw[fill] (0.55,-0.6) circle [radius=0.015];
\node[above] at (0.55,-0.57) {$-a$};
\draw[fill] (0.15,-0.6) circle [radius=0.015];
\node[above] at (0.15,-0.57) {$a$};
\draw[fill] (0.4,-0.6) circle [radius=0.015];
\node[above] at (0.4,-0.57) {$a$};

\node[below] at (0.075,-0.6) {\small{$t_1$}};
\node[below] at (0.27,-0.6) {\small{$t_2-t_1$}};
\node[below] at (0.47,-0.6) {\small{$t_1$}};

\draw (0,-0.9) to (0.75,-0.9);
\draw[fill] (0,-0.9) circle [radius=0.015];
\node[above] at (0,-0.87) {$a$};
\node[left] at (0,-0.9) {$v_0$};
\draw[fill] (0.75,-0.9) circle [radius=0.015];
\node[right] at (0.75,-0.9) {$v_1$};
\node[left] at (-0.375,-0.9) {$\widetilde{\D}_{\ul t} ^{(a,-a,a,0)}$};
\draw[fill] (0.55,-0.9) circle [radius=0.015];
\draw[fill] (0.15,-0.9) circle [radius=0.015];
\node[above] at (0.15,-0.87) {$-a$};
\node[above] at (0.55,-0.87) {$a$};
\node[below] at (0.075,-0.9) {\small{$t_1$}};
\node[below] at (0.35,-0.9) {\small{$t_2$}};
\end{scope}
\end{tikzpicture}\end{center}
\caption{Finding $\widetilde{\D}_{\ult}^{(a_0,a_1,a_2,a_3)}$ for $n=2$ and $I=\{2\}$.}\label{Fig:div_Dt2}
\end{figure}
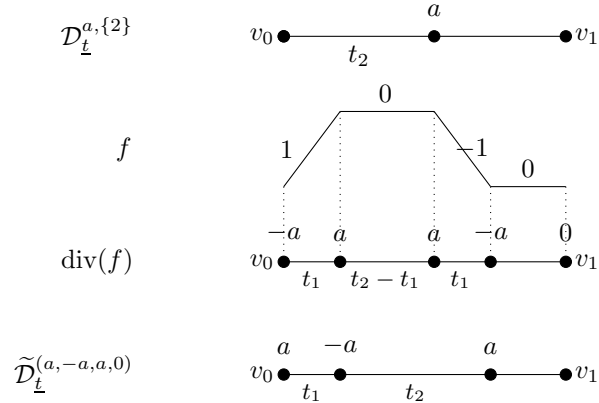

The divisors $\D_{(0,0)} ^{a,\{2\}}$, $\D_{(l,0)} ^{a,\{2\}}$, $\D_{(0,l)} ^{a,\{2\}}$ are equal to the divisors $\widetilde{\D}_{(0,0)}^{(a,-a,a,0)}$, $\widetilde{\D}_{(l,0)}^{(a,-a,a,0)}$, $\widetilde{\D}_{(0,l)}^{(a,-a,a,0)}$, respectively, as illustrated in Figure \ref{Fig:splz_Dt2}, hence item (3) holds.

\begin{figure}[h]
\begin{center}\begin{tikzpicture}[scale=4.8]
\begin{scope}[shift={(0,0)}]
\draw (0,0) to (0.5,0);
\draw (0.55,0) to (0.6,0.15);
\draw[->] (0.6,0.15) to (0.9,0.15);
\node[above] at (0.8,0.135) {\small{$\ult=(0,0)$}};
\draw[->] (0.55,0) to (0.9,0);
\node[above] at (0.8,-0.005) {\small{$\ult=(l,0)$}};
\draw (0.55,0) to (0.6,-0.15);
\draw[->] (0.6,-0.15) to (0.9,-0.15);
\node[above] at (0.8,-0.135) {\small{$\ult=(0,l)$}};

\draw (1.05,0.15) to (1.55,0.15);
\draw (1.05,0) to (1.55,0);
\draw (1.05,-0.15) to (1.55,-0.15);
\draw (2,0.15) to (2.05,0);
\draw[<-] (1.65,0.15) to (2,0.15);
\node[above] at (1.75,0.135) {\small{$\ult=(0,0)$}};
\draw[<-] (1.65,0) to (2.05,0);
\node[above] at (1.75,-0.005) {\small{$\ult=(l,0)$}};
\draw (2,-0.15) to (2.05,0);
\draw[<-] (1.65,-0.15) to (2,-0.15);
\node[above] at (1.75,-0.135) {\small{$\ult=(0,l)$}};

\draw (2.1,0) to (2.6,0);

\draw[fill] (0,0) circle [radius=0.015];
\draw[fill] (0.5,0) circle [radius=0.015];
\draw[fill] (1.05,0.15) circle [radius=0.015];
\draw[fill] (1.55,0.15) circle [radius=0.015];
\draw[fill] (1.05,0) circle [radius=0.015];
\draw[fill] (1.55,0) circle [radius=0.015];
\draw[fill] (1.05,-0.15) circle [radius=0.015];
\draw[fill] (1.55,-0.15) circle [radius=0.015];
\draw[fill] (2.1,0) circle [radius=0.015];
\draw[fill] (2.6,0) circle [radius=0.015];

\draw[fill] (0.15,0.02) rectangle (0.16,-0.02);
\draw[fill] (2.3,0.02) rectangle (2.31,-0.02);
\draw[fill] (2.45,0.02) rectangle (2.46,-0.02);

\node[above] at (0.15,0.05) {$a$};
\node[above] at (1.05,0.15) {$a$};
\node[above] at (1.05,0) {$a$};
\node[above] at (1.55,-0.15) {$a$};

\node[above] at (2.1,0.05) {$a$};
\node[above] at (2.3,0.05) {$-a$};
\node[above] at (2.45,0.05) {$a$};

\node[below] at (0.075,0) {\small{$t_2$}};
\node[below] at (2.2,0) {\small{$t_1$}};
\node[below] at (2.375,0) {\small{$t_2$}};

\end{scope}
\end{tikzpicture}\end{center}
\caption{$\D_{\ult}^{a,\{2\}}$ and  $\widetilde{\D}_{\ult}^{\{a,-a,a,0\}}$.}\label{Fig:splz_Dt2}
\end{figure}

 Next, consider an $l$-convex $n$-tuple $\ult=(t_1,\dots,t_n)$ and a divisor of  type
  $\D_{\ul t} ^{a,I}$, for $a\in\mathbb Z$ and for some 
$I\subset \{1,\ldots, n\}$. Recall that $\D_{\ul t}^{a,I}=a p_{e,d_I}$ with $d_I=\sum_{j\in I} t_j$. 
Assume that $1\in I$ and write the divisor as $\D_{\ul t}^{a,I}=0p_{e,r_1}+a p_{e,d_I}$. Consider the oriented edge $e_1=\ora{p_{e,r_1}v_{1}}$  and the divisor 
 \[\F_{\ul s}^{a,I\setminus\{1\}}=\D_{\ul t} ^{a,I}\bigg|_{e_1}=ap_{e_1,d_I-t_1}=ap_{e_1,d_{I\setminus\{1\}}}\] 
 where $\ul s=(t_2,\ldots,t_n)$ and $d_{I\setminus\{1\}}=\sum_{j\in I\setminus\{1\}}t_j$. By the inductive hypothesis,  the result holds for  $(t_2,\dots,t_n)$ and the divisor $\F_{\ul s}^{a,I\setminus\{1\}}$, giving rise to an $a$-admissible sequence $(a_1,\ldots, a_{n+1})$  and a function $f'$ on the edge $e_1$ such that $\F_{\ul s}^{a,I\setminus\{1\}}=\widetilde{\F}_{\ul s} ^{(a_1,\ldots,a_{n+1})}+\div(f')$. In particular the coefficient of $\widetilde{\F}_{\ul s} ^{(a_1,\ldots,a_{n+1})}$ on $p_{e,r_1}$ is  $a_1$, which is equal to either $0$ or $a$.
 We let $\ul a=(0,a_1,\ldots, a_{n+1})$ and $f$ be the function on $e$ defined as the extension of $f'$ which is constant over $\ora{v_0 p_{e,r_1}}$. Then $\widetilde{\D}_{\ul t}^{a}$ and $f$ are as required by the properties (1) and (2).   
 
 We check property (3). Let $\ul t=(0,\ldots, 0)$, or $\ul t=(0,\ldots, l,\ldots, 0)$. If  $t_1=0$, we have $e_1=e$, and we need only to check the identity $\F_{\ul s}^{a,I\setminus\{1\}}=\widetilde{\F}_{\ul s}^{(a_1,\ldots,a_{n+1})}$ for $\ul s=(0,\ldots, 0)$ and $\ul s=(0,\ldots, l,\ldots, 0)$, which holds by the induction hypothesis. If $t_1=l$, then $e_1=v_1$ and  $\D_{(l,\ldots,0)}^{a,I}=\widetilde{\D}_{(l,\ldots,0)}^{\ul a}=av_1$, as wanted.

 Assume now that $1\notin I$. We have two cases: $d_I-t_1>0$ and $d_I-t_1\leq0.$ In the first case, as illustrated in Figure \ref{Fig:div_D1},
 we can find a rational function $f_{1}$ on $e$ such that 
 \[\D_{\ul t} ^{a,I}-\div(f_{1})=av_0-ap_{e,r_1}+ap_{e,d_{I}+t_1}.\] 
In the second case, we can use the same rational function $f_{1}$: the only difference is that  $\div(f_1)=-av_0+ap_{e,d_I}+ap_{e,r_1}-ap_{e,t_1+d_I}.$ 
\begin{figure}[h]
\begin{center}
\begin{tikzpicture}[scale=5]
\begin{scope}[shift={(0,0)}]
\draw (0,0) to (0.75,0);
\draw[fill] (0,0) circle [radius=0.015];
\node[left] at (0,0) {$v_0$};
\draw[fill] (0.75,0) circle [radius=0.015];
\node[right] at (0.75,0) {$v_1$};
\node[left] at (-0.375,0) {$\D_{\ul t} ^{a,I}$};
\draw[fill] (0.4,0) circle [radius=0.015];
\node[above] at (0.4,0.03) {$a$};
\node[below] at (0.2,0) {\tiny{$d_I=\sum_{j\in I}t_j$}};

\draw (0,-0.4) to (0.15,-0.2);
\draw (0.15,-0.2) to (0.4,-0.2);
\draw (0.4,-0.2) to (0.55,-0.4);
\draw (0.55,-0.4) to (0.75,-0.4);
\draw[fill] (0,0) circle [radius=0.015];
\node[left] at (-0.38,-0.3) {$f_1$};
\node[left] at (0.05,-0.3) {$1$};
\node[right] at (0.43,-0.3) {$-1$};
\node[above] at (0.3,-0.2) {$0$};
\node[above] at (0.65,-0.4) {$0$};

\draw[dotted] (0,-0.4) to (0,-0.6);
\draw[dotted] (0.15,-0.2) to (0.15,-0.6);
\draw[dotted] (0.4,-0.2) to (0.4,-0.6);
\draw[dotted] (0.55,-0.4) to (0.55,-0.6);
\draw[dotted] (0.75,-0.4) to (0.75,-0.6);

\draw (0,-0.6) to (0.75,-0.6);
\draw[fill] (0,-0.6) circle [radius=0.015];
\node[above] at (0,-0.57) {$-a$};
\node[left] at (0,-0.6) {$v_0$};
\draw[fill] (0.75,-0.6) circle [radius=0.015];
\node[above] at (0.75,-0.57) {$0$};
\node[right] at (0.75,-0.6) {$v_1$};
\node[left] at (-0.375,-0.6) {$\div(f_1)$};
\draw[fill] (0.55,-0.6) circle [radius=0.015];
\node[above] at (0.55,-0.57) {$-a$};
\draw[fill] (0.15,-0.6) circle [radius=0.015];
\node[above] at (0.15,-0.57) {$a$};
\draw[fill] (0.4,-0.6) circle [radius=0.015];
\node[above] at (0.4,-0.57) {$a$};

\node[below] at (0.075,-0.6) {\small{$t_1$}};
\node[below] at (0.27,-0.6) {\small{$d_I-t_1$}};
\node[below] at (0.47,-0.6) {\small{$t_1$}};

\draw (0,-0.9) to (0.75,-0.9);
\draw[fill] (0,-0.9) circle [radius=0.015];
\node[above] at (0,-0.87) {$a$};
\node[left] at (0,-0.9) {$v_0$};
\draw[fill] (0.75,-0.9) circle [radius=0.015];
\node[right] at (0.75,-0.9) {$v_1$};
\node[left] at (-0.375,-0.9) {$\D_{\ul t} ^{a,I}-\div(f_{1})$};
\draw[fill] (0.55,-0.9) circle [radius=0.015];
\draw[fill] (0.15,-0.9) circle [radius=0.015];
\node[above] at (0.15,-0.87) {$-a$};
\node[above] at (0.55,-0.87) {$a$};
\node[below] at (0.075,-0.9) {\small{$t_1$}};
\node[below] at (0.35,-0.9) {\small{$d_I$}};
\end{scope}
\end{tikzpicture}\end{center}
\caption{Finding $f_1$ and $\D_{\ul t} ^{a,I}-\div(f_{1})$.}\label{Fig:div_D1}
\end{figure}
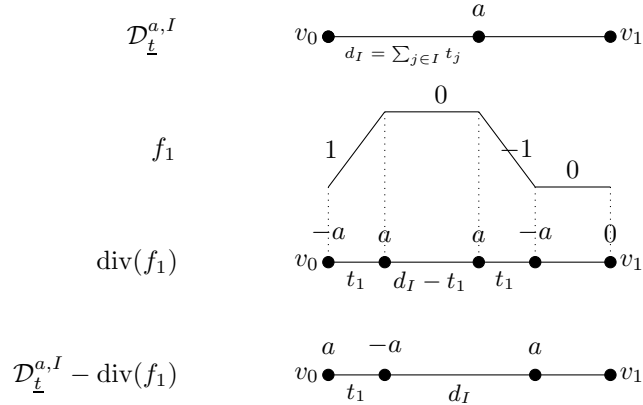

Consider the edge $e_1=\ora{p_{e,r_1}v_1}$ and the divisor $\F_{\ul s}^{a,I}=ap_{e_1,d_I}$, where $\ul s=(t_2,\dots.t_n)$ and $I$ is seen as a subset of $\{2,\dots,n\}$. Again, by the inductive hypothesis, the result holds for $(t_2,\dots,t_n)$ and the divisor $\F_{\ul s}^{a,I}$, giving rise to an admissible sequence 
$(a_1,\ldots,a_{n+1})$ 
and a function $f'$ on $e_1$ such that $\F_{\ul s}^{a,I}=\widetilde{\F}_{\ul s} ^{(a_1,\ldots,a_{n+1})}+\div(f')$. In particular, the coefficient of $\widetilde{\F}_{\ul s} ^{(a_1,\ldots,a_{n+1})}$ on $p_{e,r_1}$ is $a_1$, which is equal to either $0$ or $a$. Set $\ul a:=(a,a_1-a,a_2,\ldots, a_{n+1})$. Notice that $\ul a$ is $a$-admissible. 
Let $f_2$ be the function on $e$ defined as the extension of $f'$ which is constant over $\ora{v_0p_{e,r_1}}$. Then $\ul a$
and $f=f_1+f_2$ satisfy properties (1) and (2). 

Let us check property (3). Let $\ul t=(0,\ldots,0)$ or $\ul t=(0,\ldots, l,\ldots, 0)$. If  $t_1=0$, we have $e_1=e$, and we need only to check the identity $\F_{\ul s}^{a,I}=\widetilde{\F}_{\ul s}^{(a_1,\ldots,a_{n+1})}$ for $\ul s=(0,\ldots, 0)$ and $\ul s=(0,\ldots, l,\ldots, 0)$, which holds by the induction hypothesis. If $t_1=l$, we have
\[
\widetilde{\D}_{(l,\ldots,0)}^{ a}=av_0+(a_1-a+a_2\ldots+a_{n+1})v_1.
\]
Since $a_1-a+a_2+\ldots+a_{n+1}=0$, we have that   $\widetilde{\D}_{(l,\ldots,0)}^{\ul a}=av_0=\D_{(l,\ldots,0)}^{a,I}$, as required in (3).

As a final step, we prove the proposition for $-1\in I$. Recall that $d_I=l-\sum_{j\in I\setminus\{-1\}} t_j$. We have two cases: $2d_I<l$ and $2d_I\ge l$. In the first case, as illustrated in Figure \ref{Fig:finding_inverse}, we can find a rational function $f_1$ on $e$ such that
\[\D_{\ul t} ^{a,I}=av_0+av_1+\D_{\ul t}^{-a, I\setminus\{-1\}} +\div(f_1).\] 
In the second case, we can use the same rational function $f_1$: the only difference is that $\div(f_1)=-av_0+ap_{e,l-d_I}+ap_{e,d_I}-av_1$.
\begin{figure}[h]\begin{center}
\begin{tikzpicture}[scale=5]
\begin{scope}[shift={(0,0)}]
\draw (0,0) to (0.75,0);
\draw[fill] (0,0) circle [radius=0.015];
\node[left] at (0,0) {$v_0$};
\draw[fill] (0.75,0) circle [radius=0.015];
\node[right] at (0.75,0) {$v_1$};
\node[left] at (-0.375,0) {$\D_{\ul t} ^{a,I}$};
\draw[fill] (0.55,0) circle [radius=0.015];
\node[above] at (0.55,0.03) {$a$};
\node[below] at (0.65,0) {\small{$d_I$}};

\draw (0,-0.4) to (0.2,-0.2);
\draw (0.2,-0.2) to (0.55,-0.2);
\draw (0.55,-0.2) to (0.75,-0.4);
\draw[fill] (0,0) circle [radius=0.015];
\node[left] at (-0.38,-0.3) {$f_1$};
\node[left] at (0.1,-0.3) {$1$};
\node[right] at (0.65,-0.3) {$-1$};
\node[above] at (0.375,-0.2) {$0$};

\draw[dotted] (0,-0.4) to (0,-0.6);
\draw[dotted] (0.2,-0.2) to (0.2,-0.6);
\draw[dotted] (0.55,-0.2) to (0.55,-0.6);
\draw[dotted] (0.75,-0.4) to (0.75,-0.6);

\draw (0,-0.6) to (0.75,-0.6);
\draw[fill] (0,-0.6) circle [radius=0.015];
\node[above] at (0,-0.57) {$-a$};
\node[left] at (0,-0.6) {$v_0$};
\draw[fill] (0.75,-0.6) circle [radius=0.015];
\node[above] at (0.75,-0.57) {$-a$};
\node[right] at (0.75,-0.6) {$v_1$};
\node[left] at (-0.375,-0.6) {$\div(f_1)$};
\draw[fill] (0.55,-0.6) circle [radius=0.015];
\draw[fill] (0.2,-0.6) circle [radius=0.015];
\node[above] at (0.2,-0.57) {$a$};
\node[above] at (0.55,-0.57) {$a$};
\node[below] at (0.1,-0.6) {\small{$d_I$}};
\node[below] at (0.65,-0.6) {\small{$d_I$}};

\draw (0,-0.9) to (0.75,-0.9);
\draw[fill] (0,-0.9) circle [radius=0.015];
\node[above] at (0,-0.87) {$a$};
\node[left] at (0,-0.9) {$v_0$};
\draw[fill] (0.75,-0.9) circle [radius=0.015];
\node[above] at (0.75,-0.87) {$a$};
\node[right] at (0.75,-0.9) {$v_1$};
\node[left] at (-0.375,-0.9) {$a(v_0+v_1)+\D_{\ul t}^{-a, I\setminus\{-1\}}$};
\draw[fill] (0.55,-0.9) circle [radius=0.015];
\draw[fill] (0.2,-0.9) circle [radius=0.015];
\node[above] at (0.2,-0.87) {$-a$};
\node[above] at (0.55,-0.87) {$0$};
\node[below] at (0.1,-0.9) {\small{$d_I$}};
\node[below] at (0.65,-0.9) {\small{$d_I$}};
\end{scope}
\end{tikzpicture}\end{center}
\caption{Finding the rational function $f_1$.}\label{Fig:finding_inverse}
\end{figure}
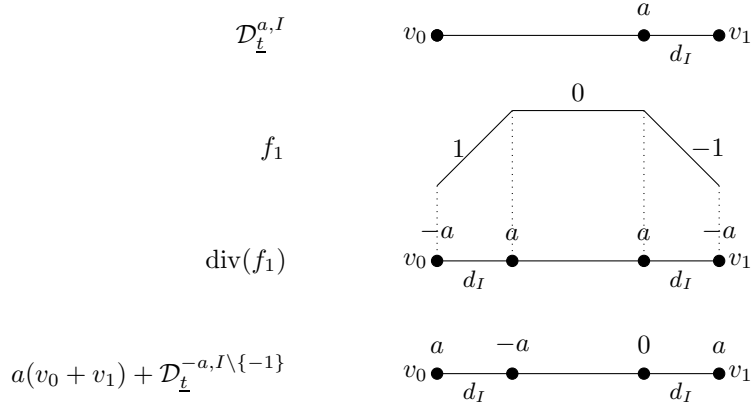

By the case in which $I\subset \{1,\ldots, n\}$, we know that there exists a $(-a)$-admissible sequence $\ul a'=(a_0,\ldots, a_{n+1})$ and a rational function $f_2$ on $e$ such that $\D_{\ul t}^{-a, I\setminus\{-1\}}=\widetilde{\D}_{\ul t}^{\ul a'}+\div(f_2)$.  
Set $\ul a:=(a+a_0,a_1,\ldots, a+a_{n+1})$. Notice that $\ul a$ is $a$-admissible. Define the function $f:=f_1+f_2$. The sequence $\ul a$ and the function $f$ satisfy item (1), because we have:
\begin{align*}
    \widetilde{\D}_{\ul t}^{\ul a}&=av_0+av_1+\widetilde{\D}_{\ul t}^{\ul a'}\\
    &=av_0+av_1+\D^{-a,I\setminus\{-1\}}_{\ul t}-\div(f_2)\\
    &=av_0+av_1-\div(f_2)+\D^{a,I}_{\ul t}-av_0-av_1-\div(f_1)\\
    &=\D^{a,I}_{\ul t}-\div(f_1+f_2).
\end{align*}
Notice that item (2) clearly holds as well.
Finally, for $\ul t=(0,\ldots,0)$ or $\ul t=(0,\ldots, l,\ldots, 0)$, we have  
\begin{equation}\label{eq:Dta}
\D_{\ul t} ^{a,I}= av_0+av_1+\D_{\ul t}^{-a, I\setminus\{-1\}}.
\end{equation} 
Using the case $I\subset\{1,\dots,n\}$, one can easily check  
 item (3). The proof is complete.
\end{proof}

\begin{Exa} Let $e$ be a metric edge with length 1. Let $\ult=(t_1,t_2,t_3)$ be a convex 3-tuple and consider the divisor $D=ap_{e,d}$. 
In Figure \ref{Fig:div_n3} we give some examples of   
the divisors obtained following the algorithm in the proof of Proposition \ref{prop:arrumar} for the cases $d=t_2$ and $t_1+t_3$. Notice that in the cases where $d=t_1, t_1+t_2, t_1+t_2+t_3$, we have nothing to do. 
In Figure \ref{Fig:div_n3_negativo}, we apply Proposition \ref{prop:arrumar} to the divisor in the cases where $-1\in I$.  
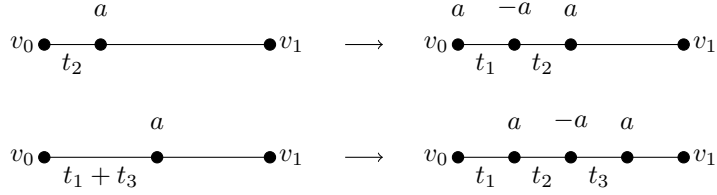
\begin{figure}[h!]
\begin{center}\begin{tikzpicture}[scale=5]
\begin{scope}[shift={(0,0)}]
\draw (0,0) to (0.6,0);
\draw[fill] (0,0) circle [radius=0.015];
\draw[fill] (0.15,0) circle [radius=0.015];
\draw[fill] (0.6,0) circle [radius=0.015];

\node[above] at (0.15,0.05) {$a$};
\node[left] at (0,0) {$v_0$};
\node[right] at (0.6,0) {$v_1$};
\node[below] at (0.075,0) {$t_2$};
\end{scope}

\begin{scope}[shift={(1.1,0)}]
\draw[->] (-0.3,0) to (-0.2,0);
\draw (0,0) to (0.6,0);
\draw[fill] (0,0) circle [radius=0.015];
\draw[fill] (0.15,0) circle [radius=0.015];
\draw[fill] (0.30,0) circle [radius=0.015];
\draw[fill] (0.6,0) circle [radius=0.015];

\node[above] at (0.0,0.05) {$a$};
\node[above] at (0.15,0.05) {$-a$};
\node[above] at (0.3,0.05) {$a$};
\node[left] at (0,0) {$v_0$};
\node[right] at (0.6,0) {$v_1$};
\node[below] at (0.075,0) {$t_1$};
\node[below] at (0.225,0) {$t_2$};
\end{scope}





\begin{scope}[shift={(0,-0.30)}]
\draw (0,0) to (0.6,0);
\draw[fill] (0,0) circle [radius=0.015];
\draw[fill] (0.30,0) circle [radius=0.015];
\draw[fill] (0.6,0) circle [radius=0.015];

\node[above] at (0.3,0.05) {$a$};
\node[left] at (0,0) {$v_0$};
\node[right] at (0.6,0) {$v_1$};
\node[below] at (0.15,0) {$t_1+t_3$};
\end{scope}

\begin{scope}[shift={(1.1,-0.30)}]
\draw[->] (-0.3,0) to (-0.2,0);
\draw (0,0) to (0.6,0);
\draw[fill] (0,0) circle [radius=0.015];
\draw[fill] (0.15,0) circle [radius=0.015];
\draw[fill] (0.30,0) circle [radius=0.015];
\draw[fill] (0.45,0) circle [radius=0.015];
\draw[fill] (0.6,0) circle [radius=0.015];

\node[above] at (0.15,0.05) {$a$};
\node[above] at (0.3,0.05) {$-a$};
\node[above] at (0.45,0.05) {$a$};
\node[left] at (0,0) {$v_0$};
\node[right] at (0.6,0) {$v_1$};
\node[below] at (0.075,0) {$t_1$};
\node[below] at (0.225,0) {$t_2$};
\node[below] at (0.375,0) {$t_3$};
\end{scope}




\end{tikzpicture}\end{center}
\caption{Using Proposition \ref{prop:arrumar}  for a convex 3-tuple.}\label{Fig:div_n3}
\end{figure}
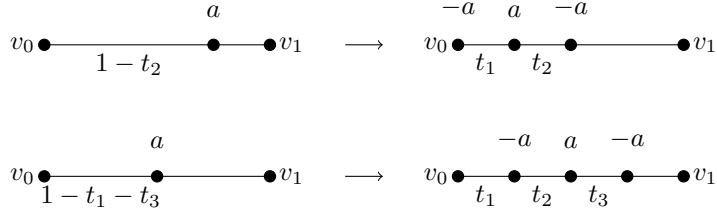
\begin{figure}[h!]
\begin{center}\begin{tikzpicture}[scale=5]





\begin{scope}[shift={(0,-0.35)}]
\draw (0,0) to (0.6,0);
\draw[fill] (0,0) circle [radius=0.015];
\draw[fill] (0.45,0) circle [radius=0.015];
\draw[fill] (0.6,0) circle [radius=0.015];

\node[above] at (0.45,0.05) {$a$};
\node[left] at (0,0) {$v_0$};
\node[right] at (0.6,0) {$v_1$};
\node[below] at (0.225,0) {$1-t_2$};
\end{scope}
\begin{scope}[shift={(1.1,-0.35)}]
\draw[->] (-0.3,0) to (-0.2,0);
\draw (0,0) to (0.6,0);
\draw[fill] (0,0) circle [radius=0.015];
\draw[fill] (0.15,0) circle [radius=0.015];
\draw[fill] (0.30,0) circle [radius=0.015];
\draw[fill] (0.6,0) circle [radius=0.015];

\node[above] at (0.0,0.05) {$-a$};
\node[above] at (0.15,0.05) {$a$};
\node[above] at (0.3,0.05) {$-a$};
\node[left] at (0,0) {$v_0$};
\node[right] at (0.6,0) {$v_1$};
\node[below] at (0.075,0) {$t_1$};
\node[below] at (0.225,0) {$t_2$};
\end{scope}

\begin{scope}[shift={(0,-0.7)}]
\draw (0,0) to (0.6,0);
\draw[fill] (0,0) circle [radius=0.015];
\draw[fill] (0.30,0) circle [radius=0.015];
\draw[fill] (0.6,0) circle [radius=0.015];

\node[above] at (0.3,0.05) {$a$};
\node[left] at (0,0) {$v_0$};
\node[right] at (0.6,0) {$v_1$};
\node[below] at (0.15,0) {$1-t_1-t_3$};
\end{scope}
\begin{scope}[shift={(1.1,-0.7)}]
\draw[->] (-0.3,0) to (-0.2,0);
\draw (0,0) to (0.6,0);
\draw[fill] (0,0) circle [radius=0.015];
\draw[fill] (0.15,0) circle [radius=0.015];
\draw[fill] (0.30,0) circle [radius=0.015];
\draw[fill] (0.45,0) circle [radius=0.015];
\draw[fill] (0.6,0) circle [radius=0.015];

\node[above] at (0.15,0.05) {$-a$};
\node[above] at (0.3,0.05) {$a$};
\node[above] at (0.45,0.05) {$-a$};
\node[left] at (0,0) {$v_0$};
\node[right] at (0.6,0) {$v_1$};
\node[below] at (0.075,0) {$t_1$};
\node[below] at (0.225,0) {$t_2$};
\node[below] at (0.375,0) {$t_3$};
\end{scope}

\end{tikzpicture}\end{center}
\caption{Using Proposition \ref{prop:arrumar}  for a convex 3-tuple when $-1\in I$.} \label{Fig:div_n3_negativo}
\end{figure}
\end{Exa}

\begin{Cons} \label{Cons:dta}
We can extend Proposition \ref{prop:arrumar} to tropical curves and more general divisors. For simplicity, we will just consider  tropical curves with edges of length 1. This will be enough for our purposes.
Let $\G$ be a graph and consider the tropical curve $X=X_\Gamma$. Let
\[
\mathfrak{a}\col E(\G)\times \P(\{-1,1,\ldots,n\})\to \mathbb{Z}
\]
be a function ($\P$ denotes the power set). Consider a  convex $n$-tuple $\ul t$. We define the divisor $\D^{\mathfrak{a}}_{\ul t}$ on $X$ as 
 \[
 \D^{\mathfrak{a}}_{\ul t}=\underset{I\subset\{-1,1,\ldots,n\}}{\sum_{e\in E(\G)}} \mathfrak{a}(e,I)p_{e,d_I}.
 \]
 We can apply Proposition \ref{prop:arrumar} to each $\mathfrak{a}(e,I)p_{e,d_I}$. We get an $\mathfrak{a}(e,I)$-admissible sequence $\ul a(e,I)$ and a rational function $f_{e,I}$ on $e$ such that $\mathfrak{a}(e,I)p_{e,d_I}=\wt{\D}^{\ul a(e,I)}_{\ul t}+\div(f_{e,I})$. Set 
 \[
\widetilde{\D}_{\ul t}^{\mathfrak{a}} :=\underset{I\subset\{-1,1,\ldots,n\}}{\sum_{e\in E(\G)}} \wt{\D}^{\ul a(e,I)}_{\ul t}
 \]
The divisors $\D^{\mathfrak{a}}_{\ul t}$ and $\widetilde{\D}_{\ul t}^{\mathfrak{a}}$ are equivalent. Indeed,  
since each $f_{e,I}$ takes the same value at the vertices incident to $e$, we can find a rational function $f$ on $X$ such that
 \[
 \D^{\mathfrak{a}}_{\ul t}=\widetilde{\D}_{\ul t}^{\mathfrak{a}}+\div(f).
 \]
 By Proposition \ref{prop:arrumar} (3) we get  $\widetilde{\D}_{\ul t}^{\mathfrak{a}}=\D^{\mathfrak{a}}_{\ul t}$
for $\ul t=(0,\ldots,0)$ and $\ul t=(0,\ldots, 1,\ldots, 0)$. We also have
\begin{equation}\label{eq:supp}
\supp(\widetilde{\D}_{\ul t}^{\mathfrak{a}})\subset \{p_{e,r_j} | e \in E(\G)\text{ and }j\in\{0,\ldots, n+1\}\},
\end{equation}
where $r_j=\sum_{1\le i\le j} t_i$.
We call $\widetilde{\D}_{\ul t}^{\mathfrak{a}}$ the \tit{organized version} of $\D_{\ul t}^{\mathfrak{a}}$. 
 \end{Cons}

\begin{Rem}\label{rem:admissible}
Let $(X,p_0)$ be a pointed tropical curve and $\mu$ be a polarization on $X.$
If $\D^{\mathfrak{a}}_{\ul t}$ is a $(p_0,\mu)$-quasistable divisor for every convex $n$-tuple $\ul t$, by Remark \ref{Rem:div_quasistable} and Proposition \ref{prop:quasiquasi}, we have that there exists a subset $\E\subset E(\Gamma)$ and a subset $I_e\subset \{-1,1,\ldots, n\}$ for each $e\in \E$ such that:
\begin{itemize}
    \item $I_e$ is different from $\emptyset$ and $\{-1\}$.
    \item $\mathfrak{a}(e,I_e)=-1$ for every $e\in \E$.
    \item $\mathfrak{a}(e,I)=0$ for every $I$ different from $\emptyset$, $\{-1\}$, and $I_e$.
\end{itemize}
Notice that $\supp(D^{\mathfrak{a}}_{\ul t})\cap e^\circ=\emptyset$ for every $e\notin \E$, and $\supp(D^{\mathfrak{a}}_{\ul t})\cap e^\circ$ consists of at most one point when $e\in \E$.\par

Applying Proposition \ref{prop:arrumar}, we see that the organized version $\widetilde{\D}_{\ul t}^{\mathfrak{a}}$ of a $(p_0,\mu)$-quasistable divisor is admissible. Notice that the organized version of a $(p_0,\mu)$-quasistable divisor $\D^{\mathfrak{a}}_{\ult}$ is nonzero at some point in the interior $e^\circ$ of every edge $e\in \E$. 
\end{Rem}
 
\begin{Exa}
Let $\G$ be a graph with two vertices $v_0$,$v_1$ and two edges $e_1=\ora{v_0v_1}$, $e_2=\ora{v_1v_0}$. Let $\ult=(t_1,t_2,t_3)$ be a convex $3$-tuple. Consider integers $a,b$. For $I_1=\{2\}$ and $I_2=\{-1,1,3\}$, we define $\mathfrak{a}(e_1,I_1)=a$, $\mathfrak{a}(e_2,I_2)=b$, $\mathfrak{a}(e_1,I_2)=\mathfrak{a}(e_2,I_1)=0$ and $\mathfrak{a}(e,I)=0$ for every $e\in E(\Gamma)$ and $I\ne I_1,I_2$. Then
\[
\D_{\ult}^{\mathfrak{a}}=
\mathfrak{a}(e_1,I_1)p_{e_1,d_{I_1}}+\mathfrak{a}(e_2,I_2)p_{e_2,d_{I_2}}=ap_{e_1,d_{I_1}}+bp_{e_2,d_{I_2}}.
\]
The divisors $\D_{\ult}^{\mathfrak{a}}$ and  $\wt{\D}_{\ult}^{\mathfrak{a}}$ are  illustrated in Figure \ref{Fig:exa_arrumargrafo}. The divisors $\widetilde{\D}_{\ul t}^{\mathfrak{a}}$ for   $\ult=(0,0,0),(1,0,0),(0,1,0)$, $(0,0,1)$ are illustrated in Figure \ref{Fig:exa_arrumargrafo2}.
\begin{figure}[h]
\begin{center}\begin{tikzpicture}[scale=7]
\begin{scope}[shift={(0,0)}]
\draw (0,0) to [out=30, in=150] (0.5,0);
\draw (0,0) to [out=-30, in=-150] (0.5,0);

\draw[fill] (0.125,0.055) circle [radius=0.01];
\draw[fill] (0.3,-0.069) circle [radius=0.01];
\draw[fill] (0,0) circle [radius=0.01];
\draw[fill] (0.5,0) circle [radius=0.01];

\node[left] at (-0.005,0) {$0$};
\node[right] at (0.505,0) {$0$};
\node[above] at (0.125,0.1) {$a$};
\node[below] at (0.3,-0.1) {$b$};
\node[above] at (0.04,0.02) {\small{$t_2$}};
\node[below] at (0.45,-0.05) {\small{$t_1+t_3$}};
\end{scope}
\begin{scope}[shift={(1,0)}]
\draw (0,0) to [out=30, in=150] (0.5,0);
\draw (0,0) to [out=-30, in=-150] (0.5,0);

\draw[fill] (0.125,0.055) circle [radius=0.01];
\draw[fill] (0.25,0.075) circle [radius=0.01];
\draw[fill] (0.25,-0.075) circle [radius=0.01];
\draw[fill] (0.125,-0.055) circle [radius=0.01];
\draw[fill] (0.375,-0.055) circle [radius=0.01];
\draw[fill] (0,0) circle [radius=0.01];
\draw[fill] (0.5,0) circle [radius=0.01];

\node[left] at (-0.005,0) {$a+b$};
\node[right] at (0.505,0) {$b$};
\node[above] at (0.11,0.09) {$-a$};
\node[above] at (0.04,0.03) {\small{$t_1$}};
\node[above] at (0.25,0.1) {$a$};
\node[below] at (0.25,-0.1) {$b$};
\node[below] at (0.2,-0.075) {\small{$t_2$}};
\node[above] at (0.2,0.065) {\small{$t_2$}};
\node[below] at (0.11,-0.09) {$-b$};
\node[below] at (0.04,-0.04) {\small{$t_1$}};
\node[below] at (0.375,-0.09) {$-b$};
\node[below] at (0.32,-0.06) {\small{$t_3$}};
\end{scope}
\end{tikzpicture}\end{center}
\caption{The divisor $\D_{\ult}^{\mathfrak{a}}$ and $\widetilde{\D}_{\ul t}^{\mathfrak{a}}.$} \label{Fig:exa_arrumargrafo}
\end{figure}
 
\begin{figure}[htp]
\begin{center}\begin{tikzpicture}[scale=4.75]
\begin{scope}[shift={(0,0)}]
\draw (0,0) to [out=30, in=150] (0.5,0);
\draw (0,0) to [out=-30, in=-150] (0.5,0);

\draw[fill] (0,0) circle [radius=0.015];
\draw[fill] (0.5,0) circle [radius=0.015];

\node[above] at (0.25,0.15) {$\ult=(0,0,0)$};
\node[left] at (0,0) {$a$};
\node[right] at (0.5,0) {$b$};
\end{scope}
\begin{scope}[shift={(0.9,0)}]
\draw (0,0) to [out=30, in=150] (0.5,0);
\draw (0,0) to [out=-30, in=-150] (0.5,0);

\draw[fill] (0,0) circle [radius=0.015];
\draw[fill] (0.5,0) circle [radius=0.015];

\node[above] at (0.25,0.15) {$\ult=(1,0,0),(0,0,1)$};
\node[below] at (0,-0.02) {$a+b$};
\node[right] at (0.5,0) {$0$};
\end{scope}
\begin{scope}[shift={(1.8,0)}]
\draw (0,0) to [out=30, in=150] (0.5,0);
\draw (0,0) to [out=-30, in=-150] (0.5,0);

\draw[fill] (0,0) circle [radius=0.015];
\draw[fill] (0.5,0) circle [radius=0.015];

\node[above] at (0.25,0.15) {$\ult=(0,1,0)$};
\node[left] at (0,0) {$0$};
\node[below] at (0.5,0) {$a+b$};
\end{scope}
\end{tikzpicture}\end{center}
\caption{The divisors $\widetilde{\D}_{\ul t}^{\mathfrak{a}}$.}\label{Fig:exa_arrumargrafo2}
\end{figure}
\end{Exa}

\subsection{Properties of organized divisors}\label{sec:prop-org}

In this section we will prove some results about organized divisors. We begin with the following result.

\begin{Prop}\label{Prop:beta_do_quasistavel} 
Let $\Gamma$ be a graph with a fixed vertex $v_0\in V(\Gamma)$ and $\mu$ be a polarization on $\Gamma$. Let $X=X_\Gamma$ be the associated tropical curve and $p_0$ the associated point. Let $\mathfrak{a}\col E(\G)\times \P(\{-1,1,\ldots,n\})\to \mathbb{Z}$ be a function such that
$\D^{\mathfrak{a}}_{\ul t}$ is a $(p_0,\mu)$-quasistable divisor on $X$ for every $\ult$. Consider the organized version  $\widetilde{\D}^{\mathfrak{a}}_{\ul t}$ of  $\D^{\mathfrak{a}}_{\ul t}$.
Then $\beta_{\widetilde{\D}^{\mathfrak{a}}_{\ul t}}(Y)\geqslant0$ for every subcurve $Y\subset X$ and the inequality is strict if $p_0\in Y$ and $V(\Gamma)\not\subset Y$.
\end{Prop}

\begin{proof}
Recall that, given a subcurve $Y$ of $X$, we have 
\[\beta_{\D^{\mathfrak{a}}_{\ul t}}(Y)=\deg(\D^{\mathfrak{a}}_{\ul t}|_{Y})-\mu(Y)+\frac{\delta_{Y}}{2}\geqslant0,
\]
with strict inequality if $p_{0}\in Y$. The divisor $\widetilde{\D}^{\mathfrak{a}}_{\ul t}$ is admissible by Remark \ref{rem:admissible} and its support is described in equation \eqref{eq:supp}.
It is enough to check the result when $Y$ is a connected curve, since the function $\beta_{\D^{\mathfrak{a}}_{\ul t}}$ is additive over connected components.

First, if $Y\cap V(\G)=\emptyset$, then $Y$ is a segment contained in the interior of some edge and in this case $\beta_{\wt{\D}^{\mathfrak{a}}_{\ult}}(Y)\ge 0$, since $\deg(\wt{\D}^{\mathfrak{a}}_{\ult}|_Y)\ge-1$ by the admissibility of $\wt{\D}^{\mathfrak{a}}_{\ult}$.

Now suppose that there is an edge $e\in E(\Gamma)$ such that $Y$ contains the set $\{s(e), t(e)\}$ of source and target of $e$, but does not contain $e$. Set $Z:=e\cap \overline{X\setminus Y}$ and $W=Y\cup Z$. 
Let us check that $\beta_{\widetilde{\D}^{\mathfrak{a}}_{\ul t}}(Y)\geq \beta_{\widetilde{\D}^{\mathfrak{a}}_{\ul t}}(W)$. Since $Y$ is connected we have that $Z=\overline{pq}$ for some $p,q\in e$ and the intersection $Y\cap Z$ consists of $p$ and $q$. We have $\mu(Z)=\mu(p)+\mu(q)$ because $\mu$ is zero in the interior of $e$ (recall that $\mu$ comes from a polarization on $\Gamma$). Thus, by Lemma \ref{lem:beta_curvas}, we have 
\begin{align*}
  \beta_{\widetilde{\D}^{\mathfrak{a}}_{\ul t}}(W)&= \beta_{\widetilde{\D}^{\mathfrak{a}}_{\ul t}}(Y)+\beta_{\widetilde{\D}^{\mathfrak{a}}_{\ul t}}(Z)-\beta_{\widetilde{\D}^{\mathfrak{a}}_{\ul t}}(Y\cap Z) \\
  &= \beta_{\widetilde{\D}^{\mathfrak{a}}_{\ul t}}(Y)+\beta_{\widetilde{\D}^{\mathfrak{a}}_{\ul t}}(Z)-\beta_{\widetilde{\D}^{\mathfrak{a}}_{\ul t}}(p)-\beta_{\widetilde{\D}^{\mathfrak{a}}_{\ul t}}(q).
\end{align*}
The expression  $\beta_{\widetilde{\D}^{\mathfrak{a}}_{\ul t}}(Z)-\beta_{\widetilde{\D}^{\mathfrak{a}}_{\ul t}}(p)-\beta_{\widetilde{\D}^{\mathfrak{a}}_{\ul t}}(q)$ is equal to
\begin{align*}
  =&\deg(\widetilde{\D}^{\mathfrak{a}}_{\ult}|_{Z})-\mu(Z) +\frac{\delta_{Z}}{2}-\left(\deg(\widetilde{\D}^{\mathfrak{a}}_{\ult}|_{p})-\mu(p)+\frac{\delta_{p}}{2}\right)\\
   &-\left(\deg(\widetilde{\D}^{\mathfrak{a}}_{\ult}|_{q})-\mu(q)+\frac{\delta_{q}}{2}\right)\\
  =&\deg(\widetilde{\D}^{\mathfrak{a}}_{\ult}|_{Z}) +1-\deg(\widetilde{\D}^{\mathfrak{a}}_{\ult}|_{p})-1
  -\deg(\widetilde{\D}^{\mathfrak{a}}_{\ult}|_{q})-1\\
   =&\deg(\widetilde{\D}^{\mathfrak{a}}_{\ult}|_{Z}) -\deg(\widetilde{\D}^{\mathfrak{a}}_{\ult}|_{p})-\deg(\widetilde{\D}^{\mathfrak{a}}_{\ult}|_{q})-1\\
   \le& 0,
\end{align*}
where in the last inequality we use that $\widetilde{\D}^{\mathfrak{a}}_{\ult}$ is admissible.
Then $\beta_{\widetilde{\D}^{\mathfrak{a}}_{\ul t}}(W)\leq \beta_{\widetilde{\D}^{\mathfrak{a}}_{\ul t}}(Y)$, as wanted.

Henceforth, we  will assume that if $Y$ contains the set $\{s(e),t(e)\}$, then it contains the edge $e$.  
Let $\E\subset E(\G)$ be the subset of edges such that $\deg(\D^{\mathfrak{a}}_{\ult}|_{e^\circ})=-1$, $e\cap Y\neq \emptyset$ and $e\not\subset Y$. 

Fix $e\in \E$ and assume that $s(e)\in Y$ (which means that $t(e)\notin Y$). Let $j\in\{1,\dots,n\}$ be the smallest index such that $\widetilde{\D}^{\mathfrak{a}}_{\ult}(p_{e,r_j})\neq0$. If $\widetilde{\D}^{\mathfrak{a}}_{\ult}(p_{e,r_j})=-1$, we set $p:=p_{e,r_j}$, otherwise we set $p:=s(e).$

If $p\not\in Y$, we define $Z^+_e:=\ol{p_{e,0}p}$ (notice that $p_{e,0}=s(e)\in Y$). If $p\in Y$, we define $Z^-_e:=Y\cap \ol{pp_{e,r_{n+1}}}$. If $t(e)\in Y$, we redo the same definitions choosing the maximum $j$ and switching $t(e)$ and $s(e)$. We let $Z^+$ be the union of all the subcurves $Z^+_e$ obtained in this way and $Z^-$ be the union of all the subcurves $Z^-_e$ obtained in this way. We define (see Figure \ref{Fig:Y_Ytil})
\[
\wt Y:=\ol{(Y\cup Z^+)\setminus Z^-}.
\]
Then we have
\begin{eqnarray}
\beta_{\widetilde{\D}^{\mathfrak{a}}_{\ult}}(\wt Y)&=& \deg(\widetilde{\D}^{\mathfrak{a}}_{\ult}|_{\wt Y})-\mu(\wt Y)+\frac{\delta_{\wt Y}}{2} \nonumber\\
&=& \deg(\widetilde{\D}^{\mathfrak{a}}_{\ult}|_{Y})+\deg(\widetilde{\D}^{\mathfrak{a}}_{\ult}|_{Z^+\setminus Y})-\deg(\widetilde{\D}^{\mathfrak{a}}_{\ult}|_{Z^-\setminus Y})-\mu(Y)+\frac{\delta_{Y}}{2} \nonumber\\
&\le&\beta_{\widetilde{\D}^{\mathfrak{a}}_{\ult}}(Y), \nonumber
\end{eqnarray}
where we use that $\deg(\widetilde{\D}^{\mathfrak{a}}_{\ult}|_{Z^+\setminus Y})\le 0$ and $\deg(\widetilde{\D}^{\mathfrak{a}}_{\ult}|_{Z^-\setminus Y})\ge0$.

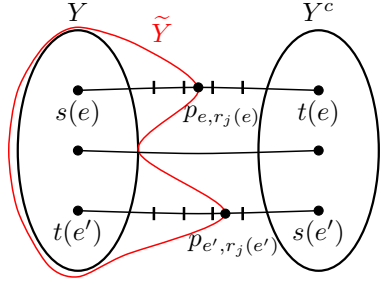
\begin{figure}[h]\begin{center}
\begin{tikzpicture}[scale=0.4]
\begin{scope}[shift={(1,0)}]
\draw[line  width=0.3mm] \boundellipse{4,1}{-2}{4};
\draw[line  width=0.3mm] \boundellipse{-4,1}{-2}{4};
\draw [red, line width=0.2mm] plot [smooth cycle, tension=0.5] coordinates {(-2.5,-2.5) (-4,-3.2) (-5,-2.8) (-6,-1) (-6.3,1) (-6,3) (-5.1,4.5) (-4,5.1) 
(-2,4.5) (0,3) (-2,1) 
(0.9,-1.1)
};

\draw [line width=0.2mm] plot [smooth, tension=0.5] coordinates {(-4,-1) (0,-1.1) (4,-1)};
\draw [line width=0.2mm] plot [smooth, tension=0.5] coordinates {(-4,1) (0,0.9) (4,1)};
\draw [black, line width=0.2mm] plot [smooth, tension=0.5] coordinates {(-4,3) (0,3.1) (4,3)};

\draw[fill] (-0.5,3.3) rectangle (-0.45,2.9);
\draw[fill] (-1.5,3.3) rectangle (-1.45,2.9);
\draw[fill] (0.5,3.3) rectangle (0.45,2.9);
\draw[fill] (1.5,3.3) rectangle (1.45,2.9);

\draw[fill] (-1.5,-1.3) rectangle (-1.45,-0.9);
\draw[fill] (-0.5,-1.3) rectangle (-0.45,-0.9);
\draw[fill] (0.5,-1.3) rectangle (0.45,-0.9);
\draw[fill] (1.5,-1.3) rectangle (1.45,-0.9);

\node[above] at (-1.2,4.3) {\textcolor{red}{$\widetilde{Y}$}};
\node[above] at (-4,5) {$Y$};
\node[above] at (4,5) {$Y^c$};
\node[below] at (-4,3) {$s(e)$};
\node[below] at (4,3) {$t(e)$};

\node[below] at (-4,-1) {$t(e')$};
\node[below] at (4,-1) {$s(e')$};

\node[below] at (0.8,2.8) {$p_{e,r_j(e)}$};
\draw[fill] (0,3.1) circle [radius=0.15];

\node[below] at (1.2,-1.5) {$p_{e',r_j(e')}$};
\draw[fill] (0.9,-1.1) circle [radius=0.15];

\draw[fill] (-4,-1) circle [radius=0.15];
\draw[fill] (-4,1) circle [radius=0.15];
\draw[fill] (-4,3) circle [radius=0.15];
\draw[fill] (4,-1) circle [radius=0.15];
\draw[fill] (4,1) circle [radius=0.15];
\draw[fill] (4,3) circle [radius=0.15];
\end{scope}
\end{tikzpicture}   
\end{center}
\caption{Curve $X$ and subcurves $Y$ and $\widetilde{Y}$.}\label{Fig:Y_Ytil}
\end{figure}

For every $e\in \E$
let $q_e$ be the unique point in $e^\circ$ such that $\D^{\mathfrak{a}}_{\ult}(q_e)=-1$. If $s(e)\in Y$, let $W_e:=\ol{p_{e,0}q_e}$, otherwise, if $t(e)\in Y$, let $W_e:=\ol{p_{e,r_{n+1}}q_e}$.
We define $\overline{Y}$ as the union of $Y$ and the segments $W_e$, for $e\in \E$ (see Figure \ref{Fig:Y_Ybarra}).

\begin{figure}[h!]\begin{center}
\begin{tikzpicture}[scale=0.4]
\begin{scope}[shift={(-1,0)}]
\end{scope}
\begin{scope}[shift={(1,0)}]
\draw[line  width=0.3mm] \boundellipse{4,1}{-2}{4};
\draw[line  width=0.3mm] \boundellipse{-4,1}{-2}{4};

\draw [red, line width=0.2mm] plot [smooth cycle, tension=0.5] coordinates {(-2.5,-2.5) (-4,-3.2) (-5,-2.8) (-6,-1) (-6.3,1) (-6,3) (-5.1,4.5) (-4,5.1) 
(-2,4.5) (0,3) (-2,1) 
(0.9,-1.1)};

\draw [line width=0.2mm] plot [smooth, tension=0.5] coordinates {(-4,-1) (0,-1.1) (4,-1)};
\draw [line width=0.2mm] plot [smooth, tension=0.5] coordinates {(-4,1) (0,0.9) (4,1)};
\draw [black, line width=0.2mm] plot [smooth, tension=0.5] coordinates {(-4,3) (0,3.1) (4,3)};

\node[above] at (-1.3,4.3) {\textcolor{red}{$\overline{Y}$}};
\node[above] at (-4,5) {$Y$};
\node[above] at (4,5) {$Y^c$};
\node[below] at (-4,3) {$s(e)$};
\node[below] at (4,3) {$t(e)$};

\node[below] at (-4,-1) {$t(e')$};
\node[below] at (4,-1) {$s(e')$};
\node[above] at (0,3.2) {$-1$};
\node[below] at (0,2.8) {$q_{e}$};
\draw[fill] (0,3.1) circle [radius=0.15];

\node[above] at (0.9,-1) {$-1$};
\node[below] at (0.9,-1.5) {$q_{e'}$};
\draw[fill] (0.9,-1.1) circle [radius=0.15];

\draw[fill] (-4,-1) circle [radius=0.15];
\draw[fill] (-4,1) circle [radius=0.15];
\draw[fill] (-4,3) circle [radius=0.15];
\draw[fill] (4,-1) circle [radius=0.15];
\draw[fill] (4,1) circle [radius=0.15];
\draw[fill] (4,3) circle [radius=0.15];
\end{scope}
\end{tikzpicture}   
\end{center}
\caption{Curve $X$ and subcurves $Y$ and $\overline{Y}$.}\label{Fig:Y_Ybarra}
\end{figure}

Since $\deg(\widetilde{\D}^{\mathfrak{a}}_{\ult}|_{\widetilde{Y}})=\deg(\D^{\mathfrak{a}}_{\ult}|_{\overline{Y}})$, $\mu(\widetilde{Y})=\mu(\overline{Y})$ and $\delta_{\widetilde{Y}}=\delta_{\overline{Y}}$, we  have
\[
\beta_{\widetilde{\D}^{\mathfrak{a}}_{\ult}}(Y)\geq \beta_{\widetilde{\D}^{\mathfrak{a}}_{\ult}}(\widetilde{Y})=\beta_{\D^{\mathfrak{a}}_{\ult}}(\overline{Y})\geq0,
\]
where the last inequality holds by the $(p_0,\mu)$-quasistability of $\D^{\mathfrak{a}}_{\ult}$, and it is strict if $p_0\in \ol Y$ and $\ol Y\ne X$ (which is true if $p_0\in Y$ and $V(\Gamma)\not\subset Y$).
\end{proof} 

Let $X$ be a tropical curve with underlying directed graph $\ora{\G}$ and with length function $\ell$ such that $\ell(e)=l$ for all $e\in E(\Gamma)$. Given a divisor $D=\sum_{v\in V(\Gamma^{(n+1)})} D(v)v$ on a subdivision $\Gamma^{(n+1)}$ of $\Gamma$, and an $l$-convex $n$-tuple $\ul t=(t_1,\ldots, t_n)$, we define the divisor $\D_{\ul t}\in \Div(X)$ as
\[
\D_{\ul t}:=\sum_{v\in V(\G)} D(v)v+\underset{j\in \{1,\dots, n\}}{\sum_{e\in E(\Gamma)}}D(x^e_j)p_{e,r_j},
\]
where, as usual, $\{x^e_1,\dots,x^e_n\}$ are the exceptional vertices over $e$, for every $e\in E(\Gamma)$. Conversely, given a divisor $\D_{\ul t}$ on $X$ such that 
\[
\supp \D_{\ul t}\subset V(\Gamma)\cup \left(\bigcup_{e\in E(\Gamma)}\{p_{e,r_j}; j=1,\dots,n\}\right)
\]
where $r_j=\sum_{1\le i\le j} t_i$,
we can write 
\[
\D_{\ul t}=\sum_{v\in V(\Gamma)} \D_{\ul t}(v)\cdot v+ \underset{j\in \{1,\dots,n\}}{\sum_{e\in E(\Gamma)}} \D_{\ul t}(p_{e,r_j})\cdot p_{e,r_j}.
\]
In this case we define the divisor $D$ on $\Gamma^{(n+1)}$ as
\begin{equation}\label{eq:Dt}
D:=\sum_{v\in V(\Gamma)} \D_{\ul t}(v)\cdot v+ \underset{j\in \{1,\dots,n\}}{\sum_{e\in E(\Gamma)}} \D_{\ul t}(p_{e,r_j})\cdot x^e_j.
\end{equation}

\begin{Prop}\label{prop:independece}
Let $\Gamma$ be a graph and consider the associated tropical curve $X=X_\Gamma$.
Let $D$ be a divisor on a subdivision $\G^{(n+1)}$ of $\Gamma$. For each convex $n$-tuple $\underline{t}$, consider the  divisor $\mc{D}_{\underline{t}}\in Div\big(X\big)$ associated to $D$. If $\mc{D}_{\ult}$ is principal for all choices of convex $n$-tuples $\ul t$ with nonzero entries, then there is a flow $\phi$ on $\G^{(n+1)}$ such that
\begin{enumerate}
    \item[(1)]
    $D+\div(\phi)=0$;
    \item[(2)]
    the flow $\phi$ induces the unique rational function $f_{\ul t}$ on $X$ (up to translations) such that $\mc{D}_{\ult}+ \div(f_{\ul t})=0$ for all $\ul t$.
\end{enumerate}
\end{Prop}

\begin{proof}
The condition that $\mc{D}_{\ult}$ is principal for a convex $n$-tuple $\ul t$ with nonzero entries implies the existence of a rational function $f_{\ul t}$ on $X$ with associated flow $\phi_{\ul t}$ on $\Gamma^{(n+1)}$ such that $D+\div(\phi_{\ult})=0$ and (recall equations  \eqref{eq:gamma},   \eqref{eq:functionflow}, and  \eqref{eq:lt})
\begin{equation}\label{eq:cyclerel}
\sum_{e\in E(\Gamma^{(n+1)})} \phi_{\ul t}(e)\ell_{\ul t}(e)\gamma(e)=0,
\end{equation}
for every cycle $\gamma\in H_1(\Gamma^{(n+1)},\mathbb Z)$.
Conversely, every flow $\phi_{\ult}$ satisfying equation \eqref{eq:cyclerel} induces a rational function on $X$. The proof consists in showing that $\phi_{\ul t}$ is independent of $\ul t$.

Recall the notation
$\{e_1,\dots,e_{m+1}\}=G^{-1}(e)$ in equation \eqref{eq:FG}. 
Given a cycle $\gamma\in H_1(\Gamma,\mathbb Z)\cong H_1(\G^{(n+1)},\mathbb{Z})$, the relation \eqref{eq:cyclerel} translates to
\begin{align*}
    0 & =\sum_{e\in\gamma}\left(\sum_{1 \le j \le n}\phi_{\ul t}(e_{j})\ell_{\ul t}(e_{j})\gamma(e_{j})\right)+ \sum_{e\in\gamma}\phi_{\ul t}(e_{n+1})\ell_{\ul t}(e_{n+1})\gamma(e_{n+1})\\
    & = \sum_{e\in\gamma}\left(\sum_{1\le j\le n}\phi_{\ul t}(e_{j})\gamma(e_{j})t_j\right)+ \sum_{e\in\gamma}\phi_{\ul t}(e_{n+1})\gamma(e_{n+1})\left(1-\sum_{j=1}^{n}t_{j}\right) \\
    & =\sum_{1\le j\le n}\left(\sum_{e\in\gamma}\left(\phi_{\ul t}(e_{j})\gamma(e_{j})-\phi_{\ul t}(e_{n+1})\gamma(e_{n+1})\right)\right)t_j+ \\
    &  \;\;\;\;\;\;\;\;\;\;\;\; +\sum_{e\in\gamma}\phi_{\ul t}(e_{n+1})\gamma(e_{n+1}).
\end{align*}

We can view the last expression as a linear combination of $t_1,\dots,t_n,1$. Since $\gamma(e_{j})=\pm1$ and $\phi_{\ul t}(e_{j})\in\mathbb{Z}$, the coefficients of the linear combination are integer numbers. Therefore, choosing $t_{1},\ldots,t_{n},1$ as linearly independent  over $\Q$, we have that these coefficients are all equal to zero. Calling $\phi$ a solution for this particular choice, we have that $\phi$ satisfies $D+\div(\phi)=0$  and satisfies equation \eqref{eq:cyclerel} for every $\ul t$. This finishes the proof.
\end{proof}

\subsection{Local conditions: tropical version}\label{sec:local-cond-trop}

In this section we give a tropical version of the combinatorial numbers $a$ and $b$ defined in equations \eqref{a_geometrico} and \eqref{b_geometrico}.

Let $\G$ be a graph. Let $d$ be a positive integer and $f\col\{1,\ldots,d\}\ra E(\G)$ be a function. We define 
\begin{equation}
    \label{eq:Vf}
    V_f
    =\prod_{1\le k\le d}\{s(f(k)),t(f(k))\}.
\end{equation}
(Recall that $s(e)$ and $t(e)$ are the source and target of an edge $e\in E(\G)$.)

For every $Q=(v_{1},\ldots,v_{d})\in V_f$ and $k=1,\dots,d$, we will denote by $\delta_k(Q)=v_{k}$ the $k$-th coordinate of $Q$, so that  $\delta_k(Q)$ is a vertex of $\Gamma$ incident to the edge $f(k)$.
 Given $Q=(v_{1},\ldots,v_{d})\in V_f$, the \emph{divisor associated to $Q$} is the divisor on $\Gamma$:
     \[\div(Q)=-\sum_{1\le k\le d} v_k.
                    \]
Moreover, given an edge $e\in E(\G)$, we let
\begin{equation}\label{eq:diveQ}
\div(e,Q)=-\sum_{\{k | f(k)=e\}}v_k.
\end{equation}
We can see $\div(e,Q)$ as the ``contribution to $\div(Q)$ along $e$".

\begin{Def}\label{Def:a_tropical} 
Let $\Gamma$ be a graph. 
Let $f\col \{1,\dots,d\}\ra E(\Gamma)$ be a function. For an edge $e\in E(\G)$, a vertex $v\in V(\G)$, and an element $Q\in V_f$, we define
\[a^e_Q(v):=\#\left\{k\in \{1,\ldots,d\} |\ f(k)=e \mbox{ and }\delta_k(Q)=v \right\}.
\]
For a subset $W\subset V(\G)$, we define
\begin{align*}
a_{Q}^e(W)&:= \#\left\{k\in \{1,\ldots,d\} | \ f(k)=e \mbox{ and } \delta_k (Q)\subset W^{c}\right\}\\
&= \sum_{v\in \{s(e),t(e)\}\setminus W} a^e_Q(v).
\end{align*}
\end{Def}

Hence, by definition, we have
\begin{equation}\label{eq:sete}
\div(e,Q)=-a^e_Q(s(e)) \cdot s(e) - a^e_Q(t(e))\cdot t(e).
\end{equation}

\begin{Exa}\label{Exa:a}
Consider the graph $\Gamma$ with two vertices $u,v$ and two edges $e_1,e_2$. Let $d=2$, $f(1)=e_1$ and $f(2)=e_2$. We can see $V_f$ as the set of vertices $Q_0,\dots, Q_3$ of a square, and we can write $Q_0=(u,u)$, $Q_1=(v,u)$, $Q_2=(u,v)$ and $Q_3=(v,v)$ (see Figure \ref{Fig:exa_a}).  
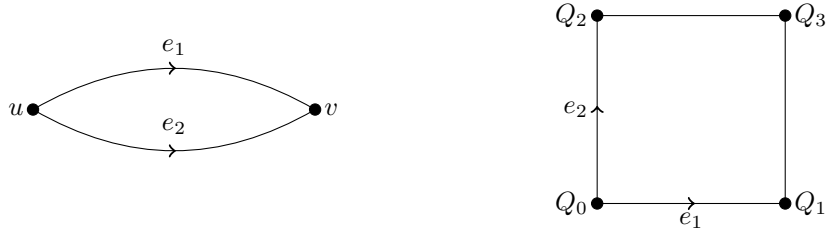
\begin{figure}[h]\begin{center}
\begin{tikzpicture}[scale=5]
\begin{scope}[shift={(0,0)}]
\draw (0,0) to [out=30, in=150] (0.75,0);
\draw (0,0) to [out=-30, in=-150] (0.75,0);
\draw[->, line  width=0.3mm]  (0.375,0.11) -- (0.38,0.11);
\draw[->, line  width=0.3mm]  (0.375,-0.11) -- (0.38,-0.11);

\draw[fill] (0,0) circle [radius=0.015];
\node[left] at (0,0) {$u$};
\draw[fill] (0.75,0) circle [radius=0.015];
\node[right] at (0.75,0) {$v$};
\node[above] at (0.375,0.12) {$e_1$};
\node[below] at (0.375,0) {$e_2$};
\end{scope}
\begin{scope}[shift={(1.5,0)}]
\draw (0,-0.25) to (0,0.25);
\draw (0,-0.25) to (0.5,-0.25);
\draw (0,0.25) to (0.5,0.25);
\draw (0.5,-0.25) to (0.5,0.25);
\draw[->, line  width=0.3mm]  (0.25,-0.25) -- (0.26,-0.25);
\draw[->, line  width=0.3mm]  (0,0) -- (0,0.01);

\draw[fill] (0,-0.25) circle [radius=0.015];
\node[left] at (0,-0.25) {$Q_0$};
\draw[fill] (0.5,-0.25) circle [radius=0.015];
\node[right] at (0.5,-0.25) {$Q_1$};
\draw[fill] (0,0.25) circle [radius=0.015];
\node[left] at (0,0.25) {$Q_2$};
\draw[fill] (0.5,0.25) circle [radius=0.015];
\node[right] at (0.5,0.25) {$Q_3$};
\node[left] at (0,0) {$e_2$};
\node[below] at (0.25,-0.25) {$e_1$};
\end{scope}
\end{tikzpicture}\end{center}
\caption{The graph $\G$ and the set $V_f$.}\label{Fig:exa_a}
\end{figure}

\noindent
An easy computation shows that: 
\[
\begin{tabular}{lllllll}
 $a_{Q_0}^{e_1}(u)=1$, && $a_{Q_0}^{e_1}(v)=0$, && $a_{Q_0}^{e_2}(u)=1$, && $a_{Q_0}^{e_2}(v)=0$.   
\end{tabular}\]
If we consider $W=\{u\}$, for 
$Q_0$ we have:
\[
\begin{tabular}{lll}
 $a_{Q_0}^{e_1}(W)=0$, && $a_{Q_0}^{e_2}(W)=0$ 
\end{tabular}\]
\end{Exa}

\begin{Not}
Let $v_0\in V(\G)$ be a vertex of $\Gamma$, and $\mu$ be a degree-$k$ polarization on $\G$.
For a degree-$d$ divisor $D$ on $\G$, we denote by $\qs(D)$ the unique $(v_0,\mu)$-quasistable divisor on the graph $\G$ which is equivalent to $D$ (recall Theorem \ref{Thm:div_qs_curve}).
\end{Not}

 Thus, given a divisor $D$ on $\Gamma$, there is $M_D\in C_0(\G,\mathbb Z)$ such that $D-\qs(D)=\partial \delta(M_D)$. We write:
 \[
D-\qs(D)=\partial \delta(M_D)=\sum_{v\in V(\G)}m_v v,
 \]
 where $m_v\in \mathbb Z$.
 We denote by $\phi_D$ the flow $\delta(M_D)\in C_1(\ora{\G},\mathbb Z)$ on $\G$:
  \[
 \phi_D:=\delta(M_D)=\sum_{e\in E(\G)}(m_{t(e)}-m_{s(e)})e.
 \] 
Thus, we have $D-\div(\phi_D)=\qs(D)$.

\begin{Rem}\label{rem:determine}
Notice that the condition $D-\qs(D)=\partial \delta(M_D)$ only determines $M_D$ up to a multiple of the divisor $\sum_{v\in V(\Gamma)} v$. Nevertheless, the flow $\phi_D$ is uniquely determined.
\end{Rem}

\begin{Def}\label{Def:b_tropical}
Let $\G$ be a graph.
Let $f\col \{1,\dots,d\}\ra E(\G)$ be a function.  
For an edge $e\in E(\G)$, an element $Q\in V_f$, and a divisor $D^{\dagger}$ on $\G$, we define 
\[
b^e_Q(D^{\dagger}):=\phi_{D^{\dagger}+\div(Q)}(e)=m_{t(e)}-m_{s(e)}.\]
For a subset $W\subset V(\G)$, we define
\[b_{Q}^{e}(W,D^{\dagger}):=\left\{ 
\begin{array}{cl}
b^e_Q(D^{\dagger}), & \mbox{ if }t(e)\in W\mbox{ and }s(e)\not \in W\\
-b^e_Q(D^{\dagger}), & \mbox{ if }t(e)\not \in W\mbox{ and }s(e)\in W\\
0,& \mbox{ otherwise.}
\end{array}\right.\]
\end{Def}

\begin{Exa}\label{Exa:b}
We let $d=2$ and consider the trivial degree-0 polarization. 
Consider the graph $\G$ in Example \ref{Exa:a}. We let $v_0=u$ and $D^{\dagger}=5u-3v$.
The divisors associated to $Q_0,$ is $\div(Q_0)=-2u$.
We have $D_0:=D^{\dagger}+\div(Q_0)=3u-3v$ and $\qs(D_0)=\qs(D^{\dagger}+\div(Q_0))=u-v$.
Keeping the notation introduced before Definition \ref{Def:b_tropical}, we find
$M_{D_0}=u$ and $\phi_{D_0}=-(e_1+e_2)$
 Then, 
\[
 b^{e_1}_{Q_0}(D^{\dagger})=\phi_{D_0}(e_1)=-1 \quad \text{ and } \quad b^{e_2}_{Q_0}(D^{\dagger})=\phi_{D_0}(e_2)=-1.\]
For the subset $W=\{u\}$ of $V(\G)$, we have
\[
b_{Q_0}^{e_1}(W,D^{\dagger})=1\quad\text{ and } \quad b_{Q_0}^{e_2}(W,D^{\dagger})=1.\]
\end{Exa}

Next, given a graph $\Gamma$ and a tropical curve $X$ having $\Gamma$
 as a model, we can present $X$ as
\[
X=\frac{\cup_{e\in E(\G)} [0,\ell(e)]}{\sim}
\]
where $\sim$ means identification of the intervals $[0,\ell(e)]$ at their endpoints, as
prescribed by the combinatorial data of $\G$.  (Recall that we identified $e$ with the interval $[0,\ell(e)]$.)

Consider a function $f\col\{1,\ldots,d\}\ra E(\G)$
and set $f_k=f(k)$. Define the box
\[
\mc{H}_f:=\prod_{1\le k\le d}f_k\cong [0,\ell(f_1)]\times \dots\times[0,\ell(f_d)]\subset \mathbb R^d.
\]
We will label the coordinates of the vertices $Q$ of any $\mc H_f$ using the vertices of $\G$, that is, we will write  $Q=(v_{1},\ldots,v_{d})$, where $v_{k}$ is a vertex of $\G$ incident to the edge $f_k$. We will identify the set of vertices $V(\mc H_f)$ of $\mc H_f$ with $V_f$ (see equation \eqref{eq:Vf}). 

Let $d$ be an integer. We can present the $d$-th product $X^d$ of $X$ as a union of hypercubes:
\[
X^d=\left(\frac{\cup_{e\in E(\G)} [0,\ell(e)]}{\sim}\right)^d=\bigcup_{f} \mc H_f.
\]

Consider a point $p_0\in X$, a degree-$k$ polarization $\mu$ on $X$, and a divisor $D^{\dagger}\in \Div^{d+k}(\G)$. Let $\D^{\dagger}$ be the divisor on $X$ induced by $D^\dagger$. We define the \emph{tropical Abel map associated to $\D^\dagger$} as 
\begin{equation}\label{eq:tropAbelmap}
\begin{array}{rcl}
\alpha^{\trop}_{d,\D^\dagger}\col& X^d \ra&J^{\trop}_{p_0,\mu}(X)\\
(p_1,\ldots,p_d)&\mapsto&\left[\D^{\dagger}-\displaystyle{\sum_{1\le i\le d}p_i}\right],
\end{array}
\end{equation}
where $[-]$ denotes the class of a divisor.
Notice that the degree-$d$ Abel map $\alpha_d^\trop$ defined in equation \eqref{eq:alpha} is the map $\alpha^{\trop}_{d,\D^\dagger}$ for $\D^{\dagger}=dp_0.$

From now on, we will restrict ourselves to the case where $X=X_{\G}$, that is all edges of $\Gamma$ are assigned length $1$. Consider a function $f\col \{1,\dots,d\}\ra E(\G)$.
Let  $S=\{Q_0,\dots,Q_n\}\subset V(\mc H_f)$ be a set of $n+1$ vertices in the hypercube $\mc{H}_f$ and $\Conv(S)\subset \mc H_f$ be the convex hull of $S$, for some positive integer $n$.
Given a convex $n$-tuple $\ult$, we consider the point $R_{\ult}$ in the convex hull $\Conv(S)$ defined as:
\begin{equation}\label{eq:Rt}
R_{\ult}=Q_0+\sum_{1\le j\le n} t_{j}(Q_{n-j+1}-Q_0).
\end{equation}
Notice that we can write $R_{\ult}=(R_1,\ldots,R_d)\in \prod_{1\le j\le d}f_i$,
with $R_i\in f_i$. 
Via the identification of $f_i$ with the interval $[0,1]$, one can easily check that  $R_i=p_{f_i,d_{J_i}}$ (that is the point on $f_i$ at distance $d_{J_i}$ from the source of $f_i$), for some $J_i\subset \{-1,1,\dots,n\}$. 
We define the divisor $\div(R_{\ul t})$ on $X$ as:
\begin{equation}\label{eq:divRt}
\div(R_{\ult})=-\sum_{1\le i \le d} p_{f_i,d_{J_i}}=\underset{J\subset \{-1,1,\ldots,n\}}{\sum_{e\in E(\Gamma)}} \mathfrak{a}(e,J)p_{e,d_{J}}    
\end{equation}
where 
\[
\mathfrak{a}(e,J)=-\#\{i\in \{1,\ldots, d\}; f_i=e\text{ and }R_i=p_{f_i,d_J}\}.
\]
Notice that $\div(R_{\ult})=\D^{\mathfrak{a}}_{\ult}$ (see Construction \ref{Cons:dta}). Moreover, 
\[
\div(R_{(0,\ldots,0)})=\div(Q_0)
\quad \text{ and } \quad
\div(R_{(0,\ldots,1,\ldots,0)})=\div(Q_{n-i+1}),
\]
where in the last formula the number $1$ appears in the $i$-th coordinate of $\ult$.
By definition, we have:
\[
\alpha^{\trop}_{d,\D^{\dagger}}(R_{\ult})=\left[\D^\dagger + \div(R_{\ul t})\right].
\]

 For every edge $e\in E(\G)$, we define 
\[
\div(e,R_{\ul t})=\sum_{J\subset\{-1,1,\ldots,n\}} \mathfrak{a}(e,J)p_{e,d_{J}}=-\sum_{\{i | f_i=e\}}p_{e,d_{J_i}}.\]
We can see $\div(e,R_{\ul t})$ as the ``contribution to $\div(R_{\ul t})$ along $e$".
Notice that $\div(e,R_{\ul t})=0$ if $e$ is not in the image of $f$. Moreover,
\[
\div(R_{\ul t})=\sum_{e\in E(\Gamma)} \div(e,R_{\ul t}).
\]

We denote by $\widetilde{\mathcal{R}}_{\ul t}$ the organized version of $\div(R_{\ul t})$. By equation \eqref{eq:supp}, 
\begin{equation}\label{eq:support}
\supp \widetilde{\mc{R}}_{\ult}\subseteq V(\G)\cup \left(\bigcup_{e\in E(\G)}\{p_{e,r_j} : j=1,\dots, n\}\right),
\end{equation}
where $r_j=\sum_{i=1}^{j}t_i$. We denote by $\widetilde{\mc{R}}^{e}_{\ult}$ the organized version of $\div(e,R_{\ult})$, so that we have: 
\[
\widetilde{\mc{R}}_{\ult}=\sum_{e\in E(\Gamma)} \widetilde{\mc{R}}^{e}_{\ult}.
\]

\begin{Exa}
Consider the tropical curve with  model as in Example \ref{Exa:a}, with edges of length 1. Recall that $d=2$. We identify $u=0$ and $v=1$, so that $Q_0=(0,0)$, $Q_1=(1,0)$, $Q_2=(0,1)$, and $Q_3=(1,1)$. Consider the set $S=\{Q_3,Q_1,Q_2\}$ and a convex 2-tuple 
 $\ult=(t_1,t_2)$. We have:
\[
R_{\ul t}=Q_3+t_1(Q_2-Q_3)+t_2(Q_1-Q_3)=(1-t_1,1-t_2).
\]
In this case, we have $J_1=\{-1,1\}$, $J_2=\{-1,2\}$ and the divisor associated to $R_{\ult}$ is \[
\div(R_{\ult})=-p_{e_1,d_{J_1}}-p_{e_2,d_{J_2}}=-p_{e_1,1-t_1}-p_{e_2,1-t_2}.
\]
Applying Proposition \ref{prop:arrumar}, we get $\widetilde{\mc{R}}_{\ult}=-p_{e_1,0}+p_{e_1,r_1}-p_{e_2,r_1}+p_{e_2,r_2}-2p_{e_2,1}$ (see Figure \ref{Fig:exa_Rt2}).
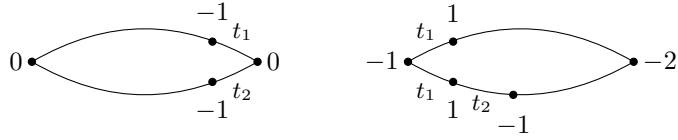
\begin{figure}[h]
\begin{center}
\begin{tikzpicture}[scale=4]
\begin{scope}[shift={(0,0)}]
\draw (0,0) to [out=30, in=150] (0.75,0);
\draw (0,0) to [out=-30, in=-150] (0.75,0);

\draw[fill] (0,0) circle [radius=0.012];
\node[left] at (0,0) {$0$};
\draw[fill] (0.75,0) circle [radius=0.012];
\node[right] at (0.75,0) {$0$};

\draw[fill] (0.6,0.067) circle [radius=0.012];
\node[above] at (0.6,0.1) {$-1$};
\node[above] at (0.7,0.04) {\small{$t_1$}};
\draw[fill] (0.6,-0.067) circle [radius=0.012];
\node[below] at (0.6,-0.1) {$-1$};
\node[below] at (0.7,-0.04) {\small{$t_2$}};
\end{scope}
\begin{scope}[shift={(1.25,0)}]
\draw (0,0) to [out=30, in=150] (0.75,0);
\draw (0,0) to [out=-30, in=-150] (0.75,0);

\draw[fill] (0,0) circle [radius=0.012];
\node[left] at (0,0) {$-1$};
\draw[fill] (0.75,0) circle [radius=0.012];
\node[right] at (0.75,0) {$-2$};
\draw[fill] (0.15,0.067) circle [radius=0.012];
\node[above] at (0.15,0.1) {$1$};
\node[above] at (0.06,0.04) {\small{$t_1$}};
\draw[fill] (0.35,-0.11) circle [radius=0.012];
\node[below] at (0.35,-0.15) {$-1$};
\node[below] at (0.245,-0.08) {\small{$t_2$}};
\draw[fill] (0.15,-0.067) circle [radius=0.012];
\node[below] at (0.15,-0.1) {$1$};
\node[below] at (0.06,-0.04) {\small{$t_1$}};
\end{scope}
\end{tikzpicture}\end{center}
\caption{The divisors  $\div(R_{\ult})$ and its organized version $\wt{\mc R}_{\ul t}$.}\label{Fig:exa_Rt2}
\end{figure}
\end{Exa}

Now, given a point $R_{\ult}\in \Conv(S)$ and an edge $e\in E(\Gamma)$, by equation \eqref{eq:support} we can write
\[
\widetilde{\mc{R}}^{e}_{\ult}=\sum_{0\le j\le n+1} \widetilde{\mc R}_{\ul t}^{e}(p_{e,r_j}) \cdot p_{e,r_j}.
\] 
for some integers $\widetilde{\mc R}_{\ul t}^{e}(p_{e,r_j})$ that does not depend on $\ul t$ (see Figure \ref{fig:Rtil}). 
In the next lemma we compute the coefficients of the divisor $\widetilde{\mc R}_{\ul t}^{e}$.

\begin{figure}[h!]
\begin{center}\begin{tikzpicture}[scale=6]
\begin{scope}[shift={(0,0)}]
\draw (0,0) to (0.3,0);
\draw[dotted] (0.3,0) to (0.45,0);
\draw (0.45,0) to (0.6,0);
\draw[dotted] (0.6,0) to (0.75,0);
\draw (0.75,0) to (1.3,0);
\draw[fill] (0,0) circle [radius=0.012];
\draw[fill] (0.15,0) circle [radius=0.012];
\draw[fill] (0.30,0) circle [radius=0.012];
\draw[fill] (0.45,0) circle [radius=0.012];
\draw[fill] (0.6,0) circle [radius=0.015];
\draw[fill] (0.75,0) circle [radius=0.012];
\draw[fill] (0.9,0) circle [radius=0.012];
\draw[fill] (1.05,0) circle [radius=0.012];
\draw[fill] (1.3,0) circle [radius=0.012];
\node[above] at (0.6,0.05) {$\widetilde{\mc R}_{\ul t}^{e}(p_{e,r_j})$};
\node[left] at (0,0) {$s(e)$};
\node[right] at (1.3,0) {$t(e)$};
\node[below] at (0.075,0) {$t_1$};
\node[below] at (0.225,0) {$t_2$};
\node[below] at (0.525,0) {$t_j$};
\node[below] at (0.825,0) {$t_{n-1}$};
\node[below] at (0.975,0) {$t_{n}$};
\end{scope}    
\end{tikzpicture}
\end{center}
\caption{The divisor $\widetilde{\mc{R}}^{e}_{\ult}$.}
\label{fig:Rtil}
\end{figure}

\begin{Lem}\label{Lem:deltaa}
  Let $S=\{Q_0,\ldots, Q_n\}$ be a subset of $V_f$ and consider the divisor $\widetilde{\mc R}_{\ul t}^{e}$ constructed above. For every edge $e\in E(\Gamma)$, we have
\begin{align*}
\widetilde{\mc R}_{\ul t}^{e}(s(e))&=-a^e_{Q_{n}}(s(e))\\    
\widetilde{\mc R}_{\ul t}^{e}(t(e))&=-a^e_{Q_0}(t(e))\\
\widetilde{\mc{R}}^{e}_{\ult}(p_{e,r_j})&=a^e_{Q_{n-j}}(t(e))-a^e_{ Q_{n-j+1}}(t(e)),\text{ for }j=1,\ldots, n.
\end{align*}
\end{Lem}

\begin{proof}
 We set $a_j:=\widetilde{\mc R}_{\ul t}^{e}(p_{e,r_j})$.
We can view $\div(e,Q_i)$ as a divisor on the tropical curve $X$ (recall equation \eqref{eq:diveQ}). 
In the case $\ul t=(0,\ldots,0)$, we have  $R_{(0,\ldots,0)}=Q_0$, hence $\div(e,R_{(0,\ldots,0)})=\div(e,Q_0)$, that is, a divisor supported on the source $s(e)=p_{e,0}$ of $e$ and on the target $t(e)=p_{e,1}$ of $e$.
By Proposition \ref{prop:arrumar} (3), we have $\widetilde{\mc{R}}^{e}_{(0,\ldots,0)}=\div(e,Q_0)$, then, by equation \eqref{eq:sete}, 
\[
\left(\sum_{0\le k\le n}a_{k}\right) \cdot p_{e,0}+a_{n+1}\cdot  p_{e,1}=-a^e_{Q_0}(s(e))\cdot p_{e,0}-a^e_{Q_0}(t(e))\cdot p_{e,1}.
\]
We deduce that $a_{n+1}=-a^e_{Q_0}(t(e)).$

Similarly, in the case $\ul t=(0,\ldots,0,1)$, we have $R_{(0,\ldots,0,1)}=Q_1$, hence $\div(e,R_{(0,\ldots,0,1)})=\div(e,Q_1)$. By Proposition \ref{prop:arrumar} (3), we have $\widetilde{\mc{R}}^{e}_{(0,\ldots,0,1)}=\div(e,Q_1)$. Then, as before, 
\[
\left(\sum_{0\le k\le  n-1}a_{k}\right)\cdot p_{e,0}+(a_{n}+a_{n+1})\cdot p_{e,1}=-a^e_{Q_1}(s(e))\cdot p_{e,0}-a^e_{Q_1}(t(e))\cdot p_{e,1}.
\]
Taking into account the value of $a_{n+1}$  already computed, we deduce that
$a_{n}=a^e_{Q_0}(t(e))-a^e_{Q_1}(t(e)).$

In general, let us compute $a_{n-1},\dots,a_1$ by iterating this process. Let $j\in\{1,\dots,n-1\}$ and consider the convex $n$-tuple $\ult=(0,\ldots,0,1,0,\ldots,0)$ where $1$ appears in the $j$-th entry. We know that $R_{\ult}=Q_{n-j+1}$, hence  $\div(e,R_{\ult})=\div(e,Q_{n-j+1}).$ By  Proposition \ref{prop:arrumar} (3), we have $\widetilde{\mc{R}}^{e}_{\ult}=\div(e,Q_{n-j+1})$, then
\begin{multline*}
\left(\sum_{0\le k\le j-1}a_{k}\right)\cdot p_{e,0}+\left(\sum_{j\le k\le n+1}a_{k}\right)\cdot p_{e,1}=-a^e_{Q_{n-j+1}}(s(e))\cdot p_{e,0}-a^e_{Q_{n-j+1}}(t(e))\cdot p_{e,1}.
\end{multline*}
Taking into account the value of $a_{k}$ for $k=j+1,\ldots,n+1$ already computed, we deduce that 
\[
a_{j}=a^e_{Q_{n-j}}(t(e))-a^e_{Q_{n-j+1}}(t(e)), \text{ for } j=1,\dots,n-1.
\]
Notice that for $j=1$ we get 
\[a_{0}\cdot p_{e,0}+\left(\sum_{1\le k\le n+1}a_{k}\right)\cdot p_{e,1}=-a^e_{Q_{n}}(s(e))\cdot p_{e,0}-a^e_{Q_{n}}(t(e))\cdot p_{e,1}.
\]
and hence 
$a_{0}=-a^e_{Q_{n}}(s(e))$, concluding the proof.
\end{proof}

Let $\Gamma$ be a graph with a fixed orientation, a fixed vertex $v_0\in V(\Gamma)$, and a degree-$k$ polarization $\mu$. Moreover, we let $X=X_\Gamma$ be the tropical curve associated to $\Gamma$, and we denote by $p_0\in X$ the point corresponding to $v_0$ and $\mu$ the induced polarization on $X$. Consider a divisor $D^\dagger$ of degree $k+d$ on $\Gamma$. We write $\D^\dagger$ to denote the divisor $D^\dagger$ seen as a divisor on $X$.  We consider a subset $S=\{Q_0,\dots,Q_n\}\subset V_f$ of $n+1$ vertices of a box $\mc H_f$.

 Given a convex $n$-tuple $\ult$. We let $R_{\ult}$ be the point in the convex hull $\Conv(S)$ as in equation \eqref{eq:Rt}, and $\div(R_{\ult})$   
 the associated divisor as in equation \eqref{eq:divRt}. Recall that $\Conv(S)^\circ$ denotes the interior of $\Conv(S)$.

\begin{Lem}\label{lem:qs}
Suppose that $\alpha^{\trop}_{d,\D^{\dagger}}(\Conv(S)^\circ)$ is contained in an open cell $\P_{\mc{E},D}^\circ \subset J^{\trop}_{p_0,\mu}(X)$, for some $\mu$-quasistable pseudo-divisor $(\E,D)$ on $\Gamma$.  Then, for each $e\in \E$ there is a subset $I_e\subset\{-1,1,\ldots, n\}$ such that:
\[
\qs(\alpha^{\trop}_{d,\D^{\dagger}}(R_{\ul t}))=\sum_{v\in V(\G)}D(v)\cdot v-\sum_{e\in \E} p_{e,d_{I_e}},
\]
for every convex $n$-tuple $\ul t$.
\end{Lem}
\begin{proof}
We begin proving the linearity of the restriction 
\[
\alpha^{\trop}_{d,\D^{\dagger}}|_{\Conv(S)}\col \Conv(S)\to \P_{\E,D}.
\]
 Indeed, let $\chi \col X^d\to \Omega(X)^\vee $ and $\rho \col \Omega(X)^\vee\to J^{\trop}(X)$ be the maps defined in equations \eqref{eq:chi} and \eqref{eq:alpha}. We have that $\rho^{-1}(\P_{\E,D}^\circ)$  is a disjoint union of (relatively) open polyhedra in $\Omega(X)^\vee$ isomorphic to $\P_{\E,D}^\circ$. Since $\chi(\Conv(S)^\circ)$ is connected, we have that it is contained in one of these open polyhedra $\P^\circ$. In particular, we have that 
 \[
 \alpha^{\trop}_{d,\D^{\dagger}}|_{\Conv(S)}=\rho|_{\P}\circ \chi|_{\Conv(S)}.
 \]
Both $\rho|_{\P}$ and $\chi|_{\Conv(S)}$ are linear, and so is $\alpha^{\trop}_{d,\D^{\dagger}}|_{\Conv(S)}$.

We know that $\alpha^{\trop}_{d,\D^{\dagger}}(R_{\ul t})$ is the class of the divisor $\D^{\dagger}+\div(R_{\ult})$ in $J_{p_0,\mu}^{\trop}(X)$. Recall that $\P_{\E,D}= \prod_{e\in \E} [0,1]$. We let $\pi_e$ be the projection $\pi_e\col \P_{\E,D}\ra e\cong [0,1]$ for every edge $e\in E(\Gamma)$. So we can identify $\pi_e\circ \alpha^{\trop}_{d,\D^{\dagger}}(Q_j)$ with one of the two end-points of $[0,1]$, i.e., either with $0$ or with $1$. 
Consider  
\[
I'_e=\left\{j\in\{1,\dots,n\} \big|\ \pi_{e}\circ\alpha_{d,D^{\dagger}}^{\trop}(Q_{n-j+1})\neq \pi_{e}\circ\alpha_{d,D^{\dagger}}^{\trop}(Q_0)\right\}.
\] 
and define
\[
I_e:=
\begin{cases}
\begin{array}{ll}
I'_e, & \text{ if } \pi_e\circ\alpha_{d,D^\dagger}^\trop(Q_0)=0; \\
\{-1\}\cup I'_e, & \text{ if } \pi_e\circ\alpha_{d,D^\dagger}^\trop(Q_0)=1.
\end{array} 
\end{cases}
\]
By equation \eqref{eq:Rt} we have
\begin{align*}
    \pi_e\circ\alpha^{\trop}_{d,\D^{\dagger}}(R_{\ult}) & = \pi_e\circ\alpha^{\trop}_{d,\D^{\dagger}}\left(Q_0+\sum_{1\le j\le n} t_{j}(Q_{n-j+1}-Q_0)\right)\\
    & =\pi_e\circ\alpha^{\trop}_{d,\D^{\dagger}}(Q_0)+\pi_e\circ\alpha^{\trop}_{d,\D^{\dagger}}\left(\sum_{1\le j\le n} t_{j}(Q_{n-j+1}-Q_0)\right)\\
    & =\pi_e\circ\alpha^{\trop}_{d,\D^{\dagger}}(Q_0)+\sum_{1\le j\le n} t_{j}\cdot \pi_e\circ\alpha^{\trop}_{d,\D^{\dagger}}(Q_{n-j+1}-Q_0)
\end{align*}
which implies
\[
\pi_e\circ\alpha^{\trop}_{d,\D^{\dagger}}(R_{\ult})  =\left\{
\begin{array}{cc}
\sum_{j\in I_e}t_j,  & \mbox{if }\pi_e\circ\alpha^{\trop}_{d,\D^{\dagger}}(Q_0)=0; \\
1-\sum_{j\in I_e}t_j, &  \mbox{if }\pi_e\circ\alpha^{\trop}_{d,\D^{\dagger}}(Q_0)=1.
\end{array}\right.
\]
Recalling the definition of $d_{I}$ in Definition \ref{def:dI}, we deduce that
\[
\qs\left(\alpha^{\trop}_{d,\D^{\dagger}}(R_{\ul t})\right)=\sum_{v\in V(\G)}D(v)\cdot v-\sum_{e\in \E} p_{e,d_{I_e}},
\]
as wanted.
\end{proof}

\begin{Rem}
\label{Ds-admissible} 
Assume that the hypotheses of Lemma \ref{lem:qs} are satisfied. Then Lemma \ref{lem:qs} gives rise to
 a function $\mathfrak{a}_S\col E(\Gamma)\times \P(\{-1,1,\dots,n\})\ra \mathbb Z$ as in Construction \ref{Cons:dta} such that 
\[
\D^{\mathfrak{a}_S}_{\ul t}=\qs(\alpha^{\trop}_{d,\D^{\dagger}}(R_{\ul t})).
\]
Since $\D_{\ul t}^{\mathfrak{a}_S}$ is quasistable, by Remark \ref{rem:admissible} we have that its organized version $\widetilde{\D}^{\mathfrak{a}_S}_{\ul t}$  is admissible.
For simplicity, we denote 
\[
\D^S_{\ul t}:=\D^{\mathfrak{a}_S}_{\ul t}
\quad \text{ and }\quad 
\widetilde{\D}^S_{\ul t}:=\widetilde{\D}^{\mathfrak{a}_S}_{\ul t}.
\]
\end{Rem}

The organized version of $\div(R_{\ult})$ is denoted by $\widetilde{\mc{R}}_{\ult}$.
Recall that we have $\alpha^{\trop}_{d,\D^{\dagger}}(R_{\ul t})=\D^\dagger+\div(R_{\ul t})$ and we denote by $f_{\ul t}$ a rational function such that 
\begin{equation}\label{eq:fourterms}
\D^\dagger+\widetilde{\mc{R}}_{\ult}=\widetilde{\D}^{S}_{\ul t}+\div(f_{\ul t}).
\end{equation}
Notice that $\qs(\alpha^{\trop}_{d,\D^{\dagger}}(R_{\ul t}))$ and $\qs(\D^\dagger+\widetilde{\mc{R}}_{\ul t})$ are the same divisor, since 
there is a unique quasistable representative in an equivalence class.

We now relate the slopes of the function $f_{\ult}$ to the numbers $b^e_Q(D^\dagger)$. Recall that $b^e_Q(D^\dagger)$ is defined by means of a flow on $\Gamma$, namely $b^e_Q(D^\dagger)=\phi_{D^\dagger+\div(Q)}(e)$ (see Definition \ref{Def:b_tropical}).

\begin{Lem}\label{Lem:dob} 
Assume that the hypotheses of Lemma \ref{lem:qs} hold. For every convex $n$-tuple $\ult$ and every edge $e\in E(\Gamma)$, the slope of $f_{\ult}$ over the oriented segment $\overrightarrow{p_{e,r_{j-1}}p_{e,r_{j}}}$ is constant and it is equal to $b^e_{Q_{n-j+1}}(D^{\dagger})$ for every $j=1,\dots,n+1$.
\begin{figure}[htp]
\begin{center}\begin{tikzpicture}[scale=7]
\begin{scope}[shift={(0,0)}]
\draw (0,0) to (0.3,0);
\draw[dotted] (0.3,0) to (0.45,0);
\draw (0.45,0) to (1,0);
\draw[->, line  width=0.3mm] (0.15,0) to (0.16,0);
\draw[->, line  width=0.3mm] (0.9,0) to (0.91,0);
\draw[->, line  width=0.3mm] (0.6,0) to (0.61,0);
\draw[fill] (0,0) circle [radius=0.012];
\draw[fill] (0.30,0) circle [radius=0.012];
\draw[fill] (0.45,0) circle [radius=0.012];
\draw[fill] (0.75,0) circle [radius=0.012];
\draw[fill] (1,0) circle [radius=0.012];

\node[left] at (0,0) {$s(e)$};
\node[right] at (1,0) {$t(e)$};
\node[below] at (0.15,0) {$e_1$};
\node[below] at (0.6,0) {$e_{n}$};
\node[below] at (0.9,0) {$e_{n+1}$};
\node[above] at (0.15,0) {$b^{e}_{Q_n}(D^{\dagger})$};
\node[above] at (0.6,0) {$b^{e}_{Q_1}(D^{\dagger})$};
\node[above] at (0.9,0) {$b^{e}_{Q_0}(D^{\dagger})$};
\end{scope}    
\end{tikzpicture}
\end{center}
\caption{The flow inducing $f_{\ult}$ over the edges $e_j$.}
\end{figure}
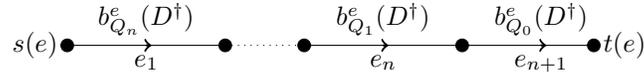
\end{Lem}

\begin{proof}
 Notice that the divisor $\D^{\dagger}+\widetilde{\mc R}_{\ul t}-\widetilde{\D}^{S}_{\ult}$ is principal for every $\ul t$, and it is induced by the divisor $D^\dagger+\widetilde{R}-\widetilde{D}^S$ on $\G^{(n+1)}$ (recall equation \eqref{eq:Dt}). Then,  it satisfies the hypotheses of Proposition \ref{prop:independece}.  So we get a flow $\phi$ (that does not depend on $\ul t$) on $\Gamma^{(n+1)}$ which, in particular, induces  $f_{\ul t}$, for every $\ul t$. This means that  $\phi(e_j)$ is equal to the slope of $f_{\ult}$ over $\overrightarrow{p_{e,r_{j-1}}p_{e,r_{j}}}$ for every $e\in E(\Gamma)$
 (recall that $e_1,\dots,e_{n+1}$ denote the edges over $e$).
 We need to prove that $\phi(e_j)=b^e_{Q_{n-j+1}}(D^\dagger)$, for every edge $e\in E(\G)$ and $j\in\{1,\dots.n+1\}$.
 
 Since $\phi$ is independent of $\ul t$, we can choose $\ul t=(0,\ldots, 1,\ldots,0)$, where $1$ is placed in the $j$-th entry, or $\ul t=(0,\ldots, 0)$. In the former case, we have  $p_{e,r_{j-1}}=s(e)$ (the source of $e$) and $p_{e,r_{j}}=t(e)$ (the target of $e$), hence, by Proposition \ref{prop:arrumar} (3),
    \[
    \div(f_{\ul t})= \D^{\dagger}+\widetilde{\mc R}_{\ult}-
    \widetilde{\D}^{S}_{\ul t}=\D^\dagger+\div(Q_{n-j+1})-\qs(\D^\dagger+\div(Q_{n-j+1})).
    \]
 In the latter case, we have  $p_{e,r_{n}}=s(e)$ and $p_{e,r_{n+1}}=t(e)$ and, using again Proposition \ref{prop:arrumar} (3), 
 \[
 \div(f_{\ul t})= \D^{\dagger}+\widetilde{\mc R}_{\ult}-
 \widetilde{\D}^{S}_{\ul t}=\D^\dagger+\div(Q_{0})-\qs(\D^\dagger+\div(Q_{0})).
 \]
    Since all edges have length $1$ and the support of $\D^\dagger+\div(Q_{n-j+1})$ is contained in $V(\Gamma)$ for every $j=1,\dots,n+1$, we have that $\div(f_{\ul t})\in \Prin(\Gamma)$.
    Using Remark \ref{rem:determine}, we deduce that $\phi(e_j)=\phi_{D^\dagger+\div(Q_{n-j+1})}(e)=b^e_{Q_{n-j+1}}(D^\dagger)$ for every $j=1,\dots,n+1$.
 \end{proof}

Now we can prove the main theorem of the section.

\begin{Thm}\label{Thm:4.2_versão_tropical}
Let $\G$ be a graph, $v_0$ a vertex of $\G$, and $D^{\dagger}$ a divisor of degree $k+d$ on $\G$. Let $X=X_\Gamma$ be the tropical curve associated to $\G$ and 
 $\mu$ a degree-$k$ polarization on $X$. Consider a set of vertices $S=\{Q_0,\ldots,Q_n\}$ in $V_f$, for some function $f\col \{1,\dots,d\}\ra E(\Gamma)$.  
Let $\D^\dagger$ be the divisor on $X$ induced by $D^\dagger$, and suppose that the image of $\Conv(S)^\circ$ via the tropical Abel map $\alpha^{\trop}_{d,\D^{\dagger}}\col X^d\ra J_{p_0,\mu}^{\trop}(X)$ is contained in an open cell  $\P^\circ_{\mc{E},D}\subset J_{p_0,\mu}^{\trop}(X)$, for some pseudo-divisor $(\E,D)$ on $\Gamma$. Then:
\begin{itemize}
\item[(1)] for every $i,j\in\{0,\dots, n\}$ and every edge $e\in E(\G)$ we have
\[
\Big|\left(a^{e}_{Q_{i}}(t(e))+b^{e}_{Q_{i}}(D^{\dagger})\right)- \left(a^{e}_{Q_{j}}(t(e))+b^{e}_{Q_{j}}(D^{\dagger})\right)\Big|\leq1;
\]
\item[(2)] for every function $j\col E(\G)\ra \{0,\ldots,n\},$ and every subset $W\subset V(\G)$ such that $v_0\in W$, we have \end{itemize}
\[
-\frac{\delta_{W}}{2}<\deg(D^{\dagger}|_{W})-\mu(W)-d+\sum_{e\in E(\G)} 
\left(a^{e}_{Q_{j(e)}}(W)-b^{e}_{Q_{j(e)}}(W,D^{\dagger})\right)\leq \frac{\delta_{W}}{2}.
\]

\end{Thm}
\begin{proof} Recall  that $\widetilde{\D}^{S}_{\ult}$ is admissible (see Remark \ref{Ds-admissible}), and 
\begin{equation}\label{Eq:Dstil}
\D^{\dagger}+\widetilde{\mc{R}}_{\ult}-\widetilde{\D}^{S}_{\ult}=\div(f_{\ult}),
\end{equation}
where $\widetilde{\mc{R}}_{\ult}$ is the organized version of $\div(R_{\ult}),$ for $R_{\ult}\in \Conv(S)$ (see equation \eqref{eq:fourterms}).
By Lemmas \ref{Lem:deltaa} and \ref{Lem:dob}, we have
\[
\widetilde{\mc{R}}_{\ult}(p_{e,r_j})=a^{e}_{Q_{n-j}}(t(e))-a^{e}_{Q_{n-j+1}}(t(e)) \text{ for } j=1,\dots,n
\]
and
\[
\div(f_{\ult})(p_{e,r_j})=-b^{e}_{Q_{n-j}}(D^{\dagger})+b^{e}_{Q_{n-j+1}}(D^{\dagger})  \text{ for } j=1,\dots,n
\]
(see Figure \ref{Fig:divft}).
\begin{figure}[htp]
\begin{center}\begin{tikzpicture}[scale=5]
\begin{scope}[shift={(0,0)}]
\draw[dotted] (0,0) to (0.15,0);
\draw (0.15,0) to (2.15,0);
\draw[->, line  width=0.3mm] (0.65,0) to (0.66,0);
\draw[->, line  width=0.3mm] (1.65,0) to (1.66,0);
\draw[dotted] (2.15,0) to (2.3,0);

\draw[fill] (0.15,0) circle [radius=0.017];
\draw[fill] (1.15,0) circle [radius=0.017];
\draw[fill] (2.15,0) circle [radius=0.017];

\node[below] at (1.15,-0.05) {$p_{e,r_j}$};
\node[below] at (0.65,-0.05) {$e_j$};
\node[below] at (1.65,-0.05) {$e_{j+1}$};
\node[above] at (0.65,0) {$b^{e}_{Q_{n-j+1}}(D^{\dagger})$};
\node[above] at (1.65,0) {$b^{e}_{Q_{n-j}}(D^{\dagger})$};
\node[above] at (1.15,0.15) {$a^{e}_{Q_{n-j}}(t(e))-a^{e}_{Q_{n-j+1}}(t(e))$};
\end{scope}    
\end{tikzpicture}
\end{center}
\caption{The divisor $\div(f_{\ult})$ around $p_{e,r_j}$. }\label{Fig:divft}
\end{figure}
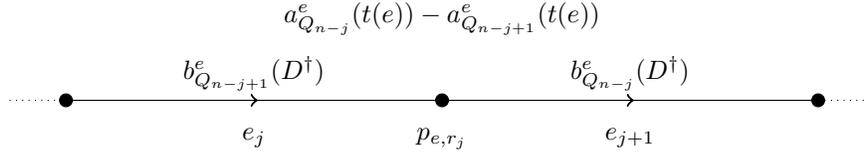

\noindent
Therefore equation (\ref{Eq:Dstil}) gives us, for every $j\in\{1,\dots,n\}$:
\[
a^{e}_{Q_{n-j}}(t(e))-a^{e}_{Q_{n-j+1}}(t(e))-\widetilde{\D}^{S}_{\ult}(p_{e,r_j})  =-b^{e}_{Q_{n-j}}(D^{\dagger})+b^{e}_{Q_{n-j+1}}(D^{\dagger}).
\]
Thanks to the admissibility of $\widetilde{\D}^{S}_{\ult}$, this proves the first statement. In fact, we proved the result only for $j=i+1$; nevertheless, by permuting the vertices $\{Q_0,Q_1,\ldots,Q_n\}$, we see that the argument works for every pair of indices $i,j.$

Next we prove the second statement. The divisor
$\D^{\dagger}+\widetilde{\mc{R}}_{\ult}-\widetilde{\D}^{S}_{\ult}$ is induced by the divisor $D^\dagger+\widetilde{R}-\widetilde{D}^S$ on $\G^{(n+1)}$ (recall equation \eqref{eq:Dt}). 
Let $W\subset V(\Gamma)$ such that $v_0\in W$. Consider the subgraph $\Gamma_W$ of $\Gamma$ formed by the edges of $\Gamma$ connecting two (possibly coinciding) vertices of $W$. Consider the subcurve $X_{\G_W}\subset X_{\G}$
  and a function $j\col E(\G)\ra \{0,\ldots,n\}.$ 
For every edge $e\in E(W,W^{c})$ either $s(e)\in W$ or $t(e)\in W$ ($s(e)$ and $t(e)$ are the source and target of $e$). We define the subcurve
\[
Y:=X_{\G_W}\cup \underset{s(e)\in W}{\bigcup_{e\in E(W,W^{c})}}\ora{p_{e,r_0}p_{e,r_{n-j(e)}}}\cup\underset{t(e)\in W}{\bigcup_{e\in E(W,W^{c})}} \ora{p_{e,r_{n+1}}p_{e,r_{n-j(e)+1}}}.
\]
We compute $\beta_{\widetilde{\D}^{S}_{\ult}}(Y)$:
\begin{align*}
\beta_{\widetilde{\D}^{S}_{\ult}}(Y) & =\deg(\widetilde{\D}^{S}_{\ult}|_{Y})-\mu(Y)+\frac{\delta_{Y}}{2} \\
& =\deg\left(\D^{\dagger}|_{Y}+\widetilde{\mc{R}}_{\ult}|_{Y}-\div(f_{\ult})|_{Y}\right)-\mu(W)+\frac{\delta_{W}}{2}  \\
& =\deg(D^{\dagger}|_{W})+\deg(\widetilde{\mc{R}}_{\ult}|_{Y})-\deg(\div(f_{\ult})|_{Y})-\mu(W)+\frac{\delta_{W}}{2}.   
\end{align*}
Using Lemma \ref{Lem:deltaa}, we have that the degree $\deg(\widetilde{\mc{R}}_{\ult}|_{Y})=d_1-d_2-d_3$ where
\begin{align*}
d_1= &\sum_{e\in E(\Gamma_W)}\deg(\widetilde{\mc{R}}_{\ult}^e)\\
d_2=&\underset{s(e)\in W}{\sum_{e\in E(W,W^{c})}}
\left(a^{e}_{Q_n}(s(e))-\sum_{k=1}^{n-j(e)} \left(a^{e}_{Q_{n-k}}(t(e))-a^{e}_{Q_{n-k+1}}(t(e))\right)\right)\\
 d_3=&\underset{t(e)\in W}{\sum_{e\in E(W,W^{c})}}
 \left(a^{e}_{Q_0}(t(e))-\sum_{k=n-j(e)+1}^{n} \left(a^{e}_{Q_{n-k}}(t(e))-a^{e}_{Q_{n-k+1}}(t(e))\right)\right).
\end{align*}
Hence $\deg(\widetilde{\mc{R}}_{\ult}|_{Y})$ is equal to
\begin{align*}
 &\sum_{e\in E(\Gamma_W)}\deg(\widetilde{\mc{R}}_{\ult}^e)+
\underset{s(e)\in W}{\sum_{e\in E(W,W^{c})}}
\left(a^{e}_{Q_{j(e)}}(t(e))-a^{e}_{Q_{n}}(t(e))-a^{e}_{Q_{n}}(s(e))\right)-
\underset{t(e)\in W}{\sum_{e\in E(W,W^{c})}}
a^{e}_{Q_{j(e)}}(t(e)).
\end{align*}

For each $e\in E(\G)$ we let $m(e)$ be the number of times that $e$ appears in the image of the function $f$. 
By Definition \ref{Def:a_tropical}, given a point $Q\in V_f$, we have 
\begin{equation}\label{Eq:m(e)}
m(e)= a^{e}_{Q}(t(e))+a^{e}_{Q}(s(e)).  
\end{equation}
Notice that $m(e)$ is independent of the choice of $Q$. Moreover, $a^{e}_{Q_{j(e)}}(W)=0$ and $\deg(\widetilde{\mc{R}}_{\ul t}^e)=-m(e)$ for every edge $e\in E(\G_W)$. Then
\begin{equation}\label{eq:GammaW}
\sum_{e\in E(\Gamma_W)}\deg(\widetilde{\mc{R}}_{\ult}^e)=-\sum_{e\in E(\Gamma_W)}m(e)=\sum_{e\in E(\Gamma_W)} \left(a^{e}_{Q_{j(e)}}(W)-m(e)\right).
\end{equation}
Of course, for every $e\in E(W,W^c)$, we have 
\begin{equation}\label{eq:ts}
a^{e}_{Q_{j(e)}}(W)=
\begin{cases}
a^{e}_{Q_{j(e)}}(t(e)), & \text{ if } s(e)\in W\\
a^{e}_{Q_{j(e)}}(s(e)), & \text{ if } t(e)\in W.
\end{cases}
\end{equation}

Combining equations \eqref{eq:GammaW}, \eqref{eq:ts}, \eqref{Eq:m(e)}, we can write:
\begin{align*}
\deg(\widetilde{\mc{R}}_{\ult}|_{Y})&
=-\sum_{e\in E(\G_W)} m(e)+
\underset{s(e)\in W}{\sum_{e\in E(W,W^{c})}}
\left(a^{e}_{Q_{j(e)}}(W)-m(e)\right) +
\underset{t(e)\in W}{\sum_{e\in E(W,W^{c})}}
\left(a^{e}_{Q_{j(e)}}(W)-m(e)\right)\\
&
=-\sum_{e\in E(\Gamma_W)}m(e)+\sum_{e\in E(W,W^{c})} \left(a^{e}_{Q_{j(e)}}(W)-m(e)\right)\\
&=\sum_{e\in E(W,W^{c})\cup E(\Gamma_W)} \left(a^{e}_{Q_{j(e)}}(W)-m(e)\right).
\end{align*}
Since for every $e\in E(\Gamma_{W^{c}})$ we have $a^{e}_{Q_{j(e)}}(W)=m(e)$, we get
\[
\deg(\widetilde{\mc{R}}_{\ult}|_{Y})=\sum_{e\in E(\G)} \left(a^{e}_{Q_{j(e)}}(W)-m(e)\right)=\left(\sum_{e\in E(\G)} a^{e}_{Q_{j(e)}}(W)\right)-d.
\]

Next, recall that the degree of $\div(f_{\ult})$ over a vertex is the sum of the slopes of $f_{\ult}$ over the edges going in the vertex. So the contribution to $\deg(\div(f_{\ult})|_Y)$ along an edge $e\in E(W,W^c)$ is equal to $-b^e_{Q_{j(e)}}(D^{\dagger})$ if $s(e)\in W$ and to $b^e_{Q_{j(e)}}(D^{\dagger})$ if $t(e)\in W$. Then  
\begin{align*}
\deg(\div(f_{\ult})|_{Y})&=
\sum_{e\in E(W,W^{c})}b^{e}_{Q_{j(e)}}(W,D^{\dagger}).
\end{align*}
Putting all together, we obtain:
\begin{align*}
\beta_{\widetilde{\D}^{S}_{\ult}}(Y) &=\deg(D^{\dagger}|_{W})-d-\mu(W)+\sum_{e\in E(\G)} \left(a^{e}_{Q_{j(e)}}(W)-b^{e}_{Q_{j(e)}}(W,D^{\dagger})\right)+\frac{\delta_{W}}{2}.
\end{align*}
Since $V(\G)\not\subset Y$, applying Proposition \ref{Prop:beta_do_quasistavel} to $Y$ and $Y^c$, we deduce that $0<\beta_{\widetilde{\mathcal D}^{S}_{\ult}}(Y)\leq \delta_{Y}$, and we are done.
\end{proof}

\section{Abel maps of smoothings via tropical geometry}\label{Chap:5}\label{Chap:Applications}

In this section we come back to algebraic geometry and prove our main result, Theorem \ref{thm:tropgeoAbel}. This results tells us that if the restriction of the tropical Abel map to a triangulation of the $d$-th product of a tropical curve is a morphism of polyhedral complexes, then we get a resolution of the geometric Abel map for a one-parameter family of curves. We use Theorem \ref{thm:tropgeoAbel}  to show that all the degree-$1$ maps for a one-parameter family of curves already extends to the total space of the family (see Theorem \ref{thm:Abel1}).

\subsection{The geometric and the tropical Abel map}\label{sec:geo-trop}

Our first goal is to relate the tropical and the geometric setups of the previous sections.

Let $C$ be a nodal curve over an algebraically closed field $k$ and $\Gamma$ be its dual graph. For each vertex $v\in V(\G)$ we denote by $C_v$ the corresponding component of $C$, and for each edge $e\in E(\G)$ we denote by $N_e$ the corresponding node.\par
Recall that a \emph{smoothing} of $C$ is a family $\pi\col \C\ra B=\Spec(k[[t]])$ of curves, with smooth generic fiber and with special fiber isomorphic to $C$, such that the total space $\C$ of the family is regular. We will always fix a section $\sigma\col B\ra \C$ of $\pi$ through its smooth locus. For a node $N\in C$, the completion of the local ring of $\C$ at $N$ is
\[\widehat{\O}_{\C,N}\cong\frac{k[[x,y,t]]}{(xy-t)}\]
where $x,y$ are \'etale coordinates of $\C$ at $N$ such that the map $\pi\col\C\ra B$ is locally around $N$ given by the natural inclusion of rings:
\[
k[[t]]\to \frac{k[[x,y,t]]}{(xy-t)}.
\]

Let $\pi\col \C\ra B$ be a smoothing of a nodal curve $C$ with dual graph $\Gamma$. We denote
\begin{equation}\label{eq:Cdlocal}
\C^d=\C\times_B\C\times_B\cdots\times_B\C
\end{equation}
the $d$-th fiber product of $\C$ over $B$. Given a 
 point $\mc{N}=(N_1,\ldots,N_d)\in \C^d$, the completion of the local ring of $\C^d$ in $\mc{N}$ is
\[\widehat{\O}_{\C^d,\mc{N}}\cong\frac{k[[x_1,y_1,\ldots,x_d,y_d,t]]}{(x_1y_1-t,\ldots,x_dy_d-t)},\]
where $x_i,y_i$ are \'etale coordinates of $\C$ at $N_i$.  
The map $\C^d\ra B$ is given by $x_1 y_1=\ldots=x_d y_d=t$ locally around $\mc{N}$.

The \emph{pointed tropical curve associated to $\pi$} is $(X,p_0)$, where $X=X_\Gamma$ and $p_0$ is the point corresponding to the vertex $v_0$  of $\Gamma$ such that $\sigma(0)\in C_{v_0}$.
Since the total space $\C$ of the family  is smooth, we can view $X$ as the tropicalization of the family.

Consider a degree-$k$ polarization $\mu$ on $\G$. Abusing notation we will also call $\mu$ the induced polarization on $C$ and $X$. The Jacobian of the tropical curve $X$ with respect to $(p_0,\mu)$ can be presented as a union of cells:
\[
J_{p_0,\mu}^{\trop}(X)=\bigcup_{(\E,\mc{D})}\mc{P}_{(\E,D)}
\]
where the union is taken over $(\E,\mc{D})\in \mc{QD}_{p_0,\mu}(\G)$ (recall Definition \ref{def:jtrop}). 

As usual, for every function $f\col \{1,\dots,d\}\ra E(\Gamma)$, we denote by $\mc H_f$ the hypercube
\[
\mc{H}_f:=\prod_{i=1}^{d}f_i=[0,1]^d,
\]
where we identify each segment $f_i=f(i)$ with $[0,1]$. 
A vertex of the hypercube $\mc H_f$ is a point  $Q=(v_1,\ldots,v_d)$, where $v_i$ is a vertex of the edge $f_i$. Recall that we call $V_f$ the set of vertices of $\mc H_f$. 
Let $\cone(\mc H_f)$ be the cone generated by the $(d+1)$-tuples $(v_1,\ldots,v_d,1)$, where $(v_1,\ldots,v_d)$ varies through the vertices of $\mc{H}_f.$ Let $\widehat{U}_{\mc H_f}$ the local toric variety associated to $\cone(\mc H_f)$ (recall equation \eqref{eq:localtoricvariety}). By Proposition \ref{Prop:alg_do_cone} and equation \eqref{eq:Cdlocal}, given a point $\mc N=(N_{f_1},\dots,N_{f_d})\in \mc C^d$, we have that
\[
\wh{U}_{\mc{H}_f}=\Spec\left(
\frac{k[[x_1,y_1,\ldots,x_d,y_d,t]]}{(x_1y_1-t,\ldots,x_dy_d-t)}\right)\cong\Spec\left(\widehat{\O}_{\C^d,\mc{N}}\right).
\]

Recall that the $d$-th product of $X$ can be presented as
\[
X^d=\bigcup_{f} \mc H_f,
\]
where $f$ varies through all functions $f\col \{1,\dots,d\}\ra E(\Gamma)$.

Since refinements of cones induce blowups of the associated toric varieties, it follows that any refinement of $\cone(\mc H_f)$ corresponds to a blowup of $\C^{d}$ locally around the point $\N$. This is the key ingredient used in the proof of Theorem \ref{thm:tropgeoAbel}.

\begin{Def}
Let $\mc H$ be a hypercube. A polyhedral complex $\Delta$ \emph{refines} $\mc H$ if $\mc H$ is equal to the union of the polyhedra of $\Delta$ and the interiors of two polyhedra of $\Delta$ have empty intersection inside $\mc H$. If $\Delta$ refines $\mc H$ then $\cone(\Delta)$ refines $\cone(\mc H)$.
A \tit{triangulation $\wt{\mc{H}}$ of $\mc{H}$} is a polyhedral complex $\wt{\mc{H}}$ refining $\mc{H}$ and whose polyhedra are simplexes. Notice that there is a natural inclusion $\wt{\mc H} \hookrightarrow\mc H$. 
The triangulation is \tit{unimodular} if its simplexes are unitary (recall Definition \ref{def:volume}). 
\end{Def}

\begin{Def}\label{def:compatible}
Let $\L$ be an invertible sheaf of degree $k$ on $\C$. Let $D^\dagger_\L$ be the divisor of degree $k+d$ on $\G$ induced by the multidegree of $\L(d\sigma(B))$, and let $\D^\dagger_\L$ the induced divisor on $X$. 
We say that a refinement $\wt{\mc{H}}_f$ of $\mc{H}_f$ is \emph{compatible with the tropical Abel map associated to $\D^\dagger_\L$} if for every polyhedral $\Conv(S)$ in the refinement $\wt{\mc{H}}_n$, there is a cell $\P_{\mc{E},D} \subset J^{\trop}_{p_0,\mu}(X)$ for some $(p_0,\mu)$-quasistable pseudo-divisor $(\E,D)$ on the graph $\G$ underlying $X$ such that the image of $\Conv(S)^\circ$ via $\alpha^{\trop}_{d,\D^{\dagger}_\L}\col X^d\ra J^{\trop}_{p_0,\mu}(X)$ is contained in $\P_{\mc{E},D}^\circ$.
\end{Def}

\begin{Rem}
Consider a refinement $\wt{\mc{H}}_f$ of $\mc{H}_f$ and the restriction $\alpha\col \mc H_f\ra J^{\trop}_{p_0,\mu}(X)$ of $\alpha^{\trop}_{d,\D^{\dagger}_\L}$ to $\mc H_f$. If the following composition is a morphism of polyhedral complexes:
\[
\wt{\mc{H}}_f\hookrightarrow \mc{H}_f \stackrel{\alpha}{\longrightarrow} J^{\trop}_{p_0,\mu}(X),
\] 
then $\wt{\mc{H}}_f$ is compatible with the tropical Abel map associated to $\D^\dagger_\L$.
\end{Rem}

\begin{Thm}\label{thm:tropgeoAbel}
Let $\pi\col \C\ra B$ be a smoothing of a nodal curve $C$ with smooth components. Let $\Gamma$ be the dual graph of $C$. Let $\sigma\col B\ra \C$ be a section of $\pi$ through its smooth locus. Let $\mu$ be a polarization on $\pi$ of degree $k$ and $\L$ be an  invertible sheaf on $\C$ of degree $k$. Consider a point $\N=(N_1,\ldots,N_d)\in \C^d$ where $N_i$ is a node of $C$ and let $f\col \{1,\dots,d\}\ra E(\Gamma)$ be the function such that $N_{f_i}=N_i$. Let $(X,p_0)$ be the pointed tropical curve associated to $\pi$ and $\sigma$. 
Assume that $\beta\col T\ra \wh{U}_{\mc H_f}$ is the local toric blowup associated to a unimodular triangulation of $\mc H_f$ compatible with the tropical Abel map associated to $\D^\dagger_\L$.
Then the composed map
\[
T\stackrel{\beta}{\ra}\wh{U}_{\mc H_f}\cong \Spec(\widehat{\O}_{\C^d,\N})\stackrel{\alpha_\L^d}{\dashrightarrow} \overline{\mathcal{J}}_\mu^\sigma
\]
is a morphism, i.e., it is defined everywhere.
\end{Thm}

\begin{proof}
  Recall that, given a subset of vertices $S=\{Q_0,\ldots,Q_n\}\subset\mc H_f$,
 we let $\Conv(S)\subset \mathcal H_f$ be the convex hull of $S$. 
 Given a simplex $\Conv(S)$  of the triangulation of $\mc H_f$ (which is unitary, by the unimodularity of the triangulation),  we set $U_S:=U_{\Conv(S)}$, which is an open subset of $T$. 
 We know that $T$ is covered by the open subsets $U_S$. By Proposition \ref{Prop:simpplexo_unitario} the map $U_S\to B$ is induced by $t=u_1u_2\ldots u_{d+1}$. Thus we can apply Theorem \ref{Thm:4.2}. 
 
 Let us check Conditions (1) and (2) of Theorem \ref{Thm:4.2}.
  We have a bijection between the vertices $Q=(v_1,\ldots,v_d)\in V_f$ and the divisors of type $C_{v_1}\times \cdots\times C_{v_d}$ of $\C^d$ containing $\N$. We will abuse the notation writing $Q=(v_1,\ldots,v_d)$ for the generic point of the divisor $C_{v_1}\times \cdots\times C_{v_d}$ of $\C^d$.
 Let $\delta_i\col \wh{U}_{\mc{H}_f}\ra \C$ be the composition of the map $\wh{U}_{\mc{H}_f}\ra \C^d$ and the $i$-th projection $\C^d \ra \C$. If $0$ is the closed point of $\wh{U}_{\mc{H}_f}$, then  $\delta_i(0)=N_i=N_{f_i}$ and $\delta_i(Q)\in C_{v_i}.$
 Moreover, we have an identification between the subcurves $Y\subset C$ and the subset of vertices $W\subset V(\G).$  If a subcurve $Y\subset C$ correspond to a subset $W\subset V(\Gamma)$, then the numerical invariants $a^e_Q(Y)$ and $b_Q^e(Y,\L)$ in equations \eqref{a_geometrico} and \eqref{b_geometrico} are equal to the numerical invariants $a_Q^{e}(W)$ and $b_Q^e(W, D^\dagger)$
 in Definitions \ref{Def:a_tropical} and \ref{Def:b_tropical}.
 Therefore, by Theorem \ref{Thm:4.2_versão_tropical}, we have that the Conditions (1) and (2) of 
 Theorem \ref{Thm:4.2} are satisfied. We deduce that the composed map 
 \[
 U_S\stackrel{\beta|_{U_S}}{\ra} \wh{U}_{\mc H_f}\stackrel{\alpha^d_\L}{\dashrightarrow} \mathcal{\ol J}_\mu^\sigma
 \]
 is defined everywhere. Since the open subsets $U_S$ cover $T$, this implies that the composed map $\alpha^d_\L\circ\beta\col T\dashrightarrow \mathcal{\ol J}_\mu^\sigma$ is defined everywhere, as required.
\end{proof}

\begin{Rem}
The criterion of Theorem \ref{thm:tropgeoAbel} can be generalized for a point $\N\in \C^d$ whose entries are not necessarily all nodes of $C$. Indeed,
if $\N=(P_1,\ldots, P_{d-m},N_1,\ldots, N_m)\in \C^d$, where $P_i$ are smooth points and $N_j$ are nodes of $C$, then the completion of the local ring of $\C^d$ at $\N$ is
\[
\wh{\O}_{\C^d,\N}\cong\frac{k[[z_1,\ldots, z_{d-m},x_1,y_1,\ldots, x_{m},y_{m},t]]}{( x_1y_1-t,\ldots, x_my_m-t)}.
\]
Writing $\C^d=\C^{d-m}\times \C^m$ we have that 
\[
\wh{\O}_{\C^d,\N}\cong\wh{\O}_{\C^{d-m},(P_1,\ldots, P_{d-m})}\otimes_{k[[t]]} \wh{\O}_{\C^m,(N_1,\ldots, N_m)}.
\]
Let $\sigma_1,\ldots, \sigma_{d-m}\col B\to \C$ be sections of $\pi$ passing through $P_1,\ldots, P_{d-m}$. Consider 
\[
\L'=\L((d-m)\sigma(B)-\sigma_1(B)-\ldots-\sigma_{d-m}(B)).
\]
We apply Theorem \ref{thm:tropgeoAbel} to the Abel map $\alpha^m_{\mc L'}$ to construct a map:
\[
\gamma\col \Spec(\wh{\O}_{\C^{d-m},(P_1,\ldots, P_{d-m})})\times_B T\ra  T\to \widehat{U}_{H_f}\cong \wh{\O}_{\C^m,(N_1,\ldots, N_m)}\stackrel{\alpha^m_{\L'}}{\dashrightarrow}\overline{\mc J}_\mu^\sigma,
\]
which is a morphism.
We denote by $\mc I'$ the $(\sigma,\mu)$-quasistable torsion-free rank-$1$ sheaf on \[
\Spec(\wh{\O}_{\C^{d-m},(P_1,\ldots, P_m)})\times_B T\times_B \C
\]
inducing the morphism $\gamma$. We call $f$ and $g$ the following maps 
\[
\Spec(\wh{\O}_{\C^{d-m},(P_1,\ldots, P_m)})\times_B T\times_B \C\stackrel{g}{\longrightarrow} \Spec(\wh{\O}_{\C^{d},\N})\times_B \C \stackrel{f}{\longrightarrow} \C,
\]
and let $h=f\circ g$.
We define:
\[
\I=\I'(h^*\sigma_1(B)+\ldots +h^*\sigma_{d-m}(B)-g^*\Delta_{1,d+1}-g^*\Delta_{2,d+1}-\ldots - g^*\Delta_{d-m,d+1}).
\]
The sheaves $\I$ and $\I'$ have the same multidegree on each fiber, hence $\I$ is $(\sigma,\mu)$-quasistable. Moreover, $\I$ and $f^*\L\otimes \O_{\Spec(\wh{\O}_{\C^d,\N})\times_B\C}(d\cdot f^*(\sigma(B))-\Delta_{1,d+1}-\Delta_{2,d+1}-\ldots - \Delta_{d,d+1})$ agree on the generic point of $\Spec(\wh{\O}_{\C^{d-m},(P_1,\ldots, P_{d-m})})\times_B T$ (which is the same as the generic point of $\Spec(\wh{\O}_{\C^{d},\N})$), hence the map
\[
\Spec(\wh{\O}_{\C^{d-m},(P_1,\ldots, P_{d-m})})\times_B T \ra \overline{\mc J}_\mu^\sigma
\]
induced by $\I$ is a resolution of the rational Abel map $\alpha^d_{\mc L}$. 
\end{Rem}
\begin{Rem}\label{rem:image-multidegree}
   With a more careful analysis of the proofs of \cite[Theorem 4.2]{AAJCMP} and  Theorems \ref{Thm:4.2_versão_tropical} and \ref{thm:tropgeoAbel} we can actually conclude that if  $S=\{Q_0,\ldots, Q_d\}$ is such that $\alpha^{\trop}_{d,\D^\dagger_{\mathcal{L}}}(\Conv(S)^\circ)$ is contained in $\P_{\E,D}^\circ$ and $0$ is the distinguished point of $U_S$, then the multidegree of the torsion-free rank-1 sheaf $\alpha_{\mathcal{L}}^d\circ \beta|_{U_s}(0)$ is $(\E,D)$.
\end{Rem}

\begin{Exa}\label{exa:two} 
Let $C$ be a curve with $3$ irreducible components $C_1, C_2,C_3$, and $4$ nodes $N_1=C_1\cap C_3$, $N_2=C_1\cap C_2$ and $\{N_3,N_4\}\subset C_2\cap C_3$. The marked point $p_0$ is contained in $C_3$ and the polarization is $\mu=(0,0,0).$ The dual graph of $C$ is as in Figure \ref{Fig:Exa_casquinhatropical}.
\begin{figure}[h]
\begin{center}
\begin{tikzpicture}
\begin{scope}[shift={(0,0)}]
\draw [black] plot [smooth, tension=0.5] coordinates {(1,0) (3,2) (3,3) (1,5)};
\draw [black] plot [smooth, tension=0.5] coordinates {(4,0) (2,2) (2,3) (4,5)};
\draw [black] plot [smooth, tension=0.5] coordinates {(0,4) (2.5,4.25) (5,4)};
\node[above] at (2.5,4.5) {$C_1$};
\node[left] at (1.5,2.5) {$C_2$};
\node[right] at (3.5,2.5) {$C_3$};
\node[left] at (2,3.8) {$N_1$};
\draw[fill] (1.85,4.2) circle [radius=0.05];
\node[right] at (3,3.8) {$N_2$};
\draw[fill] (3.15,4.2) circle [radius=0.05];
\node[below] at (2.5,3.25) {$N_3$};
\draw[fill] (2.5,3.55) circle [radius=0.05];
\node[below] at (2.5,1) {$N_4$};
\draw[fill] (2.5,1.45) circle [radius=0.05];
\draw[fill] (2.85,1.8) circle [radius=0.05];
\node[right] at (2.8,1.5) {$P_0$};
\end{scope}
\begin{scope}[shift={(6,2.5)}]
\draw (0,0) to [out=-30, in=-150] (4,0);
\draw (0,0) to (2,1);
\draw (2,1) to (4,0);
\draw (0,0) to (4,0);

\draw[fill] (0,0) circle [radius=0.05];
\node[left] at (0,0) {$v_2$};
\draw[fill] (4,0) circle [radius=0.05];
\node[right] at (4,0) {$v_3$};
\draw[fill] (2,1) circle [radius=0.05];
\node[above] at (2,1) {$v_1$};

\node[above] at (2,0) {$e_3$};
\node[below] at (2,-0.55) {$e_4$};
\node[above] at (1,0.5) {$e_2$};
\node[above] at (3,0.5) {$e_1$};

\node[above] at (2,2) {$\G$};
\end{scope}
\end{tikzpicture}   
\end{center}
\caption{The curve $C$ and the dual graph $\G$.}\label{Fig:Exa_casquinhatropical}
\end{figure}
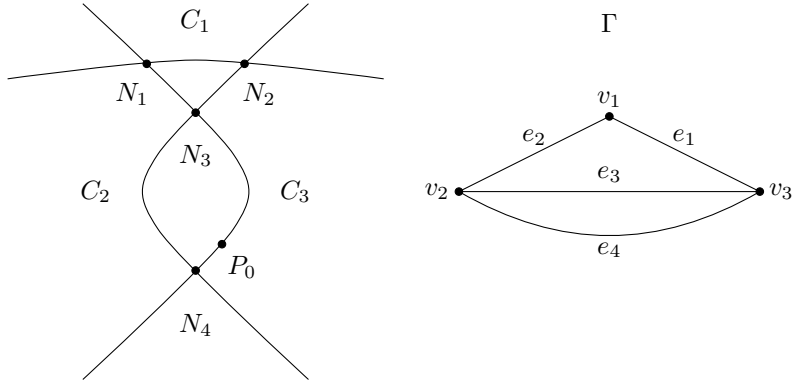

We are interested in the degree-$3$ 
Abel map locally around $\N=(N_4,N_3,N_2)$. So we consider the hypercube $\mc{H}_f=e_4\times e_3\times e_2,$ where the function $f\:\{1,2,3\}\ra E(\G)$ is $f(1)=e_4,\ f(2)=e_3$ and $f(3)=e_2.$ The vertices of $\mc{H}_f$ are $V_f=\{Q_1, Q_2,Q_3,Q_4,Q_5,Q_6,Q_7,Q_8\}$ where 
\begin{align*} Q_1=(v_2,v_2,v_2),\  Q_2=(v_3,v_2,v_2),\  Q_3=(v_3,v_3,v_2),\  Q_4=(v_2,v_3,v_2),\\ Q_5=(v_2,v_2,v_1),\  Q_6=(v_3,v_2,v_1),\  Q_7=(v_3,v_3,v_1),\  Q_8=(v_2,v_3,v_1).\end{align*}
(see Figure \ref{Fig:H3_casquinha}). Moreover, we have $D^{\dagger}=3v_3.$
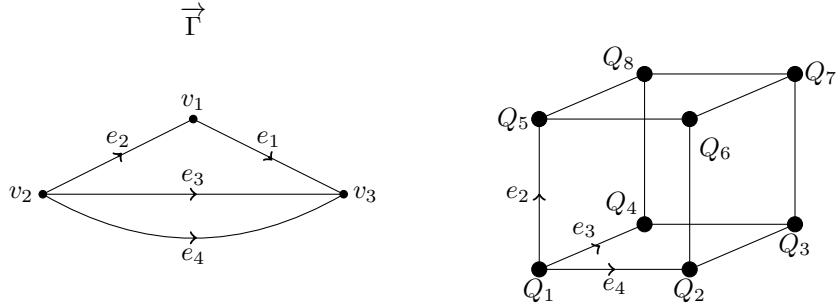
\begin{figure}[h!]
\begin{center}
\begin{tikzpicture}[scale=2]

\begin{scope}[shift={(0,0)},scale=0.5]
\draw (0,0) to [out=-30, in=-150] (4,0);
\draw (0,0) to (2,1);
\draw (2,1) to (4,0);
\draw (0,0) to (4,0);

\draw[->, line  width=0.3mm] (2,0) to (2.05,0);
\draw[->, line  width=0.3mm] (2,-0.58) to (2.05,-0.58);
\draw[->, line  width=0.3mm] (1,0.5) to (1.05,0.55);
\draw[->, line  width=0.3mm] (3,0.5) to (3.05,0.45);

\draw[fill] (0,0) circle [radius=0.05];
\node[left] at (0,0) {$v_2$};
\draw[fill] (4,0) circle [radius=0.05];
\node[right] at (4,0) {$v_3$};
\draw[fill] (2,1) circle [radius=0.05];
\node[above] at (2,1) {$v_1$};

\node[above] at (2,0) {$e_3$};
\node[below] at (2,-0.6) {$e_4$};
\node[above] at (1,0.5) {$e_2$};
\node[above] at (3,0.5) {$e_1$};

\node[above] at (2,2) {$\ora{\G}$};
\end{scope}

\begin{scope}[shift={(3.3,-0.5)}]
\draw[fill] (0,0) circle [radius=.05];
\node[below] at (0,0) {$Q_1$};
\draw[fill] (1,0) circle [radius=.05];
\node[below] at (1,0) {$Q_2$};
\draw[fill] (1.7,0.3) circle [radius=.05];
\node[below] at (1.7,0.3) {$Q_3$};
\draw[fill] (1.7,1.3) circle [radius=.05];
\node[right] at (1.7,1.3) {$Q_7$};
\draw[fill] (1,1) circle [radius=.05];
\node[right] at (1,0.8) {$Q_6$};
\draw[fill] (0,1) circle [radius=.05];
\node[left] at (0,1) {$Q_5$};
\draw[fill] (0.7,1.3) circle [radius=.05];
\node[left] at (0.7,1.4) {$Q_8$};
\draw[fill] (0.7,0.3) circle [radius=.05];
\node[above] at (0.55,0.3) {$Q_4$};

\draw[->, line  width=0.3mm] (0.5,0) to (0.505,0);
\node[below] at (0.5,0) {$e_4$};
\draw[->, line  width=0.3mm] (0,0.5) to (0,0.505);
\node[left] at (0,0.5) {$e_2$};
\draw[->, line  width=0.3mm] (0.4,0.155) to (0.405,0.158);
\node[above] at (0.3,0.15) {$e_3$};

 \draw (0,0) to (1,0);
 \draw (1,0) to (1,1);
 \draw (0,1) to (1,1);
 \draw (0,0) to (0,1);
 
 \draw (0.7,0.3) to (1.7,0.3);
 \draw (1.7,0.3) to (1.7,1.3);
 \draw (0.7,1.3) to (1.7,1.3);
 \draw (0.7,0.3) to (0.7,1.3);
 
 \draw (0,0) to (0.7,0.3);
 \draw (1,0) to (1.7,0.3);
 \draw (0,1) to (0.7,1.3);
 \draw (1,1) to (1.7,1.3);
\end{scope}\end{tikzpicture}\end{center}
\caption{The graph $\protect\ora{\G}$ and the hypercube $\mc{H}_f.$} \label{Fig:H3_casquinha}
\end{figure}

Let us take a point $R_{\ult}=(x,y,z)\in [0,1]^{3}=e_4\times e_3\times e_2$. The  induced divisor is:
\[
    \div(R_{\ult})=-p_{e_4,x}-p_{e_3,y}-p_{e_2,z}.
\]
Applying the tropical Abel map we have that 
\[
\alpha^{\trop}_{3,\D^{\dagger}}(R_{\ult})=[D^{\dagger}+\div(R_{\ult})]
=[3v_3-p_{e_4,x}-p_{e_3,y}-p_{e_2,z}] \]
We have to compute (see Remark \ref{rem:algo}) the quasistable divisor $\qs(\alpha^{\trop}_{3,\D^{\dagger}}(R_{\ult}))$ equivalent to $\alpha^{\trop}_{3,\D^{\dagger}}(R_{\ult})$  (see Figure \ref{Fig:qs(Dt_casquinha)_grau3}).

\begin{figure}[h!]
\begin{center}
\begin{tikzpicture}
\begin{scope}[shift={(0,0)},scale=1.1]
\draw (0,0) to [out=-30, in=-150] (4,0);
\draw (0,0) to (2,1);
\draw (2,1) to (4,0);
\draw (0,0) to (4,0);


\draw[fill] (0,0) circle [radius=0.05];
\node[left] at (0,0) {$0$};
\draw[fill] (4,0) circle [radius=0.05];
\node[right] at (4,0) {$3$};
\draw[fill] (2,1) circle [radius=0.05];
\node[above] at (2,1) {$0$};
\draw[fill] (1,0.6) rectangle (1.02,0.4);
\draw[fill] (1.2,0.1) rectangle (1.22,-0.1);
\draw[fill] (1.1,-0.35) rectangle (1.12,-0.55);

\node[above] at (1,0.6) {$-1$};
\node[above] at (1.2,0.1) {$-1$};
\node[below] at (1,-0.5) {$-1$};
\node[above] at (0.7,0.05) {$z$};
\node[below] at (0.8,0) {$y$};
\node[below] at (0.5,-0.3) {$x$};

\end{scope}
\end{tikzpicture}
\caption{The divisor  $\alpha^{\trop}_{3,\D^{\dagger}}(R_{\ult})$.}
\label{Fig:qs(Dt_casquinha)_grau3}
\end{center}
\end{figure}

We have six possibilities for the combinatorial types of $\qs(\alpha^{\trop}_{3,\D^{\dagger}}(R_{\ult}))$. Below we put the conditions on $(x,y,z)$ under which the combinatorial type is fixed: the corresponding regions in the hypercube are unitary simplexes (which are identified by means of $4$ vertices of the hypercube) inducing a unitary triangulation of the hypercube.
\[
\begin{array}{lcl}
    1.\ x>y>z &\Rightarrow & S_1=\{Q_1,Q_2,Q_3,Q_7\};\\
    2.\ y>z>x  &\Rightarrow & S_2=\{Q_1,Q_4,Q_8,Q_7\};\\
    3.\ z>x>y &\Rightarrow & S_3=\{Q_1,Q_5,Q_6,Q_7\};\\
    4.\ y>x>z &\Rightarrow & S_4=\{Q_1,Q_4,Q_3,Q_7\};\\
    5.\ z>y>x  &\Rightarrow & S_5=\{Q_1,Q_5,Q_8,Q_7\};\\
    6.\ x>z>y   &\Rightarrow & S_6=\{Q_1,Q_2,Q_6,Q_7\}.
\end{array}
\]
In Figure \ref{Fig:triangulation_H3_casquinha} we draw the simplexes $\Delta_i$ associated to the set of vertices $S_i$. The upshot is that, thanks to Theorem \ref{thm:tropgeoAbel}, this triangulation gives rise to a resolution of the Abel map $\alpha^3_\L$ locally around $\mc N$.
\begin{figure}[h]
\begin{center}
\begin{tikzpicture}[scale=1.5]
\begin{scope}[shift={(0,0)}]
\draw[fill] (0,0) circle [radius=.05];
\draw[fill] (1,0) circle [radius=.05];
\draw[fill] (1.7,0.3) circle [radius=.05];
\draw[fill] (1.7,1.3) circle [radius=.05];

\draw[line  width=0.3mm] (0,0) to (1.7,1.3);
\draw[line  width=0.3mm] (1,0) to (1.7,1.3);
\draw[line  width=0.3mm] (0,0) to (1.7,0.3);
\filldraw[fill=black!50!white, draw opacity=0, fill opacity = 0.2] (0,0) -- (1.7,1.3) -- (1.7,0.3)  -- (1,0) -- cycle; 

 \draw[line  width=0.3mm] (0,0) to (1,0);
 \draw (1,0) to (1,1);
 \draw (0,1) to (1,1);
 \draw (0,0) to (0,1);
 
 \draw (0.7,0.3) to (1.7,0.3);
 \draw[line  width=0.3mm] (1.7,0.3) to (1.7,1.3);
 \draw (0.7,1.3) to (1.7,1.3);
 \draw (0.7,0.3) to (0.7,1.3);
 
 \draw (0,0) to (0.7,0.3);
 \draw[line  width=0.3mm] (1,0) to (1.7,0.3);
 \draw (0,1) to (0.7,1.3);
 \draw (1,1) to (1.7,1.3);
 \node[above] at (0,1.4) {$\Delta_1$};
\end{scope}

\begin{scope}[shift={(3,0)}]
\draw[fill] (0,0) circle [radius=.05];
\draw[fill] (0.7,1.3) circle [radius=.05];
\draw[fill] (0.7,0.3) circle [radius=.05];
\draw[fill] (1.7,1.3) circle [radius=.05];

\draw[line  width=0.3mm] (0,0) to (0.7,1.3);
\draw[line  width=0.3mm] (0,0) to (1.7,1.3);
\draw[line  width=0.3mm] (0.7,0.3) to (1.7,1.3);
\filldraw[fill=black!50!white, draw opacity=0, fill opacity = 0.2] (0,0) -- (0.7,1.3) -- (1.7,1.3)  -- (0.7,0.3) -- cycle;  

 \draw (0,0) to (1,0);
 \draw (1,0) to (1,1);
 \draw (0,1) to (1,1);
 \draw (0,0) to (0,1);
 
 \draw (0.7,0.3) to (1.7,0.3);
 \draw (1.7,0.3) to (1.7,1.3);
 \draw[line  width=0.3mm] (0.7,1.3) to (1.7,1.3);
 \draw[line  width=0.3mm] (0.7,0.3) to (0.7,1.3);
 
 \draw[line  width=0.3mm] (0,0) to (0.7,0.3);
 \draw (1,0) to (1.7,0.3);
 \draw (0,1) to (0.7,1.3);
 \draw (1,1) to (1.7,1.3);
 \node[above] at (0,1.4) {$\Delta_2$};
\end{scope}
\begin{scope}[shift={(6,0)}]
\draw[fill] (0,0) circle [radius=.05];
\draw[fill] (0,1) circle [radius=.05];
\draw[fill] (1,1) circle [radius=.05];
\draw[fill] (1.7,1.3) circle [radius=.05];

\draw[line  width=0.3mm] (0,0) to (1.7,1.3);
\draw[line  width=0.3mm] (0,1) to (1.7,1.3);
\draw[line  width=0.3mm] (0,0) to (1,1);
\filldraw[fill=black!50!white, draw opacity=0, fill opacity = 0.2] (0,0) -- (0,1) -- (1.7,1.3)   -- cycle;

 \draw (0,0) to (1,0);
 \draw (1,0) to (1,1);
 \draw[line  width=0.3mm] (0,1) to (1,1);
 \draw[line  width=0.3mm] (0,0) to (0,1);
 
 \draw (0.7,0.3) to (1.7,0.3);
 \draw (1.7,0.3) to (1.7,1.3);
 \draw (0.7,1.3) to (1.7,1.3);
 \draw (0.7,0.3) to (0.7,1.3);
 
 \draw (0,0) to (0.7,0.3);
 \draw (1,0) to (1.7,0.3);
 \draw (0,1) to (0.7,1.3);
 \draw[line  width=0.3mm] (1,1) to (1.7,1.3);
 \node[above] at (0,1.4) {$\Delta_3$};
\end{scope}

\begin{scope}[shift={(0,-2.5)}]
\draw[fill] (0,0) circle [radius=.05];
\draw[fill] (0.7,0.3) circle [radius=.05];
\draw[fill] (1.7,0.3) circle [radius=.05];
\draw[fill] (1.7,1.3) circle [radius=.05];

 \draw[line  width=0.3mm] (0,0) to (1.7,1.3);
\draw[line  width=0.3mm] (0.7,0.3) to (1.7,1.3);
\draw[line  width=0.3mm] (0,0) to (1.7,0.3);
\filldraw[fill=black!50!white, draw opacity=0, fill opacity = 0.2] (0,0) -- (1.7,1.3) -- (1.7,0.3)  -- cycle;

 \draw (0,0) to (1,0);
 \draw (1,0) to (1,1);
 \draw (0,1) to (1,1);
 \draw (0,0) to (0,1);
 
 \draw[line  width=0.3mm] (0.7,0.3) to (1.7,0.3);
 \draw[line  width=0.3mm] (1.7,0.3) to (1.7,1.3);
 \draw (0.7,1.3) to (1.7,1.3);
 \draw (0.7,0.3) to (0.7,1.3);
 
 \draw[line  width=0.3mm] (0,0) to (0.7,0.3);
 \draw (1,0) to (1.7,0.3);
 \draw (0,1) to (0.7,1.3);
 \draw (1,1) to (1.7,1.3);
 \node[above] at (0,1.4) {$\Delta_4$};
\end{scope}

\begin{scope}[shift={(3,-2.5)}]
\draw[fill] (0,0) circle [radius=.05];
\draw[fill] (0,1) circle [radius=.05];
\draw[fill] (0.7,1.3) circle [radius=.05];
\draw[fill] (1.7,1.3) circle [radius=.05];

\draw[line  width=0.3mm] (0,0) to (1.7,1.3);
\draw[line  width=0.3mm] (0,0) to (0.7,1.3);
\draw[line  width=0.3mm] (0,1) to (1.7,1.3);
\filldraw[fill=black!50!white, draw opacity=0, fill opacity = 0.2] (0,0) -- (0,1) -- (0.7,1.3)  -- (1.7,1.3) -- cycle;  

 \draw (0,0) to (1,0);
 \draw (1,0) to (1,1);
 \draw (0,1) to (1,1);
 \draw[line  width=0.3mm] (0,0) to (0,1);
 
 \draw (0.7,0.3) to (1.7,0.3);
 \draw (1.7,0.3) to (1.7,1.3);
 \draw[line  width=0.3mm] (0.7,1.3) to (1.7,1.3);
 \draw (0.7,0.3) to (0.7,1.3);
 
 \draw (0,0) to (0.7,0.3);
 \draw (1,0) to (1.7,0.3);
 \draw[line  width=0.3mm] (0,1) to (0.7,1.3);
 \draw (1,1) to (1.7,1.3);
 \node[above] at (0,1.4) {$\Delta_5$};
\end{scope}

\begin{scope}[shift={(6,-2.5)}]
)\draw[fill] (0,0) circle [radius=.05];
\draw[fill] (1,0) circle [radius=.05];
\draw[fill] (1,1) circle [radius=.05];
\draw[fill] (1.7,1.3) circle [radius=.05];

\draw[line  width=0.3mm] (0,0) to (1.7,1.3);
\draw[line  width=0.3mm] (1,0) to (1.7,1.3);
\draw[line  width=0.3mm] (0,0) to (1,1);
\filldraw[fill=black!50!white, draw opacity=0, fill opacity = 0.2] (0,0) -- (1,1) -- (1.7,1.3)  -- (1,0) -- cycle;  

 \draw[line  width=0.3mm] (0,0) to (1,0);
 \draw (1,0) to (1,1);
 \draw[line  width=0.3mm] (0,1) to (1,1);
 \draw (0,0) to (0,1);
 
 \draw (0.7,0.3) to (1.7,0.3);
 \draw (1.7,0.3) to (1.7,1.3);
 \draw (0.7,1.3) to (1.7,1.3);
 \draw (0.7,0.3) to (0.7,1.3);
 
 \draw (0,0) to (0.7,0.3);
 \draw (1,0) to (1.7,0.3);
 \draw (0,1) to (0.7,1.3);
 \draw[line  width=0.3mm] (1,1) to (1.7,1.3);
 \node[above] at (0,1.4) {$\Delta_6$};
\end{scope}
\end{tikzpicture}
\end{center}
   \caption{The unitary 3-simplexes $\Delta_i$ in $\mc{H}_3$.}
   \label{Fig:triangulation_H3_casquinha}
\end{figure}
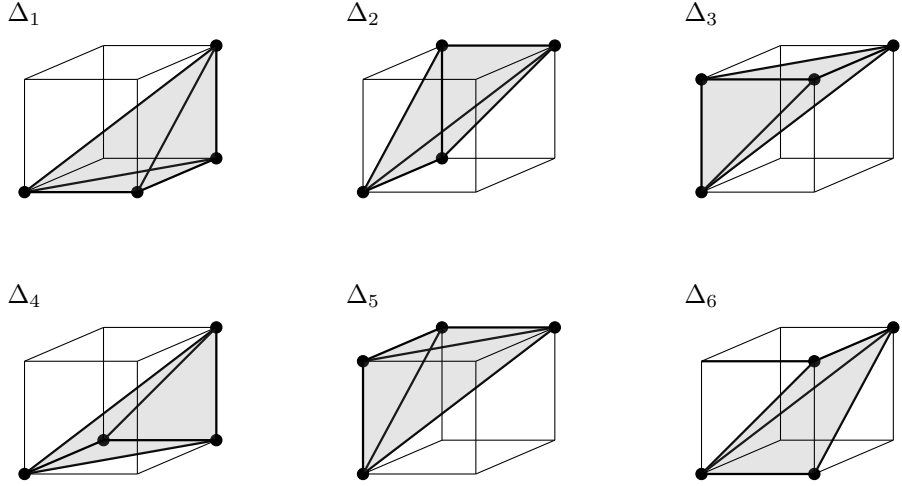


\end{Exa}

\subsection{Degree-1 Abel maps}\label{sec:degree1}

In this section we analyze the degree-$1$ Abel map.  Using Theorem \ref{thm:tropgeoAbel} we will prove the following result: 

\begin{Thm}[Degree-$1$ Abel map]\label{thm:Abel1} 
Let $\pi\col\C\ra B$ be a smoothing of a nodal curve.
Let $\sigma\col B\ra \C$ be a section of $\pi$ through its smooth locus. Let $\mu$ be a polarization of degree $k$ over the family and $\L$ be an invertible sheaf on $\C$ of degree $k$.
The degree-$1$ Abel map $\alpha^1_{\L}\col\C\ra \overline{\J}_{\mu}^\sigma$ is defined everywhere.
\end{Thm}

Before proving Theorem \ref{thm:Abel1} we need a lemma about divisors on graphs.   

Let $\G$ be a graph with a vertex $v_0$ and a polarization $\mu^\sharp$ of degree $k+1$ on $\Gamma$. Consider a $(v_0,\mu^\sharp)$-quasistable divisor $D$ of degree $k+1$ on $\G$ and an edge $e_0\in E(\G)$. Define the divisor $D^\flat$ of degree $k$ on $\Gamma^{\{e_0\}}$ as:
\begin{equation}
    \label{eq:Dflat}
D^\flat(v)=
\begin{cases} 
\begin{array}{ll}
D(v), & \text{ if } v\in V(\G)\\
-1 ,&  \text{ if } v=v_{e_0},
\end{array}
\end{cases}
\end{equation}
where $v_{e_0}$ is the exceptional vertex over $e_0$.
Define the polarization 
\[
\mu(v)=
\begin{cases}
\begin{array}{ll}
\mu^\sharp(v) & \text{ if } v\ne s(e_0)\\
\mu^\sharp(v)-1 & \text{ if } v=s(e_0).
\end{array}
\end{cases}
\]

From now on, to avoid ambiguities in the notation, we will include the  polarization in the notation of the invariant $\beta$, so for example we will write $\beta_{D^\flat,\mu}$ instead of $\beta_{D^\flat}$. 

\begin{Def}\label{def:V}
We define $\ol{V}\subset V(\G^{\{e_0\}})\setminus\{s(e_0)\}$ as the maximal subset $\ol V$ with $\beta_{D^\flat,\mu}(\ol{V})$ minimal. We set $V:=\ol V\cap V(\Gamma)$.
\end{Def}

We also define 
\begin{equation}\label{eq:Etilde}
\wt\E:=E(V,V^c)\setminus \{e_0\}
\end{equation}
and we consider the divisor $\wt D$ on $\Gamma^{\wt \E}$ of degree $k$ taking $v\in V(\Gamma^{\wt \E})$ to 
\begin{equation}\label{eq:Dtilde}
\wt D(v)=\left\{ 
\begin{array}{ll}
D(v)-1, & \mbox{if }v=s(e_0)\\
D(v)+\val_{E(V,V^c)}(v) ,& \mbox{if } v\in V\\
-1, & \mbox{if $v$ is exceptional}\\
D(v), & \mbox{otherwise}.
\end{array}\right.
\end{equation}

\begin{Lem}\label{lem:abel1}
If $(\{e_0\},D^\flat)$ is not $(v_0,\mu)$-quasistable, then $(\wt \E,\wt D)$ is $(v_0,\mu)$-quasistable.
\end{Lem}

\begin{proof}
We begin by noting that $\{t(e_0),v_{e_0}\}\subset \ol{V}$. If $\ol{W}\subset V(\Gamma^{\{e_0\}})\setminus\{s(e_0)\}$, with $\{t(e_0),v_{e_0}\}\subset \ol{W}$ and $W=\ol{W}\cap V(\G)$, we have
\begin{align*}
\beta_{D^\flat,\mu}(\ol{W})=&\deg(D^\flat|_{\ol{W}})-\mu(\ol{W})+\frac{\delta_{\ol{W}}}{2}\\
=&\deg(D|_{W})-1-\mu^\sharp(W)+\frac{\delta_{W}}{2}\\
=&\beta_{D,\mu^\sharp}(W)-1.
\end{align*}
In the other case, it is not hard to see that $\beta_{D^\flat,\mu}(\ol{W})=\beta_{D,\mu^\sharp}(W)$ or $\beta_{D^\flat,\mu}(\ol{W})=\beta_{D,\mu^\sharp}(W)+1$.\par
  Since $D$ is $(v_0,\mu^\sharp)$-quasistable, we have that $\beta_{D,\mu^\sharp}(W)\geq0$ for every $W\subset V(\G)$, with strict inequality if $v_0\in W$.   Since $(\{e_0\},D^\flat)$ is not $(v_0,\mu)$-quasistable, we have that the minimal value for $\beta_{D^\flat,\mu}$ is either negative, or $0$ with $v_0\in V$. In either case, we must have that $\{t(e_0),v_{e_0}\}\subset \ol{V}$. Moreover, we have that 
  \begin{equation}\label{obeta}
0\leq\beta_{D,\mu^\sharp}(V)\leq1.
\end{equation}
with strict inequality on the left if $v_0\in V$, and strict inequality on the right if $v_0\notin V$. Also notice that $e_0\in E(V,V^c)$.

Consider the graph $\Gamma_{\wt \E}$ obtained by removing the edges in $\wt \E$. We define the polarization $\mu_{\wt \E}$ of degree $k+1+|\wt \E|$ on $\Gamma_{\wt \E}$ such that
\[
\mu_{\wt\E}(v)=\mu(v)+\frac{1}{2}\val_{\wt \E}(v).
\]
We now show that $(\wt \E, \wt D)$ is $(v_0,\mu)$-quasistable, which is equivalent to prove that $\wt D_{\wt \E}$ is $\mu_{\wt \E}$-quasistable on $\Gamma_{\wt \E}$ by \cite[Proposition 4.6 (ii)]{AAMPJac}.

Consider a subset $W\subset V(\Gamma_{\wt \E})=V(\Gamma)$. Set $W_1=V\cap W$ and $W_2:=V^c\cap W$, so that $W=W_1\cup W_2$. Notice that $s(e_0)\not\in W_1$ (see Figure \ref{Fig:map_grau1}).

\begin{figure}[htp]\begin{center}
\begin{tikzpicture}[scale=0.4]
\draw[line  width=0.3mm] \boundellipse{4,1}{-2}{4};
\draw[line  width=0.3mm] \boundellipse{-4,1}{-2}{4};
\draw[red, line width=0.2mm] \boundellipse{-4,2}{-1}{2};
\draw[red, line width=0.2mm] \boundellipse{4,-0.1}{-1}{2};

\draw[fill] (-0.5,3.3) rectangle (-0.45,2.9);
\draw [dashed, line width=0.2mm] plot [smooth, tension=0.5] coordinates {(-4,-1) (0,-1.1) (4,-1)};
\draw [dashed, line width=0.2mm] plot [smooth, tension=0.5] coordinates {(-4,1) (0,0.9) (4,1)};
\draw [black, line width=0.2mm] plot [smooth, tension=0.5] coordinates {(-4,3) (0,3.1) (4,3)};

\node[above] at (-4,5) {$V$};
\node[above] at (4,5) {$V^c$};
\node[below] at (-4,3) {$t(e_0)$};
\node[below] at (4,3) {$s(e_0)$};
\node[left] at (-4.1,0) {$W_1$};
\node[right] at (4.5,1) {$W_2$};

\node[above] at (1.5,3) {$e_0$};

\draw[fill] (-4,-1) circle [radius=0.15];
\draw[fill] (-4,1) circle [radius=0.15];
\draw[fill] (-4,3) circle [radius=0.15];
\draw[fill] (4,-1) circle [radius=0.15];
\draw[fill] (4,1) circle [radius=0.15];
\draw[fill] (4,3) circle [radius=0.15];
\end{tikzpicture}   
\end{center}
\caption{The subsets $W_1$ and $W_2$.}
\label{Fig:map_grau1}
\end{figure}
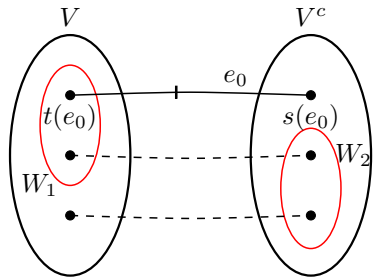

Moreover, by Lemma \ref{lem:beta},
\begin{equation}\label{eq:W1W2}
\beta_{\wt D_{\wt \E},\mu_{\wt \E}}(W)=\left\{ 
\begin{array}{ll}
\beta_{\wt D_{\wt \E},\mu_{\wt \E}}(W_1)+\beta_{\wt D_{\wt \E},\mu_{\wt \E}}(W_2), & \mbox{if }e_0\notin E(W_1,W_2)\\
\beta_{\wt D_{\wt \E},\mu_{\wt \E}}(W_1)+\beta_{\wt D_{\wt \E},\mu_{\wt \E}}(W_2)-1 ,& \mbox{if } e_0\in E(W_1,W_2).
\end{array}\right.
\end{equation}

We compute $\beta_{\wt D_{\wt \E},\mu_{\wt \E}}(W_1)$:
\begin{eqnarray}\label{eq:betaW1}
\beta_{\wt D_{\wt \E},\mu_{\wt \E}}(W_1)&=& \deg(\wt D|_{W_1})-\mu_{\wt\E}(W_1)+\frac{\delta_{\G_{\mc{E}},W_1}}{2} \nonumber\\
&=& (\deg(D|_{W_1})+\val_{\wt \E}(W_1))-\big(\mu^\sharp(W_1)+\frac{1}{2}\val_{\wt \E}(W_1)\big)+\frac{\delta_{\G,W_1}-\val_{\wt \E}(W_1)}{2}\nonumber \\ 
&=&\beta_{D,\mu^\sharp}(W_1)
\end{eqnarray}
By the $(v_0,\mu^\sharp)$-quasistability of $D$ we have 
 $\beta_{\wt D_{\wt \E},\mu_{\wt \E}}(W_1)\geq0$, where the inequality is strict if $v_0\in W_1$.

For the subset $W_2$, we have:
\begin{eqnarray}\label{eq:betaW2}
\beta_{\wt D_{\wt \E}}(W_2)&=& \deg(\wt D|_{W_2})-\mu_{\wt \E}(W_2)+\frac{\delta_{\G_{\mc{E}},W_2}}{2} \nonumber\\
&=& \left(\deg(D|_{W_2})-|\{s(e_0)\}\cap W_2|\right)-\big(\mu(W_2)+\frac{1}{2}\val_{\wt \E}(W_2)\big) +\frac{\delta_{\G,W_2}-\val_{\wt \E}(W_2)}{2}\nonumber \\
&=& \deg(D|_{W_2})-|\{s(e_0)\}\cap W_2|-\mu^\sharp(W_2)+|\{s(e_0)\}\cap W_2|\nonumber\\&&-\frac{1}{2}\val_{\wt \E}(W_2)+\frac{\delta_{\G,W_2}-\val_{\wt \E}(W_2)}{2}\nonumber \\
&=&\beta_{D,\mu^\sharp}(W_2)-\val_{\wt \E}(W_2)
\end{eqnarray}
 On the other hand,
\begin{eqnarray}\label{eq0}
\beta_{D,\mu^\sharp}(V\cup W_2)=\beta_{D,\mu^\sharp}(V)+\beta_{D,\mu^\sharp}(W_2)-\val_{\wt\E}(W_2)-|\{s(e_0)\}\cap W_2|.
\end{eqnarray}

If $s(e_0)\notin W_2$, then, by the properties defining the subset $V$, we have $\beta_{D,\mu^\sharp}(V\cup W_2)>\beta_{D,\mu^\sharp}(V)$. Hence, if $s(e_0)\not\in W_2$, from equation \eqref{eq0} we get $\beta_{D,\mu^\sharp}(W_2)-\val_{\wt \E}(W_2)>0$.
At any rate we have $\beta_{D,\mu^\sharp}(V\cup W_2)\geqslant0$ since $D$ is $(v_0,\mu^\sharp)$-quasistable. Hence in the case $s(e_0)\in W_2$,  again from equation \eqref{eq0}, we get
\[
\beta_{D,\mu^\sharp}(W_2)-\val_{\wt \E}(W_2)\ge 1-\beta_{D,\mu^\sharp}(V).
\]
However, we have $\beta_{D,\mu^\sharp}(V)\le 1$ by equation \eqref{obeta}, and hence  $\beta_{D,\mu^\sharp}(W_2)-\val_{\wt \E}(W_2)\ge0$. Moreover, if $v_0\in W_2$, then $v_0\notin V$ and hence $\beta_{D,\mu^\sharp}(V)<1$ again by equation \eqref{obeta}, from which we obtain $\beta_{D,\mu^\sharp}(W_2)-\val_{\wt  \E}(W_2)>0$.
By equation \eqref{eq:betaW2} we conclude that, in any case, $\beta_{\wt D_{\wt \E}}(W_2)\ge 0$, and the inequality is strict if $v_0\in W_2$.

Therefore, if $e_0\not\in E(W_1,W_2)$, by equation \eqref{eq:W1W2} we get $\beta_{\wt D_{\wt \E},\mu_{\wt \E}}(W)\ge0$ and the inequality is strict if $v_0\in W$. If $e_0\in E(W_1,W_2)$, then:
\begin{align*}
\beta_{\wt D_{\wt \E},\mu_{\wt \E}}(W)&=    \beta_{\wt D_{\wt \E},\mu_{\wt \E}}(W_1)+\beta_{\wt D_{\wt \E},\mu_{\wt \E}}(W_2)-1, \;\text{ by } \eqref{eq:W1W2}
    \\&=\beta_{D,\mu^\sharp}(W_1)+\beta_{D,\mu^\sharp}(W_2)-\val_{\wt\E}(W_2)-1, \;\text{ by } \eqref{eq:betaW1} \text{ and } \eqref{eq:betaW2}\\
                                &= \beta_{D,\mu^\sharp}(W_1)+(1-\beta_{D,\mu^\sharp}(V)+\beta_{D,\mu^\sharp}(V\cup W_2))-1, \;\text{ by }\eqref{eq0}\\
                           &=\beta_{D,\mu^\sharp}(W_1)-\beta_{D,\mu^\sharp}(V)+\beta_{D,\mu^\sharp}(V\cup W_2)\\
                           &\geq\beta_{D,\mu^\sharp}(V\cup W_2) , \;\text{ by the property defining $V$}
                           \\
                           &\geq0, \;\text{ by the $(v_0,\mu^\sharp)$-quasistability of $D$}.
\end{align*}
Moreover, if $v_0\in W=W_1\cup W_2$, then $v_0\in V\cup W_2$, hence $\beta_{D,\mu^\sharp}(V\cup W_2)>0$. This concludes the proof.
\end{proof}

We are now ready to prove the main theorem of this section.

\begin{proof}[Proof of Theorem \ref{thm:Abel1}]
First we assume that the components of $C$ are smooth. 
Let us check that the hypotheses of Theorem \ref{thm:tropgeoAbel} are satisfied. Let $(X,p_0)$ be the pointed tropical curve associated to $\pi$. We denote by $\mu$ the induced polarization of degree $k$ on $X$ and $\G$ the dual graph of $C$ (so that $X=X_\Gamma$).  As in Definition \ref{def:compatible} we let $D^\dagger_{\L}$ and $\D^\dagger_\L$ be the divisors on $\Gamma$ and $X$, respectively, associated to $\L$.\par

Since we deal with the Abel map of degree $1$, the hypercubes $\mc{H}_{f}$ are just segments of length $1$. As such, they are also  unitary simplexes, and the unique triangulation is the trivial one. It is enough to prove that this trivial triangulation is compatible (in the sense of Definition \ref{def:compatible}) with the tropical degree-$1$ Abel map 
$\alpha_{1,\D^\dagger_\L}^\trop\col X\ra  J^{\trop}_{p_0,\mu}(X)$. 

Let $e_0\in E(\G)$  be an edge, which, as we have just said, corresponds to a hypercube $\mc{H}_f$. Consider the degree $k+1$ polarization $\mu^\sharp$ defined by 
\[\mu^\sharp(v)=\left\{ 
\begin{array}{cl}
\mu(v), & \mbox{ if } v\neq s(e_0)\\
\mu(v)+1, & \mbox{ if }v=s(e_0).
\end{array}\right.\]
Let $D$ be the $(v_0,\mu^\sharp)$-quasistable divisor of degree $k+1$ equivalent to $D^\dagger_\L$. Denote by $\D$ the divisor $D$, seen as a divisor on $X$. Consider the divisor of degree $k$ on $X$ defined as $\D_t^\flat=\D-p_{e_{0},1-t}$, where $t$ is a real number $t\in [0,1]$. Since $\D$ and $\D^\dagger_\L$ are equivalent, we have: 
\[
\alpha^{\trop}_{1,\D^\dagger_\L}(p_{e_0,1-t})=[\D_t^\flat].
\]
Notice that $\D_t^\flat$ has combinatorial type $(\{e_0\},D^\flat)$ for every $t\in(0,1)$, 
where $D^\flat$ is the divisor defined in the equation \eqref{eq:Dflat}.
Therefore, if $(e_0,D^\flat)$ is $(v_0,\mu)$-quasistable, then $\D_t^\flat$ is $(p_0,\mu)$-quasistable for every $t\in[0,1]$ (see Proposition \ref{prop:quasiquasi}) and hence the classes $[\D_t^\flat]$ are all contained in the cell $\P_{\{e_0\},D^\flat}$, and we are done. So, we can assume that $(\{e_0\},D^\flat)$ is not $(v_0,\mu)$-quasistable. 

We will prove that $[\D_t^\flat]$ is contained in the cell $\P_{\wt \E,\wt D}$ for every $t\in [0,1]$, where $\wt \E$ and $\wt D$ are defined in equations \eqref{eq:Etilde} and \eqref{eq:Dtilde}. To prove this, we will write down explicitly the $(p_0,\mu)$-quasistable divisor $\qs(\D_t^\flat)$ equivalent to $\D_t^\flat$ and show that its combinatorial type is constant and equal to $(\wt\E,\wt D)$ for $t\in (0,1)$. The idea is that $\D_t^\flat$ is almost $(p_0,\mu)$-quasistable, so a little modification of $\D_t^\flat$ allows us to find $\qs(\D_t^\flat)$. 
   Let $\ol{V}$ and $V$ be as in Definition \ref{def:V}.
 
 Denote by $\mc{P}$ the principal divisor on $X$:
\begin{align*}
\mc{P}=&
s(e_0)
+\underset{s(e)\in V}{\sum_{e\in \wt{\E}}}
p_{e,t}
+\underset{t(e)\in V}{\sum_{e\in \wt{\E}}}
p_{e,1-t}
-p_{e_0,1-t}
-\underset{s(e)\in V}{\sum_{e\in \wt{\E}}}p_{e,0}
-\underset{t(e)\in V}{\sum_{e\in \wt{\E}}}p_{e,1}.
\end{align*}
 
Then, the divisor $\wt\D_t:=\D_t^\flat-\mc{P}$ has combinatorial type  equal to $(\wt\E,\wt D)$ for every $t\in(0,1)$. By Lemma \ref{lem:abel1}, we have that $(\wt{\E},\wt{D})$ is $(v_0,\mu)$-quasistable, which means that $\wt{\D}_t$ is $(p_0,\mu)$-quasistable. Hence, we have that
\[
\alpha^{\trop}_{\D^\dagger_\L,1}(p_{e_0,1-t})=[\D_t^\flat]=[\wt \D_t]\in \P_{\wt \E,\wt D},
\]
as claimed.

For the general case in which $C$  could have non-smooth components we argue as follows. First, let $C_1, \ldots, C_h$ be the components of $C$, and  assume without loss of generality that $\sigma(0)\in C_1$. We consider a curve $C_{\sm}=C_{\sm,1}\cup C_{\sm,2}\cup\ldots C_{\sm,h}$  obtained from $C$ by smoothing just the nodes of components of $C$. Let $\pi_{\sm}\col\C_{\sm}\to B$ be a smoothing of $C_{\sm}$ with a section $\sigma_{\sm}\col B\to \C_{\sm}$ of $\pi_{\sm}$ such that $\sigma_{\sm}(0)\in C_{\sm,1}$.  Let $\mu_{\sm}$ be the polarization on $\C_{\sm}$ induced by $\mu$. We let $\widetilde{\C}^2_{\sm}\to \C^2_{\sm}$ be a desingularization of $\C^2$ obtained by blowing up (strict transforms of) Weil divisors of the form $C_{\sm,i}\times C_{\sm,j}$ in some order (see \cite[Remark 3.7]{AAJCMP}). We define $g\col \widetilde{\C}^2\to \C^2$ to be the series of blow-ups of $\C^2$ along the strict transforms of Weil divisors of the form $C_i\times C_j$ in the same order as before. We note that $\C^2$  could smooth. However, each strict transform to $\widetilde{\C}^2$ of a divisor of type $C_i\times C_j$ is Cartier. 
 Let $\L_{\sm}$ be an invertible sheaf over $\C_{\sm}$ with the same multidegree of $\L$. Consider the  sheaves on $\widetilde{\C}^2_{\sm}$ and $\widetilde{\C}^2$:
 \[
 \widetilde{\M}_{\sm}:=f_{\sm}^*(\L_{sm})\otimes \O_{\widetilde{\C}^2_{\sm}}(f_{\sm}^*(\sigma_{\sm}(B)))\otimes \I_{\wt{\Delta}_{\sm}|\wt{\C}^2_{\sm}}\otimes \O_{\widetilde{\C}^2_{\sm}}\Big(-\sum_{1\le i\le h} \overline{Z}_{\sm,i}\Big)
 \]
and
\[
 \widetilde{\M}:=f^*(\L)\otimes \O_{\widetilde{\C}^2}(f^*(\sigma(B)))\otimes \I_{\wt{\Delta}|\wt{\C}^2}\otimes \O_{\widetilde{\C}^2}\Big(-\sum_{1\le i\le h} \overline{Z}_i\Big)
 \]
 where:
 \begin{itemize}
     \item the maps $f_{\sm}\col \wt{\C}^2_{\sm}\ra \C^2_{\sm}\ra \C_{\sm}$ and $f\col \wt{\C}^2\ra \C^2\ra \C$  are the compositions of the blowups and the projections onto the second factor;
     \item the Weil divisors $\wt{\Delta}_{\sm}$ and $\wt{\Delta}$ are the strict transforms to $\wt{\C}^2_{\sm}$ and $\wt{\C}^2$ of the diagonal subschemes of $\C^2_{\sm}$ and $\C^2$;  
     \item the Cartier divisors $\ol{Z}_{\sm,i}$ and $\ol{Z}_i$ are the closure in $\wt{\C}^2_{\sm}$ and $\wt{\C}^2$ of the divisors defined in equation \eqref{Eq:twZ} (which are obtained one from the other just replacing $C_{\sm,i}$ and $\dot{C}_{\sm,i}$ with $\dot{C}_i$ and $C_i$).
 \end{itemize}
It is clear that $\widetilde{\M}$ is a torsion-free rank-1 sheaf, as the only non invertible sheaf in the tensor products defining $\widetilde{\M}$ is $\I_{\Delta|\C^2}$, which is a torsion-free rank-1 sheaf. Moreover, we have that $\widetilde{\M}_{\sm}$ and $\widetilde{\M}$ have the same multidegree. By the first part of the proof and \cite[Equation 16 and Theorem 4.2]{AAJCMP}, we have that $\widetilde{\M}_{\sm}$ is $(\sigma_{\sm}, \mu_{\sm})$-quasistable. Since quasistability is a numerical condition, we have that $\widetilde{\M}$ is $(\sigma, \mu)$-quasistable on $\wt\C^2/\C$. By \cite[Theorem 2.1]{AAJCMP}, we have that $g_*(\widetilde{\M})$ is $(\sigma,\mu)$-quasistable on $\C^2/\C$ and hence induces a map $\C\to \overline{\J}^{\sigma}_{\mu}$ extending $\alpha^1_\L$.
\end{proof}

\begin{Def}
We say that a connected graph is \emph{biconnected} if it remains connected after the removal of any of its vertices.
\end{Def}

In the next statement, for a curve $C$ with a polarization $\mu$ and a smooth point $P$, we denote $\ol{\mc J}^P_{C,\mu}:=\ol{\mc J}^{\sigma}_\mu$, where $\sigma\col \Spec(k)\ra C$ has image $P$.

\begin{Thm}\label{thm:injective}
Let $\pi\col \C\ra B$ be a smoothing of a nodal curve $C$ whose dual graph is biconnected. Let $\sigma\col B \ra \C$ be a section of $\pi$ through its smooth locus. Let $\mu$ be a polarization of degree $k$ over the family and $\L$ be an invertible sheaf of degree $k$. Then the degree-$1$ Abel map $\alpha_\L^1\col \C\ra \ol{\J}_\mu^\sigma$ is injective. In particular, given a smooth point $P$ of $C$, we have that any point of $\ol{\J}_{C,\mu}^P$ parametrizing an invertible sheaf is contained in a subscheme of $\ol{\J}_{C,\mu}^P$ homeomorphic to $C$. 
\end{Thm}

\begin{proof}
Let $\Gamma$ be the dual graph of $C$. Let $(X,p_0)$ be the pointed tropical curve associated to $\pi$. We consider the tropical Abel map $\alpha^{\trop}_{1,\D^\dagger_\L}\col X\to J^{\trop}_{(p_0,\mu)}$. By \cite{BF} we have that $\alpha^{\trop}_{1,\D^{\dagger}_\L}$ is injective. By Theorem \ref{thm:Abel1} above, we have that for each edge $e\in E(\Gamma)$ and every $t\in (0,1)$, there exists a $(v_0,\mu)$-quasistable pseudo-divisor $(\E_e, D_e)$ on $\Gamma$ such that $\D^{\dagger}_{\mathcal{L}}-p_{e,t}$ is equivalent to a $(p_0,\mu)$-quasistable divisor with combinatorial type $(\E_e,D_e)$. Similarly, for each vertex $v\in V(\Gamma)$ we have a $(v_0,\mu)$-quasistable divisor $D_v$ on $\Gamma$ equivalent to $D^{\dagger}_{\mathcal{L}}-v$. 

We claim that the function
\begin{align*}
  f\col V(\Gamma)\cup E(\Gamma)&\to \mathcal{QD}_{p_0,\mu}(\Gamma)\\
                v&\mapsto (\emptyset, D_v)\\
                e&\mapsto (\E_e, D_e)
\end{align*}
is injective (recall that $\mathcal{QD}_{p_0,\mu}(\Gamma)$ denotes the poset of $(v_0,\mu)$-quasistable pseudo-divisors on $\Gamma$). 
 Indeed,
  by Lemma \ref{lem:abel1} and by the fact that, being $\Gamma$ biconnected, each cut in $\Gamma$ has at least two edges, it follows that $\E_e\neq\emptyset$ for every $e$. Then we have that $f(v)\neq f(e)$ for every $v\in V(\Gamma)$ and $e\in E(\Gamma)$.  If $f(v_1)=f(v_2)$ for distinct vertices $v_1,v_2\in V(\Gamma)$, then $v_1-v_2$ will be equivalent to $0$ in $\Gamma$, which is a contradiction since $\Gamma$ is biconnected. If $f(e_1)=f(e_2)$ for distinct edges $e_1,e_2\in E(\Gamma)$, then the images of $e_1$ and $e_2$ via $\alpha^{\trop}_{1,\D^\dagger_\L}$ are contained in the same hypercube $\P_{\E,D}$. By the first part of the proof of Lemma \ref{lem:qs}, the restrictions $\alpha^{\trop}_{1,\D^\dagger_\L}\big|_{e_j}\col e_j\to \P_{\E,D}$ are linear. Then the images of $e_1$ and $e_2$ in the hypercube $\P_{\E,D}$ are long diagonals, so they must intersect. This contradicts the injectivity of $\alpha^{\trop}_{1,\D^\dagger_\L}$.\par

    By Remark \ref{rem:image-multidegree}, we have that if $q$ is a smooth point of $C_v$ for $v\in V(\Gamma)$, then the multidegree of $\alpha^1_{\mathcal{L}}(q)$ is $f(v)$, while if $N_e$ is a node of $C$ for $e\in E(\Gamma)$, then the multidegree of $\alpha^{1}_{\mathcal{L}}(N_e)$ is $f(e)$. In particular,  if $q_1\neq q_2$ and $\alpha^{1}_{\mathcal{L}}(q_1)=\alpha^{1}_{\mathcal{L}}(q_2)$ for some $q_1,q_2\in C$, then $q_1,q_2$ are smooth points of $C$ in the same component $C_v$. Then we would have  $\mathcal{O}_C(q_1-q_2)\cong \mathcal{O}_C$ which is a contradiction with the fact that the dual graph of $C$ is biconnected.\par

     The last statement of the theorem follows from the fact that each invertible sheaf $I\in \overline{\J}_{C,\mu}^P$ is the image $\alpha^{1}_{\mathcal{I}}(P)$, where $\mathcal{I}$ is any extension of $I$ to $\C$.
\end{proof}

\begin{Rem} Notice that Theorem \ref{thm:injective} could fail if the dual graph of $C$ is not biconnected. For instance,
   if $C$ has a \emph{separating line}, that is a rational smooth irreducible component $E$ such that each node of the set $E\cap E^c$ is disconnecting, then it is clear that two smooth points in the separating line will have the same image via the degree-1 Abel map.\par
   We believe that Theorem \ref{thm:injective} still holds for curves without separating lines, but we were unable to find a suitable proof. If $\mathcal{L}$ is trivial and $\mu$ is the canonical polarization, then 
   the injectivity part of Theorem \ref{thm:injective} is true for every curve without separating lines:  this is \cite[Theorem 6.3]{CCE}.
\end{Rem}

\section{Acknowledgements}
\noindent
This is the first part of the Ph.D. thesis of the second author, under the supervision of the first and third authors. We want to thank Juliana Coelho, Ethan Cotterill, Eduardo Esteves, Rodrigo Gondim, Nivaldo Medeiros, Margarida Melo,  for the carefully reading of a preliminary version of the work. We feel that their observations and questions really improve our work. 

\bibliographystyle{abbrv}
\bibliography{bibliography}

\bigskip
\bigskip

\noindent{Alex Abreu, Sally Andria, and Marco Pacini
\\
Universidade Federal Fluminense
\\ 
Rua Prof. M. de Freitas, Instituto de Matem\'atica, Rio de Janeiro, Brazil}\\
{E-mail: \small\verb?alexbra1@gmail.com? \;\;
\small\verb?sally.andrya@gmail.com ?
\;\;   \small\verb?pacini.uff@gmail.com?}

\end{document}